\documentclass[11pt,a4paper,final]{article}
\author{Hannes Matt \and Daniel Steenebr\"ugge \and Heiko von der Mosel}
\title{Banach gradient flows for various families of knot energies}

\usepackage[utf8]{inputenc}
\usepackage[english]{babel}
\usepackage[intlimits]{amsmath}
\usepackage{amsfonts}
\usepackage{amssymb}
\usepackage{lmodern}
\usepackage[thref, amsthm, hyperref, thmmarks, amsmath]{ntheorem}	
\usepackage[T1]{fontenc}
\usepackage{etoolbox}
\usepackage{calc}
\usepackage{mathtools}
\usepackage{geometry}
\usepackage[multiuser]{fixme}

\usepackage{mathrsfs}				
\usepackage{enumerate}
\usepackage{todonotes}
\usepackage{bm}						

\usepackage[pdftex=true,breaklinks=true,bookmarks=true,linktocpage=true,pdfpagelabels=true,plainpages=false,pagebackref=true, draft=false]{hyperref}

\newcommand{\R}{\mathbb{R}}
\newcommand{\N}{\mathbb{N}}
\newcommand{\Z}{\mathbb{Z}}

\newcommand{\EL}{\mathcal{E}}
\newcommand{\FL}{\mathcal{F}}

\newcommand{\LL}{\mathcal{L}}

\newcommand{\CL}{\mathcal{C}}
\newcommand{\AL}{\mathcal{A}}
\newcommand{\WL}{\mathcal{W}}
\renewcommand{\epsilon}{\varepsilon}
\renewcommand{\rho}{\varrho}
\renewcommand{\d}{\ensuremath{\,\mathrm{d}}}

\newcommand{\dif}{\mathrm{d}}

\newcommand{\FJ}{\mathfrak{J}}
\newcommand{\CB}{\mathcal{B}}
\newcommand{\CA}{\mathcal{A}}

\newcommand{\eps}{\varepsilon}
           \newcommand{\sS}{\mathscr{S}} \newcommand{\sD}{\mathscr{D}}             \newcommand{\sM}{\mathscr{M}}

\DeclareMathOperator{\intM}{intM}
\DeclareMathOperator{\TP}{TP}
\DeclareMathOperator{\BiLip}{BiLip}

\DeclareMathOperator*{\essinf}{essinf}

\newcommand{\Eap}[1][\alpha,p]{\ensuremath{E^{#1}}}
\newcommand{\ea}[1][\alpha]{\ensuremath{e_{#1}}}
\newcommand{\eap}[1][p]{\ensuremath{\ea^{#1}}}
\newcommand{\ir}{\mathrm{ir}}
\newcommand{\ia}{\mathrm{ia}}

\DeclareMathOperator{\AC}{AC}

\newcommand{\loc}{\mathrm{loc}}

\DeclareMathOperator{\dist}{dist}

\newcommand{\intrinsicDistance}[1][\g]{d_{#1}}
\newcommand{\SobSpaceParams}{1 + \frac {\a p -1} {2p}, 2p}
\newcommand{\SobSpace}{W^{\SobSpaceParams}}
\newcommand{\SobSpaceir}{\SobSpace_{\ir}}
\newcommand{\SobSpaceDParams}{\frac {\a p -1} {2p}, 2p}
\newcommand{\SobSpaceD}{W^{\SobSpaceDParams}}

\newcommand{\g}{\gamma}
\renewcommand{\a}{\alpha}

\newcommand{\Lpg}[2][\gamma]{\norm[L^p_{#1}]{#2}}

\newcommand{\arclengthParam}{\Gamma}

\DeclarePairedDelimiter{\abs}{\lvert}{\rvert}
\makeatletter
\newcommand{\norm}{%
    \@ifstar
    \normStar%
    \normNoStar%
}
\newcommand{\set}{%
    \@ifstar
    \setStar%
    \setNoStar%
}
\makeatother

\newcommand{\normNoStar}[2][]{\lVert #2 \rVert_{#1}}
\newcommand{\normStar}[2][]{\left\lVert #2 \right\rVert_{#1}}
\newcommand{\Norm}[1]{\left\lVert#1\right\rVert}
\newcommand{\seminorm}[2][]{\left[#2\right]_{#1}}

\newcommand{\setStar}[2][]{
    \left\{#2%
    \ifthenelse{\equal{#1}{}}{}{\ \middle\vert \ #1}%
    \right\}%
}
\newcommand{\setNoStar}[2][]{
    \{#2%
    \ifthenelse{\equal{#1}{}}{}{\mathrel{\vert}#1}
    \}
}
\DeclarePairedDelimiter{\paren}{(}{)}
\newcommand{\newterm}[1]{\emph{#1}}
\newcommand{\yrightarrow}[1]{\xrightarrow[#1]{}}
\newcommand{\numberthis}{\refstepcounter{equation}\tag{\theequation}}

\newlength{\adjustedalignmentskip}
\DeclareRobustCommand{\adjustedalignment}[3][1]{
    \setlength{\adjustedalignmentskip}{\widthof{${}\displaystyle#2{}$}- \widthof{${}\displaystyle#3{}$}}%
    \setlength{\adjustedalignmentskip}{#1\adjustedalignmentskip}%
    \hspace{\adjustedalignmentskip}#3%
}

\newcommand{\stackrelx}[2]{\adjustedalignment[0.5]{#2}{\stackrel{#1}{#2}}} 

\renewcommand{\[}{\begin{equation*}}
\renewcommand{\]}{\end{equation*}}

\newcommand*{\fres}[2]{ {\left.\kern-\nulldelimiterspace #1 \vphantom{\big|} \right|_{\kern-1pt #2} }}

\theoremstyle{plain}
\newtheorem{lemma}{Lemma}[section]
\newtheorem{theorem}[lemma]{Theorem}
\newtheorem{proposition}[lemma]{Proposition}
\newtheorem{corollary}[lemma]{Corollary}

\renewtheorem*{theorem*}{Theorem}
\renewtheorem*{lemma*}{Lemma}
\renewtheorem*{definition*}{Definition}

\newtheorem*{assM}{Assumption (M) on the metric space}

\theoremstyle{definition}
\newtheorem{remark}[lemma]{Remark}

\newtheorem{claim}{Claim}

\FXRegisterAuthor{d}{dlang}{Daniel}
\FXRegisterAuthor{h}{hlang}{Hannes}
\FXRegisterAuthor{vdM}{vdMlang}{vdM}
\fxusetheme{color}

\fxuselayouts{inline, nomargin}

\numberwithin{equation}{section}

\newcommand{\Fo}{\,\,\,\text{for }\,\,}
\newcommand{\Foa}{\,\,\,\text{for all }\,\,}

\begin{document}
\maketitle

\begin{abstract}
We establish long-time existence of Banach gradient flows 
for generalised integral Menger curvatures and tangent-point energies, and 
for
O'Hara's self-repulsive potentials $E^{\alpha,p}$.
In order to do so, we employ the theory of curves of maximal slope in slightly smaller 
spaces compactly embedding into the respective energy spaces 
associated to these functionals, and add a term involving the logarithmic strain,
which controls the parametrisations of the flowing (knotted) loops.
As a prerequisite, we prove in addition that O'Hara's knot energies $E^{\alpha,p}$
are continuously differentiable.
\end{abstract}

\section{Introduction}\label{sec:1}
It is an interesting and analytically challenging problem in geometric
knot theory to evolve knots according to the gradient flow of a 
self-repelling interaction energy. Such energies are called
\emph{knot energies}, and one may categorise them into two types.

Firstly, there are singular \emph{self-repulsive potentials} such as
J. O'Hara's \cite{OHara:1992:Familyofenergyfunctionalsofknots}
two-parameter energy family $\Eap$ defined on closed regular curves
$\g:\R/\Z\to\R^n$ by
\begin{equation}\label{eq:ohara}
\textstyle
 \Eap(\g) := \iint_{(\R/\Z)^2} \paren[\big]{\frac1{|\g(x)-\g(y)|^\alpha}-
 \frac1{d_\g^\alpha(x,y)}}^p  \abs{\g'(x)} \abs{\g'(y)} \d x \d y 
 \end{equation}
for $\alpha>0$ and $p>1$ satisfying $2\le \alpha p < 2p+1$. Here, 
$d_\g(x,y)$ denotes the intrinsic distance between the points
$\g(x)$ and $\g(y)$ along the curve.
Zh.-X. He \cite{he_2000} had shown short-time existence for the $L^2$-gradient
flow of the \emph{M\"obius}\footnote{This name reflects the invariance
of $E^{2,1}$ under M\"obius transformations; see the seminal work on 
the M\"obius energy by
 Freedman, He, and Wang in \cite{freedman-etal_1994}.} energy $E^{2,1}$
for smooth initial data, before S. Blatt investigated this flow 
systematically for the subfamily $E^{\alpha,1}$. He established long-time existence results and convergence to a critical point:
for $\alpha\in (2,3)$
in \cite{blatt_2018}, and for $\alpha=2$ and initial data sufficiently
close to a local minimiser in \cite{blatt_2012b}. Moreover, again
for $E^{2,1}$ he proved an $\eps$-regularity result together with a blow-up analysis
in \cite{blatt_2020a},  resulting in the 
convergence to the round circle 
if one restricts the $L^2$-flow to planar loops. According
to \cite[p. 31]{BlattReiter:2013:StationarypointsofOHarasknotenergies}
such strong results on the $L^2$-flow for
the general energy family $E^{\alpha,p}$ may be out of reach
because of a degenerate elliptic operator in the first variation
formula for $E^{\alpha,p}$.
As shown in \cite{blatt_2012a} the 
underlying energy space of the M\"obius energy $E^{2,1}$ is
a fractional Sobolev space, the     
Sobolev-Slobodecki\v{\i}\footnote{For the definition and some facts 
concerning these spaces, see Appendix \ref{sec:appendix}; for a condensed 
selection of pertinent results on periodic
Sobolev-Slobodecki\v{\i} spaces see \cite[Appendix A]{KnappmannSchumacherSteenebruggeEtAl:2021:AspeedpreservingHilbertgradientflowforgeneralizedintegralMengercurvature}.}
space $W^{\frac32,2}(\R/\Z,\R^n)$, which is a Hilbert space. 
This fact was recently used by Ph. Reiter and H. Schumacher  
to establish short-time existence for a Hilbert gradient flow in
that space 
for $E^{2,1}$ with a method, however, that seems to be restricted
to $E^{2,1}$; see
\cite[Remark after Theorem 1.2]{ReiterSchumacher:2021:SobolevGradientsfortheMobiusEnergy}.

The second type of knot energies are \emph{geometric curvature energies}
whose integrands are constructed from circles (as first suggested by
O. Gonzalez and J.H. Maddocks in \cite{gonzalez-maddocks_1999}), such
as \emph{integral Menger curvatures}, analysed in \cite{strzelecki-etal_2010,strzelecki-etal_2013a},
\begin{equation}\label{eq:menger}
\textstyle
\mathscr{M}_p(\g):=\iiint_{(\R/\Z)^3}\frac1{R^p(\g(x),\g(y),\g(z))}
\abs{\g'(x)} \abs{\g'(y)} \abs{\g'(z)}\d x \d y \d z\Fo p>3,
\end{equation}
where $R(a,b,c)$ denotes the circumcircle of the three points $a,b,c
\in\R^n$, or the \emph{tangent-point energies} \cite{strzelecki-vdm_2012}
\begin{equation}\label{eq:tan-point}
\textstyle
\TP_q(\g):=\iint_{(\R/\Z)^2}\frac1{r^q_\textnormal{tp}[\g](\g(x),\g(y))}
\abs{\g'(x)} \abs{\g'(y)} \d x \d y \Fo q>2,
\end{equation}
where $r_\textnormal{tp}[\g](\g(x),\g(y))$ is the radius of the unique
circle through $\g(x)$ and $\g(y)$ that is tangent to the curve
$\g$ at the point $\g(x)$.

Replacing the $p$-th power of the circumcircle radius  
in \eqref{eq:menger}
by the more general (but less geometric) expression
$
R^{(p,q)}(a,b,c):= \frac{(|b-c||b-a||c-a|)^p}{|(b-a)\wedge (c-a)|^q}$,
or similarly, replacing the $q$-th power of the tangent-point radius in \eqref{eq:tan-point}
by
\[
\textstyle
r^{(p,q)}[\g](\g(x),\g(y)):= \frac{|\g(x)-\g(y)|^p}{\dist^q\big(\g(x)+
\R\g'(x),\g(y)\big)}
\]
one obtains \emph{generalised integral Menger curvatures}
\begin{equation}\label{eq:gen-menger}
\textstyle
\intM^{(p,q)}(\g):=\iiint\limits_{(\R/\Z)^3}\frac{\abs{\g'(x)} \abs{\g'(y)} 
\abs{\g'(z)}}{R^{(p,q)}(\g(x),\g(y),\g(z))}
\d x \d y \d z\Fo q\in (1,\infty),\,p\in (\frac23 q +1,q+\frac23  ),
\end{equation}
and \emph{generalised tangent-point energies}
\begin{equation}\label{eq:gen-tan-point}
\textstyle
\TP^{(p,q)}:=\iint_{(\R/\Z)^2}\frac{\abs{\g'(x)} \abs{\g'(y)}}{r^{(p,q)}[\g](\g(x),\g(y))}\d x \d y \Fo q\in (1,\infty),\,p\in (q+2,2q+1).
\end{equation}
These two-parameter energy families contain the original geometric curvature
energies \eqref{eq:menger} and \eqref{eq:tan-point} up to a
constant factor, i.e., $\sM_p =2^p\intM^{(p,p)}$ and 
$\TP_q=2^q\TP^{(2q,q)} $, and they were introduced by Blatt and Reiter 
to prove smoothness of critical points specifically
of $\intM^{(p,2)}$ for $p\in (\frac73,\frac83)$ 
\cite[Theorem 4]{BlattReiter:2015:TowardsaregularitytheoryforintegralMengercurvature},
and of $\TP^{(p,2)}$ for $p\in (4,5)$
\cite[Theorem 1.5]{BlattReiter:2015:Regularitytheoryfortangentpointenergiesthenondegeneratesubcriticalcase}. 
 There is no existence result on the $L^2$-flow for these energies yet, but the underlying
Hilbert space structure in this specific range of
parameters was recently used in \cite{KnappmannSchumacherSteenebruggeEtAl:2021:AspeedpreservingHilbertgradientflowforgeneralizedintegralMengercurvature} 
 to prove long-time existence for a suitably projected Hilbert gradient flow for
 $\intM^{(p,2)}$ for $p\in (\frac73,\frac83)$, where the projection was chosen
 to conserve the speed of the curves' parametrisations along the flow. 
 A corresponding long-time existence result for $\TP^{(p,2)}$ for
 $p\in (4,5)$ is also available \cite{steenebruegge_2022}.
 Considering gradient flows as ordinary differential equations in the respective
 energy space has the additional potential of developing numerical procedures 
 that are substantially more efficient and robust in comparison to the 
 existing numerical methods for the $L^2$-flow; see
 \cite{ReiterSchumacher:2021:SobolevGradientsfortheMobiusEnergy}
and 
 \cite[Sections 1.4 \& 7]{KnappmannSchumacherSteenebruggeEtAl:2021:AspeedpreservingHilbertgradientflowforgeneralizedintegralMengercurvature}.
 
 It is the purpose of the present paper to prove long-time existence results
 for the gradient flows of the three two-parameter energy families
 $E^{\alpha,p}$, $\intM^{(p,q)}$, and $\TP^{(p,q)}$ for the respective complete
 range of parameters given in \eqref{eq:ohara}, \eqref{eq:gen-menger}, 
 and \eqref{eq:gen-tan-point}. Since the underlying
 energy spaces are Sobolev-Slobodecki\v{\i} spaces
 that are in general only Banach spaces, we employ the general metric gradient approach
 of Ambrosio et al. \cite{AGS08} to construct curves of maximal slope, which turn
 out to be solutions of the doubly-nonlinear
  gradient flow equation.
 In order to do
 this, we need to control the curves' parametrisations along the
 flow, and for that we use the same constraint as in 
 \cite{KnappmannSchumacherSteenebruggeEtAl:2021:AspeedpreservingHilbertgradientflowforgeneralizedintegralMengercurvature}  
 but in a different manner. We add to the knot energy
 a suitable norm of the \emph{logarithmic
 strain}
 \begin{equation}\label{eq:log-strain}
 \Sigma(\g):=\log |\g'|
 \end{equation}
 as a lower order penalty term, instead of projecting 
 onto the null space of
 its differential $D\Sigma[\g]$ as in 
 \cite{KnappmannSchumacherSteenebruggeEtAl:2021:AspeedpreservingHilbertgradientflowforgeneralizedintegralMengercurvature}.  
 As a second ingredient we restrict 
 the total energy  to 
 a reflexive and uniformly convex Banach space $\CL $ compactly 
 embedded in the respective  energy space $\CB$ of closed curves,
 to account for the quite restrictive assumptions in the general 
 existence theory
 for curves of maximal slope. 
 
 To summarise these ideas, given any knot energy 
 $\EL :\CB\to (-\infty,\infty]$, 
 and some number $\kappa > 1$, we consider the  
 \emph{total energy} $\phi:\CL\to (-\infty,\infty]$
defined as  
 \begin{equation}\label{eq:total-energy}
 \textstyle
 \phi(\g):=\begin{cases}
 \EL(\g)+\|\Sigma(\g)\|_{\AL}^\kappa, & \textnormal{if $\g\in\CL$ is regular and injective}\\
 +\infty, & \textnormal{else,}
 \end{cases}
 \end{equation}
 where the Banach space $\AL$ consists of real-valued
 functions with exactly one order lower 
 in differentiability than the curves $\g
 \in\CB$, 
 since the scalar-valued
 logarithmic strain consumes one derivative.

 By means of the \emph{$\theta$-duality mapping $\FJ_{\CL,\theta}:\CL\to 2^{\CL^*}$ with $\theta\in (1,\infty)$} defined by
 \begin{equation}\label{eq:dual}
 \textstyle
 \FJ_{\CL,\theta}(x):=\{\xi\in\CL^*\colon\langle \xi,x\rangle_{\CL^*\times\CL}=\|x\|_\CL\cdot \|\xi\|_{\CL^*},\, \|x\|_\CL^\theta=\|\xi\|_{\CL^*}^{\beta}\},\,
 \tfrac1\theta+\tfrac1{\beta}=1,
 \end{equation}
where $\CL^*$ denotes the dual space of $\CL$,
 we state our main result.

 \begin{theorem}[Long-time existence]\label{thm:1.1}
 For any $\kappa,\theta\in (1,\infty)$, $\eps\in (0,\infty)$, and any given regular and
 injective curve $\g_0\in\CL_\eps$, there exists a mapping $
 {\bm u}\in C^1([0,\infty),\CL_\eps)$ with ${\bm u}(0)=\g_0$, such 
 that the closed curves $\bm u(t)$ are injective and 
 regular for all $t\ge 0$, satisfying
 \begin{equation}\label{eq:gradflow}
 \textstyle
 \frac{\dif}{\dif t}{\bm u}(t)=
 -\FJ_{\CL_\eps,\theta}^{-1}\big(D\phi[{\bm u}(t)]\big)\quad\Foa t\ge 0,
 \end{equation}
 for {\bf any one of the following choices}:
 \begin{enumerate}
 \item[\rm (i)]
 $\EL:=E^{\alpha,p}$ for $\alpha\in (0,\infty)$, $p\in [1,\infty)$ 
 satisfying
 $2< \alpha p < 2p +1$, and
\\$
 \AL:=W^{\frac{\alpha p-1}{2p},2p}(\R/\Z),\,
\CB:=W^{1+\frac{\alpha p-1}{2p},2p}(\R/\Z,\R^n),\,
 \CL_\eps:=W^{1+\frac{\alpha p-1}{2p}+\eps,2p}(\R/\Z,\R^n);$
 \item[\rm (ii)]
 $\EL:=\intM^{(p,q)}$ for $q\in (1,\infty), $ $ p\in (\frac23 q+1,q+\frac23),$ and
\\$
 \AL:=W^{\frac{3p-2}q-2,q}(\R/\Z),\,
 \CB:=W^{\frac{3p-2}q-1,q}(\R/\Z,\R^n),\,
 \CL_\eps:=W^{\frac{3p-2}q-1+\eps,q}(\R/\Z,\R^n);$
  \item[\rm (iii)]
  $\EL:=\TP^{(p,q)}$ for $q\in (1,\infty)$, $p\in (q+2,2q+1),$ and
\\$
 \AL:=W^{\frac{p-1}q-1,q}(\R/\Z),\,
 \CB:=W^{\frac{p-1}q,q}(\R/\Z,\R^n),\,
 \CL_\eps:=W^{\frac{p-1}q+\eps,q}(\R/\Z,\R^n).$
 \end{enumerate}
  \end{theorem}
  Since \eqref{eq:gradflow} is a gradient flow
  the energy  decreases along the flow.
  In addition,  the Banach space
  $\CL$ continuously embeds into $C^1(\R/\Z,\R^n)$ in all three 
  alternatives. Therefore,
  the knot type $[\g_0]$ of  the initial loop $\g_0$ is preserved along the flow,
  since knot classes are stable with respect to $C^1$-deformations \cite{reiter_2005,blatt_2009a}.
  \begin{corollary}\label{cor:1.2}
 Any $\bm u \in \AC_\textnormal{loc}([0, \infty), \CL_\eps)$ starting at an injective and regular curve $\g_0\in\CL_\eps$ and solving \eqref{eq:gradflow} for almost 
  all $t > 0$ is a $\theta$-curve of maximal slope with respect to the strong upper gradient\footnote{See Section \ref{sec:2} where we briefly
review
the essentials for metric gradient flows.} $\norm[\CL_\eps^*]{D\phi[\cdot]}$. 
  In particular, the
  total energy $\phi$ 
  decreases along $\bm u$, i.e., 
  $\phi(\bm u(t)) \le
  \phi(\bm u(s))$ for all $t>s\ge 0$.
  Additionally, $u \in C^1([0,\infty), \CL_\eps)$ and the knot type is preserved along the flow, that is,
  $[\bm u(t)]=[\g_0]$ for all $t\ge 0.$
  \end{corollary}
  
\begin{remark}\label{rem:after-main-thm}
1.\, 
Notice that in general the duality mapping $\FJ_{\CL,\theta}$
defined in \eqref{eq:dual} is set-valued, so that one would expect
to at most solve
the differential inclusion $-D\phi[{\bm u}(\cdot)]
\in\FJ_{\CL,\theta}({\bm u}'(\cdot))$ on $[0,\infty)$ instead of 
the gradient flow equation
\eqref{eq:gradflow}. The specific properties of  the 
Banach space $\CL:=\CL_\eps$ in \thref{thm:1.1}, however, imply that $\FJ_{
\CL_\eps,\theta}$
is not only single-valued, but also a homeomorphism between 
$\CL_\eps$ and $\CL_\eps^*$.

2.\,
The choice of the Banach space $\CB$
in the alternatives
(i), (ii), and (iii) of \thref{thm:1.1} is maximal in the
following sense:  A regular injective curve $\g\in\CB$ has finite
energy $\EL$, and, on the other hand, if a regular injective 
$C^1$-curve $\g$ satisfies
$\EL(\g)<\infty$ then its suitably rescaled
arc length parametrisation is contained
in $\CB$;
see 
\cite[Theorem~1.1]{Blatt:2012:BoundednessandRegularizingEffectsofOHarasKnotEnergies}\footnote{Unfortunately, 
there is a typographical error in said theorem, 
for the correct constant cf.\ its proof or later papers by 
the same author, e.g. 
\cite{BlattReiter:2013:StationarypointsofOHarasknotenergies}.},
\cite[Theorem 1]{BlattReiter:2015:TowardsaregularitytheoryforintegralMengercurvature}, and
\cite[Theorem 1.1 \& Remark 1.2]{BlattReiter:2015:Regularitytheoryfortangentpointenergiesthenondegeneratesubcriticalcase}. 

3.\,
There is complete freedom in the choice of the parameters
$\kappa,\theta\in (1,\infty)$, and $\eps>0$ in \thref{thm:1.1}. 
For $\kappa=1$ 
we still obtain  a curve of maximal slope for the total 
energy $\phi$ with respect to a strong
upper gradient
which, however,  
does not coincide with $\|D\phi\|_{\CL^*}$ any more;
see \thref{prop:kappa1}.
But it is presently unclear if that curve of maximal slope
also solves a differential inclusion.

The limiting process,
$\epsilon\searrow 0$, on the other hand, approximating the correct energy space $\CB$ for the 
respective knot energy $\EL$ in cases (i)--(iii) of \thref{thm:1.1}, 
yields a subsequence of solutions of \eqref{eq:gradflow}
that converges pointwise weakly to a limit mapping $\bm{u}^{\ast} \in AC^{\theta}\left( [0,\infty),\CB \right)$, provided that the initial curves $\gamma_{0,\epsilon}$ are well-prepared; see \thref{cor:eps_to_0} and the more general result \thref{prop:eps_to_0} for metric spaces, which is similar in spirit as \cite[Theorem 2]{Serfaty.2011}, where a more general limiting process is investigated. There, however, the existence of the limiting curve is assumed, see also \cite[Remark 1]{Serfaty.2011}.
While we know that the energy does not exceed its initial value $\phi(\bm{u}^{\ast}(0))$, it is unclear whether  $\bm{u}^{\ast}$ is a curve of maximal slope 
in the limiting energy space $\CB$. It would be if we had  a weakly lower
semicontinuous strong upper gradient for $\phi$ as shown in \thref{lemma:Missing_Ingredient},
a condition that is also used in \cite[Theorem 2.5]{Braides.2016}
 and in \cite[Theorem 2]{Serfaty.2011}.

In addition, multiplying 
the term $\|\Sigma(\g)\|_\CB^\kappa$ in \eqref{eq:total-energy} with a 
 small prefactor $\vartheta>0$ and sending $\vartheta$ to zero 
 gives rise 
 to yet another interesting limiting process, 
 similarly as in \cite{gerlach-etal_2017}, where such a procedure
 was used to study elastic knots.

 4.\,
 We do not know at this point if the solution $\bm{u}(\cdot)$ 
 of \eqref{eq:gradflow} is unique. Moreover, it is open whether
 $\bm{u}(t)$
 subconverges or even converges
 to a critical point of the total energy $\phi$ as  $t\to\infty$,
 not to speak of any information about convergence rates. 
 These convergence issues 
would possibly
 require to study the  second variation of the total energy
 $\phi$ and a suitable
 {\L}ojasiewicz-Simon inequality as carried out for the $L^2$-flow
 of the M\"obius energy $E^{2,1}$ 
 in \cite[Sections 4 \& 5]{blatt_2012b}; 
 see also the initial
 analysis of the kernel of the second variation
 of  $E^{2,1}$ in 
  \cite[Section 3]{biesenbach_2021}.

5.\,
Let us finally mention why we preferred the metric gradient approach 
to the study of ordinary differential equations in Banach spaces. 
First of all, the Picard-Lindel\"of theory requires Lipschitz
continuity of the right-hand side of the equation. This seems out of
reach in the present context, where the duality mapping $\FJ_{\CL,\theta}$ fails to be Lipschitz unless the underlying Banach space
is Lindenstrau{\ss} convex \cite{zemek_1991}. Indeed, here we deal
with Sobolev-Slobodecki\v{\i} spaces, and the Lindenstrau{\ss}
convexity requires an integrability that is at most quadratic; see
\cite[Theorem 7]{bynum_1976} in combination with 
\cite[Proposition 3.6]{cheng-ross_2015}. However, more general 
existence results with a compact operator on the right-hand side
such as \cite[Ch. VI, Thm. 3.1]{martin_1976} could probably be
used to obtain at least short-time existence. To extend this
to long-time existence would then require further estimates which
do not seem to provide a short cut. An interesting alternative
approach could be a vanishing viscosity method as performed  recently
for the $p$-curvature integral 
by Blatt and N. Vorderobermeier, and this might lead to stronger results; see \cite{blatt-vorderobermeier_2021b} in comparison to
\cite{blatt-vorderobermeier_2021a}.
\end{remark}

The paper is structured as follows. In Section 
\ref{sec:2} we  recall the basic  notions of the metric gradient
flow approach following \cite{AGS08} but slightly adapted to
our context. There, we also also revisit how these notions manifest in Banach spaces. 
 In Section
\ref{sec:3} we prove an abstract existence theorem for curves of maximal
slope for the total energy $\phi$ in \eqref{eq:total-energy}
under certain assumptions
on an otherwise arbitrary knot energy $\EL$; see
\thref{theorem:existenceCOMS}. We
also treat in \thref{prop:kappa1}
the limiting case $\kappa=1$.
In Section \ref{sec:4}
we verify in detail the assumptions of
\thref{theorem:existenceCOMS} for the three energy families
$E^{\alpha,p}$, $\intM^{(p,q)}$, and $\TP^{(p,q)}$, thus proving
\thref{thm:1.1}. One of these assumptions is the continuous
differentiability of the knot energy, which is known for the
generalised integral Menger curvature and tangent-point energies,
but -- to the best of our knowledge -- 
for O'Hara's energies $E^{\alpha,p}$ so far only for $p=1$;
see 
\cite[Theorem 1.1]{BlattReiter:2013:StationarypointsofOHarasknotenergies}. 
Kawakami and Nagasawa 
\cite{KawakamiNagasawa:2020:VariationalformulaeandestimatesofOHarasknotenergies} established $L^1$-bounds for the integrands of the first and 
second variation of a variant of O'Hara's energy that
coincides with $\Eap$ only on curves parametrised by arc length.
Since
we need bounds for more general parametrisations and their ansatz seems non-transferable, we have included a full proof of continuous 
differentiability of $\Eap$ in Section \ref{sec:5}; see \thref{theorem:C1OHara}, which may be of independent interest.
Let us add that the limiting process $\varepsilon\to 0$ for solutions
of \eqref{eq:gradflow}
is treated
in \thref{prop:eps_to_0} in the general metric
setting, and specified in \thref{cor:eps_to_0}
to the 
situation described in \thref{thm:1.1}.
The appendix contains some technical 
material on the geometry of Sobolev-Slobodecki\v{\i} spaces and
some differentiation rules.

\section{Preliminaries}\label{sec:2}

\subsection{Curves of Maximal Slope in Metric Spaces}\label{sec:2.1}

Given a complete metric space $(\sS,d)$, an interval $I\subset \R$, 
and a number $\theta\ge 1$, a 
\textit{$\theta$-absolutely continuous curve} 
is a curve ${\bm u}: I \to \sS$ such that there exists a 
map $m\in L^{\theta}\left( I \right)$  with the property
$
	d\big(\bm u(s),\bm u(t)\big) 
	\le \int_s^t m(r)\;\dif r
    $
    for all $s,t\in I$ with $s<t$.
The set of all such curves is denoted by 
$AC^{\theta}(I,\sS)$ (writing $\AC_{\loc}^\theta(I,\sS)$
for the local variant of this space if $m$ is only in 
$L^\theta_{\loc}(I)$,
and abbreviating $AC(I,\sS):=AC^1(I,\sS)$). 
According to \cite[Theorem 1.1.2]{AGS08} 
every $\bm u\in AC_{\loc}^\theta(I,\sS)$ is 
\textit{metrically differentiable} almost everywhere in the 
following sense: its \textit{metric derivative}  
$
|\bm u'|(t):=\lim_{s\to t} 
\frac{d(\bm u(t),\bm u(s))}{|t-s|}
$ 
exists for almost every $t\in I$.

 Any functional $\phi\colon\sS\to (-\infty,\infty]$ with non-empty 
 \emph{effective domain} $\sD(\phi):=\{u\in\sS\colon \phi(u)
 < \infty\}$
admits a so-called \textit{local slope} $|\partial\phi|(u)$ of 
$\phi$ at $u\in\sD(\phi)$ given by
$
	|\partial \phi|(u) := \limsup_{d(v,u)\to 0}  
	\tfrac{(\phi(u)-\phi(v))^+}{d(u,v)} 
	\in [0,\infty]$,
where $a^+:= \max\{0,a\}$ for $a\in \R$ \cite[Definition 1.2.4]{AGS08}.
A map $g:\sS\to[0,\infty]$ is a \textit{strong upper gradient} for 
$\phi$ on $\sD(\phi)$, if for every curve $\bm u \in AC
( I,\sD(\phi))$ the composition $g \circ \bm u$ is Borel-measurable 
and
\begin{align}
    \textstyle
	\big|\phi\circ \bm u(t)-\phi\circ \bm u(s)\big| 
	\le \int_s^t g\circ \bm u (r)| \bm u'|(r) \;\dif r \quad\textnormal{for all $s,t\in I$ with $s<t$.}
	\label{EQN_h_1}
\end{align}
Notice that this notion of a strong upper gradient modifies slightly
that of 
\cite[Definition 1.2.1]{AGS08} in that inequality \eqref{EQN_h_1}
is only required for absolutely continuous curves whose image
is contained in the effective domain  $\sD(\phi)$.
This restriction does not affect  \thref{Theorem_Existence_Metric_COMS} 
below upon which our existence results build.
In the situations encountered in this work, the local slope turns out to
be a strong upper gradient.

With these notions at hand, one can define a 
\emph{$\theta$-curve of maximal slope for $\phi$ with respect to the strong upper gradient $g$ starting at $u_0\in\sD(\phi)$}. It is a curve 
$\bm u \in \linebreak AC^{\theta}_{\loc}\big([0,\infty),\sS\big)$ that satisfies $\bm u(0)=u_0$ and the \textit{energy dissipation equality}
\begin{align}
    \label{eq:energyDissipation}
    \textstyle
	\phi(\bm u(t))-\phi(\bm u(s)) = -\frac{1}{\beta}\int_{s}^t g^{\beta}(\bm u(r)) \;\dif r - \frac{1}{\theta} \int_s^t |\bm u'|^{\theta}(r)\,\dif r
\end{align}
for all $0\le s< t<\infty$, where $\theta^{-1} + \beta^{-1} =1$. 
At first sight, this definition seems to differ from that given
in \cite[Definition 1.3.2]{AGS08}.
Notice, however,  that in the present more specific situation,  
where $u_0\in\sD(\phi)$ and $g$ is 
a strong upper gradient, our definition is actually equivalent
to that in \cite{AGS08}.  
Indeed, since $\phi(u_0)<\infty$ and $\phi(u)>-\infty$ for all 
$u\in\sS$, equality \eqref{eq:energyDissipation} with $s=0$ and 
Young's inequality imply that $g\circ \bm u |\bm u'|$ is integrable on every compact interval $[0,t]$. Hence, by the defining inequality 
\eqref{EQN_h_1} for strong upper gradients, $\phi\circ\bm u$ is 
locally absolutely continuous and we deduce the defining
inequality (1.3.13) of \cite[Definition~1.3.2]{AGS08} 
with $\varphi = \phi\circ \bm u$. 
Conversely, 
\cite[Remark 1.3.3]{AGS08} implies that 
\eqref{eq:energyDissipation} holds. 
In that case, we also have the identities
\begin{align}
	g^{\beta}(\bm u(r)) = |\bm u'|^{\theta}(r) = 
	|\bm u'|(r)g\circ\bm u(r)\quad\textnormal{for a.e. 
	$r\in (0,\infty)$.} \label{EQN_h_5}
\end{align}
In fact, one can also require \eqref{eq:energyDissipation} only to 
hold with ``$\le$'' instead of equality, since the converse 
inequality is always satisfied as a consequence of 
Young's inequality and \eqref{EQN_h_1}.

A common way to obtain a curve of maximal slope is by carrying out 
three steps. Firstly, for a given step size, a discretised version of 
the energy dissipation equation is solved by iteratively solving 
minimisation problems associated with the step size. Secondly, 
using an interpolation method, each of the thus obtained piecewise 
constant solutions is transformed into a continuous curve. These 
curves form a relatively compact set and a limit curve is extracted. 
This limit curve is commonly referred to as a 
\textit{(generalised) minimising movement} and satisfies a weaker form 
of the energy dissipation equality. Thirdly, it remains to check that 
the minimising movement in fact also satisfies the energy dissipation 
equality. General assumptions on the functional have been 
formulated that ensure that each of the previous steps can be executed. 
We state these assumptions collected from \cite[Section 2.1 \& Remark 2.3.4]{AGS08} and the corresponding existence theorem.

\begin{assM}\label{Assumptions_Metric}
	The complete metric space $(\sS,d)$  is en\-dow\-ed with an 
additional \emph{weak topology} $\sigma$ which is Hausdorff, 
weaker than the topology induced by the metric $d$, and such that  $d$
is sequentially weakly lower semi-continuous\footnote{In our 
application, $\sigma$ will be the weak topology on a Banach space, 
so there is no risk of confusion regarding the expression ``weak''.}, 
i.e., for all sequences $u_k \xrightharpoonup[]{\sigma} u$ and 
$v_k \xrightharpoonup[]{\sigma} v$ as $k\to\infty$ there holds
$
			d(u,v) \le \liminf_{k\to\infty} d(u_k,v_k)$.
\end{assM}
{\bf Assumptions on the functional 
$\phi:\sS\to (-\infty,\infty]$ with $\sD(\phi)\not=\emptyset$ and 
fixed $\theta\in (1,\infty).$}\,
\begin{enumerate}
		\item[\rm ($\Phi$1)]
		\label{Assumption_Metric_2}
The functional $\phi$ is sequentially weakly lower semi-continuous 
on $d$-bounded sets, i.e., if $\sup_{k,l\in\N}d(u_k,u_l)<\infty$
		for a sequence $u_k \xrightharpoonup[]{\sigma} u$ as
		$k\to\infty$, then
$
\phi(u)\le\liminf_{k\to\infty}\phi(u_k)$.
\item[\rm ($\Phi$2)] \label{Assumption_Metric_3}
The functional $\phi$ is coercive in the sense that there exists 
$u_{\ast}\in \sS$, $B\in\R$, and $C>0$ such that 
$\phi(u) \ge B - C d(u,u_{\ast})^{\theta}$ for all $u\in \sS$.
\item[\rm ($\Phi$3)] \label{Assumption_Metric_4} 
Every $d$-bounded subset of a sublevel set of $\phi$ is relatively
weakly sequentially compact, i.e., for a sequence $(u_k)_k\subset\sS$
with $\sup_{k\in\N}\phi(u_k)<\infty$ and $\sup_{k,l\in\N}d(u_k,u_l)<
\infty $ there is $u\in\sS$ and a subsequence $(u_{k_m})_m
\subset (u_k)_k$ such that $u_{k_m}\xrightharpoonup[]{\sigma} u$ as
                $m\to\infty$.
\item[\rm ($\Phi$4)] \label{Assumption_Metric_5} 
The local slope $|\partial\phi|$ of $\phi$ is weakly sequentially 
lower semi-continuous on $d$-bounded subsets of sublevel sets 
of $\phi$, that is,  for every sequence 
$u_k \xrightharpoonup[]{\sigma} u$ as $k\to\infty$ with
	$
			\sup_{k\in\N}\{ d(u_k,u),\phi(u_k) \} <\infty$
			one has 
$
|\partial\phi|(u) \le \liminf_{k\to\infty} |\partial\phi|(u_n)$.
\item[\rm ($\Phi$5)] \label{Assumption_Metric_6}  
The local slope of $\phi$ is a strong upper gradient on\footnote{This 
is a slightly restricted version  of \cite[Remark 2.3.4 (i)]{AGS08},
but this does not affect the validity of \thref{Theorem_Existence_Metric_COMS} below.}
$\sD(\phi)$.
	\end{enumerate}

\begin{theorem}[Curves of maximal slope exist.]
\label{Theorem_Existence_Metric_COMS}
	Let $(\sS,d)$ be a metric space satisfying Assumption {\rm (M)},
	and suppose
$\phi:\sS\to (-\infty,\infty]$ satisfies for given
$\theta\in (1,\infty)$ Assumptions {\rm ($\Phi$1)--($\Phi$5)}, and let
$u_0\in\sD(\phi)$. Then there exists a $\theta$-curve of 
maximal slope for $\phi$ with respect to the strong
upper gradient 
$|\partial \phi|$ starting at $u_0$.
\end{theorem}
\begin{proof}	
Notice that Assumption ($\Phi$2) is equivalent to  $
\inf_{v\in\sS} [Cd(v,u_{\ast})^{\theta} + \phi(v)] >-\infty,$
	which ensures that $\tau_{\ast} := (\theta C)^{-\frac{1}{\theta-1}}$ satisfies
$
\inf_{v\in\sS} 
[\frac{1}{\theta}\tau_{\ast}^{1-\theta} d(u_{\ast},v)^{\theta} 
+ \phi(v)] > -\infty$.
Therefore, for every uniform partition 
$\{ 0< \frac{1}{m}<\frac{2}{m}<\dots<\frac{k}{m}<\dots \}$ of 
$[0,\infty)$ with $m\in \N$ such that $\frac{1}{m}<\tau_{\ast}$, 
Assumptions (M) and ($\Phi$1)-($\Phi$3) 
enable us to employ \cite[Corollary 2.2.2]{AGS08} to find a solution 
of the recursive minimisation problem 
\cite[(2.0.4)]{AGS08} starting at $u_0$. 
Then, \cite[Proposition 2.2.3]{AGS08} implies that the family of 
these discrete solutions admit, up to a subsequence, an absolutely 
continuous limit curve $ \bm u$ (called a 
\emph{generalised minimising movement}). Now, Assumption ($\Phi$4) 
implies that the relaxed slope $|\partial^{-}\phi|$, 
\cite[(2.3.1)]{AGS08}, is equal to the local slope on $\sD(\phi)$. 
Moreover, by Assumption ($\Phi$5), it is also a strong upper gradient on 
$\sD(\phi)$. Therefore, Theorem \cite[Theorem 2.3.3]{AGS08} is 
applicable and we deduce that $\bm u$ is a curve of maximal slope 
with respect to $|\partial^{-}\phi| = |\partial\phi|$. 
Notice that our additional restrictions in the definition of 
upper gradients and Assumption ($\Phi$5) are non-essential, as both 
only play a role at \cite[(3.4.2)]{AGS08}. There, the curve $
\bm u$ has finite energy because of \cite[(3.4.1)]{AGS08}.

\end{proof}

The following proposition shows how curves of maximal slope in smaller subspaces of a given metric space give rise to a limiting curve.

\begin{proposition} \label{prop:eps_to_0}
	Let $\theta\in(1,\infty)$. Let $(\sS_0,d_0, \sigma)$ be a metric space that satisfies Assumption {\rm (M)} and let $\big((\sS_{\epsilon},d_{\epsilon})\big)_{\epsilon>0}$ be metric spaces such that $\sS_{\epsilon} \subset \sS_{0}$ and such that there exists a constant $c_0>0$ such that $d_{0}(u,v) \le c_0 d_{\epsilon}(u,v)$ holds for all $u,v\in \sS_{\epsilon}$ and all $\epsilon>0$. Let $\phi:\sS\to(-\infty,\infty]$ be a functional that satisfies Assumptions {\rm ($\Phi$1)--($\Phi$3)} with $\sS=\sS_0$ and let $u_0 \in \sD(\phi)$. Assume that for every $\epsilon>0$ there exists $u_{0,\epsilon} \in\sS_{\eps}$ such that
	\begin{equation}
	\textstyle
		\sup\limits_{\epsilon>0} d_0(u_{0,\epsilon}, u_0) <\infty,\quad u_{0,\epsilon} \xrightharpoonup[(\epsilon\to0)]{\sigma} u_0, \quad \text{and}\quad \phi(u_{0,\epsilon})\xrightarrow[(\epsilon\to 0)]{}\phi(u_0)\label{anfangskonvergenz}
	\end{equation}
	and let $\bm{u}_{\epsilon} \in AC_{\mathrm{loc}}^{\theta}\left( [0,\infty),(\sS_{\epsilon}, d_{\eps}) \right)$ be a $\theta$-curve of maximal slope for $\phi$ with respect to the strong upper gradient $g_{\epsilon}$ starting at $u_{0,\epsilon}$. Then there exists a subsequence $\epsilon_k \to 0$ and a curve $\bm{u^{\ast}} \in AC^{\theta}_{\loc}\left( [0,\infty),(\sS_0,d_0) \right)$ such that $\bm{u^{\ast}}(0)=u_0$ and
		$
			\bm{u}_{\epsilon_k}(t) \xrightharpoonup[(k\to \infty)]{\sigma} \bm{u}^{\ast}(t)$
	holds for every $t\ge 0$. Moreover, $\phi(\bm{u}^{\ast}(t)) \le \phi(u_0)$ for all $t\ge 0$.
\end{proposition}
\begin{proof}
Let us first fix some notation. For a given metric space $\sS \in \{ \sS_0, \sS_{\epsilon}\}$, a locally absolutely continuous curve $\bm{v} \in AC_{\mathrm{loc}}\left( [0,\infty),\sS \right)$ and a point $v\in \sS$, we denote by $|\bm{v}'|_{\sS}(r)$ the metric derivative of $\bm{v}$ at $r\ge 0$ taken with respect to the metric on $\sS$. Moreover, since we are only interested in small $\epsilon$, we may assume that $\sup\limits_{\epsilon >0} \phi(u_{0,\epsilon})<\infty$.

The arguments presented here are continuous analogues of those used in \cite[Lemma 3.2.2]{AGS08} and \cite[Corollary 3.3.4]{AGS08}. The existence of a converging subsequence relies on the theorem of Arzel{\`a}-Ascoli \cite[Proposition 3.3.1]{AGS08}. We will check its prerequisites, which are consequences of the energy dissipation equalities
\begin{align}
    \textstyle
	\phi\big(\bm{u}_{\epsilon}(t)\big) + \frac{1}{\theta} \int_s^t |\bm{u}_{\epsilon}'|_{\sS_{\epsilon}}^{\theta}(r) \;\dif r + \frac{1}{\beta} \int_s^t g_{\epsilon}^{\beta}\big(\bm{u}_{\epsilon}(r)\big)\;\dif r =  \phi\big(\bm{u}_{\epsilon}(s)\big),\; 0\le s<t<\infty. \label{eq:1}
\end{align}
From the assumptions on the metrics $d_0,d_{\epsilon}$, it follows that
\begin{align*}
    \textstyle
	d_0(\bm{u}_{\epsilon}(t),\bm{u}_{\epsilon}(s)) \le c_0 d_{\epsilon}(\bm{u}_{\epsilon}(t),\bm{u}_{\epsilon}(s)) \le c_0 \int_s^t |\bm{u}_{\epsilon}'|_{\sS_{\epsilon}}(r)\;\dif r\quad\Foa 0\le s<t<\infty.
\end{align*}
Therefore, $\bm{u}_{\epsilon} \in AC_{\mathrm{loc}}^{\theta}\left( [0,\infty),(\sS_0, d_0) \right)$ with $|\bm{u}_{\epsilon}'|_{\sS_0}(r) \le c_0 |\bm{u}_{\epsilon}'|_{\sS_{\epsilon}}(r)$ for almost every $r\ge 0$. Notice that Assumption {\rm ($\Phi$2)} also holds for $u_{\ast} =u_0$, possibly with different constants. We deduce from Assumption {\rm ($\Phi$2)} and \eqref{eq:1} with $s=0$ that
\begin{align}
\textstyle
\int_{0}^{t} |\bm{u}_{\epsilon}'|^{\theta}_{\sS_0}(r) \;\dif r &\textstyle\le c_0^{\theta}\int_0^{t} |\bm{u}_{\epsilon}'|^{\theta}_{\sS_{\epsilon}}(r)\;\dif r \le \theta c_0^{\theta} \big(\phi(u_{0, \epsilon}) - \phi(\bm{u}_{\epsilon}(t)) \big)
\nonumber\\
&\le \theta c_0^{\theta} \big(\phi(u_{0, \epsilon}) - B \big) + C\theta c_0^{\theta}\, d_0^{\theta}(\bm{u}_{\epsilon}(t),u_0)  \label{eq:2}
\end{align}
On the other hand, since $\bm{u}_{\epsilon}$ and $d_0(\bm{u}_{\epsilon}(\cdot),u_0)$ are locally absolutely continuous, we have
\begin{gather*}
	\textstyle\frac{\dif }{\dif r} d_0( \bm{u}_{\epsilon}(r), u_0) = \lim\limits_{s\to r} \frac{d_0(\bm{u}_{\epsilon}(s), u_0)- d_0(\bm{u}_{\epsilon}(r), u_0)}{s-r} \le \lim\limits_{s\to r} \frac{d_0(\bm{u}_{\epsilon}(s), \bm{u}_{\epsilon}(r))}{|s-r|} = |\bm{u}_{\epsilon}'|_{\sS_0}(r)\\
	\textstyle \text{and thus} \quad\frac{\dif}{\dif r} d_0^{\theta}( \bm{u}_{\epsilon}(r), u_0) \le \theta d_0^{\theta-1}( \bm{u}_{\epsilon}(r), u_0) |\bm{u}_{\epsilon}'|_{\sS_0}(r)\quad  \text{ for almost every $r\ge 0$.}
\end{gather*}
Using Young's inequality with $\delta>0$, $a = \sqrt[\theta]{\delta}\abs{\bm u_\epsilon'}_{\sS_0}$, $b=\theta\sqrt[\theta]{\delta^{-1}} d_0^{\theta - 1}$, and recalling that $\beta = \tfrac{\theta}{\theta-1}$, we obtain for almost every $r\ge 0$
\begin{align}
\textstyle
	\frac{\dif}{\dif r} d_0^{\theta}( \bm{u}_{\epsilon}(r), u_0) \le \frac{\delta}{\theta} |\bm{u}_{\epsilon}'|_{\sS_0}^{\theta}(r) + \frac{\theta^{\beta}}{\beta\delta^{\beta-1}} d_0^{\theta}( \bm{u}_{\epsilon}(r), u_0).\label{2.8}
\end{align}
Therefore,
\begin{align*}
\textstyle
	d^{\theta}_0(\bm{u}_{\epsilon}(t),u_0) &\textstyle= d_0^{\theta}(u_{0,\eps},u_0) + \int_{0}^{t} \frac{\dif}{\dif r} d^{\theta}_0(\bm{u}_{\epsilon}(r),u_0)\;\dif r\\
	&\textstyle \overset{\eqref{2.8}}{\le} d_0^{\theta}(u_{0,\eps},u_0) + \frac{\delta}{\theta}\int_0^t |\bm{u}_{\epsilon}'|_{\sS_0}^{\theta}(r) \;\dif r + \frac{\theta^{\beta}}{\beta\delta^{\beta-1}} \int_0^t d_0^{\theta}( \bm{u}_{\epsilon}(r), u_0)\;\dif r\\
	\textstyle
	\textstyle\overset{\eqref{eq:2}}{\le} d_0^{\theta}(u_{0,\eps},u_0)& + \delta c_0^{\theta} \big(\phi(u_{0, \epsilon}) - B \big) +\delta Cc_0^{\theta}\, d_0^{\theta}(\bm{u}_{\epsilon}(t),u_0) + \frac{\theta^{\beta}}{\beta\delta^{\beta-1}} \int_0^t d_0^{\theta}( \bm{u}_{\epsilon}(r), u_0)\;\dif r.
\end{align*}
Choosing, e.g., $\delta = (2 C c_0^{\theta})^{-1}$, and using \eqref{anfangskonvergenz}, we deduce for some constants $c_1,c_2>0$ independent of $\epsilon$ and $t$ the inequality
\begin{align}
\textstyle
	d_0^{\theta}(\bm{u}_{\epsilon}(t),u_0) \le c_1 + c_2 \int_0^t d_0^{\theta}( \bm{u}_{\epsilon}(r), u_0)\;\dif r.\nonumber
\end{align}
It follows from Gronwall's inequality \cite[Corollary 6.6]{Hale.1980} that
\begin{align}
\textstyle
	d_0^{\theta}(\bm{u}_{\epsilon}(t),u_0) \le c_1 \exp(c_2 t). \label{eq:6}
\end{align}

Now fix $T>0$. It follows immediately from \eqref{eq:1} and from \eqref{eq:6} that 
\begin{align}
\textstyle
	\sup\limits_{\epsilon>0} \sup\limits_{0\le t\le T}\phi(\bm{u}_{\epsilon}(t)) <\infty,\quad \text{and} \quad \sup\limits_{\epsilon>0}\sup\limits_{0\le t\le T} d_0(\bm{u}_{\epsilon}(t),u_0) <\infty. \label{eq:4}
\end{align}
Therefore, the maps $\fres{\bm{u}_{\epsilon}}{[0,T]}$ take their values in a $d_0$-bounded sublevel set of $\phi$, which, by Assumption {\rm ($\Phi$3)} is relatively weakly sequentially compact. Moreover, a consequence of \eqref{eq:2} and \eqref{eq:6} with $t=T$ is that
\begin{align}
    \textstyle
    \label{prop:eps_to_0:eq:LThetaBound}
	\sup\limits_{\epsilon>0}\Norm{\abs{\bm{u}_{\epsilon}'}_{\sS_0}}^{\theta}_{L^{\theta}([0,T])} = \sup\limits_{\epsilon>0}\int_{0}^{T} |\bm{u}_{\epsilon}'|^{\theta}_{\sS_0}(t) \;\dif t  <\infty.
\end{align}
Since $L^{\theta}([0,T])$
is reflexive, there exists a null-sequence $(\epsilon_k)_{k\in \N}$ such that, with the notation $\bm{u}_k := \bm{u}_{\epsilon_k}$, the metric derivatives $|\bm{u}_{k}'|_{\sS_0}$ converge weakly to some $A_{T}\in L^{\theta}([0,T])$. For all $0\le s<t\le T$ the characteristic function $\chi_{[s,t]}$ is of class $ L^{\beta}([0,T])$ (recall that $\theta^{-1}+\beta^{-1}=1$) and therefore
\begin{align}
    \textstyle
	\limsup\limits_{k\to \infty}  d_0(\bm{u}_{k}(t),\bm{u}_{k}(s)) \le \limsup\limits_{k\to \infty} \int_s^{t} |\bm{u}_{k}'|_{\sS_0}(r)\;\dif r = \int_s^t A_{T}(r) \;\dif r. \label{eq:3}
\end{align}

Using \cite[Proposition 3.3.1]{AGS08} with $\omega(s,t) := |\int_s^t A_{T}(r)\;\dif r|$, we deduce from \eqref{eq:3} and \eqref{eq:4} that there exists a map $\bm{u}^{\ast}_{T}:[0,T]\to \sS_0$, such that, up to a subsequence,
\begin{align}
\textstyle
\bm{u}_{k}(t) \xrightharpoonup[(k\to\infty)]{\sigma} \bm{u}^{\ast}_{T}(t) \Foa
t\in [0,T].
\label{eq:5}
\end{align}
By a diagonal argument, we find $\bm{u}^{\ast}:[0,\infty)\to \sS_0$ and $A \in L^{\theta}_{\loc}\left( [0,\infty) \right)$ such that, up to another subsequence, \eqref{eq:5} holds with $\bm{u}^{\ast}$ for every $t\ge 0$ and such that for every $T>0$ there holds $\bm{u}^{\ast} = \bm{u}^{\ast}_{T}$ and $A=A_{T}$ almost everywhere on $[0,T]$. Since the metric $d_0$ is sequentially weakly lower semicontinuous and because of \eqref{eq:3} and \eqref{eq:5}, we deduce that for all $0\le s<t<\infty$
\begin{align}
\textstyle
	d_0(\bm{u}^{\ast}(t),\bm{u}^{\ast}(s)) \le \liminf_{k\to\infty} d_0(\bm{u}_{k}(t),\bm{u}_{k}(s)) \le \int_s^t A(r) \;\dif r  \label{eq:7}
\end{align}
whence $\bm{u}^{\ast} \in AC^{\theta}_{\loc}\left( [0,\infty),(\sS_0,d_0) \right)$ and $|(\bm{u}^{\ast})'|_{\sS_0}(r) \le  A(r)$ for almost every $r\ge 0$.

Because of \eqref{eq:5}, Assumption {\rm ($\Phi$1)}, \eqref{eq:1} and \eqref{anfangskonvergenz}, the energy remains bounded by its initial value:
\begin{align}
 \phi(\bm{u}^{\ast}(t)) \le \liminf_{k\to\infty} \phi(\bm{u}_{k}(t)) \le \liminf_{k\to\infty} \phi(\bm{u}_{k}(0)) = \phi(u_0). \nonumber
\end{align}
\end{proof}

The following lemma shows that under additional assumptions the limiting curve is actually a curve of maximal slope. We will not need this result in the remainder of
this paper.
\begin{lemma} \label{lemma:Missing_Ingredient}
	If in addition to the assumptions of \thref{prop:eps_to_0} with $c_0=1$, we also have for a strong upper gradient $g$ of $\phi$ with respect to $\sS_0$ and for every $u_{\epsilon}\xrightharpoonup[(\epsilon\to0)]{\sigma} u$ in $\sS_0$ that $\liminf\limits_{\epsilon\to 0} g_{\epsilon}(u_{\epsilon}) \ge g(u)$, then the limiting curve $\bm{u}^{\ast}$ is a $\theta$-curve of maximal slope for $\phi$ starting at $u_0$.
\end{lemma}
\begin{proof}
	Let $\bm{u}^{\ast}$ and $(\bm{u}_k)_k$ be as in the proof of \thref{prop:eps_to_0}. Since $c_0=1$, we have $|\bm{u}_{\epsilon}'|_{\sS_0}(r) \le |\bm{u}_{\epsilon}'|_{\sS_{\epsilon}}(r)$ for a.e. $r\ge 0$. Using this inequality, {\rm ($\Phi$1)}, the weak convergence of $|\bm{u}_{k}'|_{\sS_0}$  to $A$ together with the inequality \eqref{eq:7}, the newly introduced assumption on the upper gradients, and \eqref{anfangskonvergenz}, we may pass to the limit inferior in the energy dissipation equality \eqref{eq:1} with $s=0$ to deduce that the limit curve $\bm{u}^{\ast}$ satisfies
	\begin{align}
	\textstyle
		\phi(\bm u^{\ast}(t)) + \frac{1}{\theta} \int_0^t |(\bm{u}^{\ast})'|_{\sS_0}^{\theta}(r) \;\dif r + \frac{1}{\beta} \int_0^{t} g^{\beta}(\bm{u}^{\ast}(r))\;\dif r \le \phi(u_0). \label{eq:8}
	\end{align}
  	Since $g$ is a strong upper gradient, we may use \eqref{EQN_h_1} in combination with Young's inequality to find that \eqref{eq:8} actually holds with equality. Subtracting equality \eqref{eq:8} with $s$ instead of $t$ from equality \eqref{eq:8} yields the energy dissipation equality \eqref{eq:energyDissipation} for $\bm{u}^{\ast}$.
\end{proof}

\subsection{Curves of Maximal Slope in Banach Spaces}\label{sec:2.2}

If the metric space is actually a real Banach space $(\sS,d) = 
(\CB, \Norm{\cdot}_{\CB})$ with dual space 
$(\CB^{\ast},\Norm{\cdot}_{\CB^{\ast}})$, then the objects of the 
previous section can be characterised in terms of classical 
derivatives; see, e.g. 
 \cite[Section 2.3]{Matt.2022} for a proof of the following result
 regarding the local slope and the metric derivative.

\begin{proposition} \label{Proposition_Metric_Quantities_in_Banach}
Let $(\CB,\Norm{\cdot}_{\CB})$ be a Banach space, $\phi:\CB\to (-\infty,\infty]$ a mapping, and suppose $\bm u: I\to \CB$ is a curve
defined on an interval $I\subset \R$. 
If $\phi$ is Fr{\'e}chet differentiable at some point 
$v\in \CB$ with derivative $D\phi[v]$, 
then $|\partial \phi|(v) = \Norm{D\phi[v]}_{\CB^{\ast}}$, and if
$\bm u$ is differentiable at $t\in I$ with 
derivative $\bm u'(t)$, then $|\bm u'|(t) = \Norm{\bm u'(t)}_{\CB}$.
\end{proposition}

If one has a more specific
Banach space $\CL$ with additional properties, then
$\theta$-curves of maximal slope satisfy differential inclusions. 
The following result is a simplified and slightly modified version of 
\cite[Proposition 1.4.1]{AGS08} and is stated in terms of the 
duality mapping
$\FJ_{\CL,\theta}:\CL\to 2^{\CL^*}$ with weight $\theta\in (1,\infty) $
defined in \eqref{eq:dual} of the introduction.

\begin{proposition} \label{Proposition_COMS_are_GFs}
Suppose $\CL$ is a Banach space,  
$\phi:\CL\to\R\cup \{\infty\}$ a mapping that is
Fr{\'e}chet-differentiable\footnote{One can easily show 
that under this
strong assumption on $\phi$ its local slope $|\partial\phi|$
is a strong upper gradient; see, e.g., \cite[Lemma 2.3.7]{Matt.2022}.}
on an open subset $\Omega\subset\CL$, and let
$\bm u \in AC^{\theta}_{\loc}\big([0,\infty),\Omega \big)$ be a 
$\theta$-curve of maximal slope for $\phi$ with respect to 
its local slope $|\partial \phi|$.
\begin{enumerate}
\item[\rm (i)] If $\CL$ has the Radon-Nikodym property then 
\begin{align}
\textstyle
-D\phi[\bm u(t)] \in \FJ_{\CL,\theta}\big( \bm u'(t)\big)
\quad \textnormal{for a.e. $t > 0$.} 
\label{EQN_h_6}
\end{align}
\item[\rm (ii)] If $\CL$ is reflexive, strictly convex, 
and has a G{\^a}teaux-differentiable norm on $\CL\setminus\{0\}$, then 
\begin{align}
\textstyle
\frac{\dif}{\dif t} \bm u(t) = 
- \FJ^{-1}_{\CL,\theta}\big( D\phi[\bm u(t)]\big)
\quad \textnormal{for a.e. $t > 0$.}  \label{EQN_h_7}
\end{align}
\end{enumerate}
\end{proposition}

\begin{proof}
(i)\,	The Radon-Nikodym property implies that absolutely continuous 
curves are differentiable almost everywhere, 
see e.g. \cite[Definition 1.2.5]{Arendt.2011}. 
Using the energy dissipation equality \eqref{eq:energyDissipation} 
together with \eqref{EQN_h_5} and \thref{Proposition_Metric_Quantities_in_Banach}, we infer that 
\begin{align*}
\textstyle
\phi\circ\bm u(t)-\phi\circ\bm u(s) &= 
-\int_s^t |\bm u'|(r) |\partial\phi|\circ \bm u(r)\;\dif r = 
-\int_s^t \Norm{\bm u'(r)}_{\CL} 
\Norm{D\phi[\bm u(r)]}_{\CL^{\ast}}\;\dif r
\end{align*}
	holds for all $0\le s<t<\infty$. By  the chain rule we obtain
\begin{align*}
\textstyle
-\Norm{\bm u'(t)}_{\CL} \Norm{D\phi[\bm u(t)]}_{\CL^{\ast}} = 
\frac{\dif}{\dif t} \phi\circ \bm u(t) = 
\langle D\phi[\bm u(t)],\bm u'(t) \rangle_{\CL^{\ast}\times \CL}
\quad\textnormal{for a.e. $t\ge 0$.}
\end{align*}
Together with identity \eqref{EQN_h_5} for $g:=|\partial\phi|$
 in combination with the definition of the duality 
mapping \eqref{eq:dual}
and \thref{Proposition_Metric_Quantities_in_Banach},
this implies \eqref{EQN_h_6}.

(ii)\,
If $\CL$ is reflexive, it has the Radon-Nikodym property 
\cite[Corollary 1.2.7]{Arendt.2011}. The differentiability of the norm 
$\Norm{\cdot}_{\CL}$ on $\CL\setminus\{0\}$
ensures that the duality mapping $\FJ_{\CL,\theta}$
is single-valued 
\cite[Corollary I.4.5]{Cioranescu:1990:GeometryofBanachSpacesDualityMappingsandNonlinearProblems}. 
The reflexivity of $\CL$ guarantees that $\FJ_{\CL,\theta}$ is 
surjective 
\cite[Theorem II.3.4]{Cioranescu:1990:GeometryofBanachSpacesDualityMappingsandNonlinearProblems}, 
and the strict convexity of $\CL$
implies that it is injective 
\cite[Theorem II.1.8]{Cioranescu:1990:GeometryofBanachSpacesDualityMappingsandNonlinearProblems}\footnote{Note that in the language of \cite{Cioranescu:1990:GeometryofBanachSpacesDualityMappingsandNonlinearProblems}, we use the weight function $\varphi(s) := s^{\frac \theta \beta}$.}.
Thus, the inclusion \eqref{EQN_h_6} is actually an  identity
and we deduce \eqref{EQN_h_7}.
\end{proof}

\section{Curves of maximal slope for abstract knot energies regularised by the logarithmic strain }\label{sec:3}

For most of the known knot energies $\EL$ the underlying energy space is
a Sobolev-Slobodecki\v{\i} space $W^{1+s,\rho}(\R/\Z,\R^n)$ for
some $s\in (0,1)$, $\rho>1/s$. In the still abstract setting of the present
section we regularise
any such 
$\EL$ by adding a power of the norm of the logarithmic strain
$\Sigma$ defined in \eqref{eq:log-strain} to form
a total energy $\phi$ as in \eqref{eq:total-energy}. 
It was shown in
\cite{KnappmannSchumacherSteenebruggeEtAl:2021:AspeedpreservingHilbertgradientflowforgeneralizedintegralMengercurvature}
that $\Sigma$ is continuously differentiable on regular injective
curves of that class.
\begin{lemma}[\protecting{\cite[Proposition~5.1]{KnappmannSchumacherSteenebruggeEtAl:2021:AspeedpreservingHilbertgradientflowforgeneralizedintegralMengercurvature}}] 
\label{Lemma_Log_Strain_is_Fréchet}
For any $s\in (0,1)$ and $\rho\in (\frac1s,\infty)$ one has
$
  \Sigma \in C^1( W^{1+s,\rho}_{\ir}( \R/\Z,\R^n) , 
  W^{s,\rho}( \R/\Z ) ).
	    $
          \end{lemma}
Here, and from now on we use the  following {\bf abbreviated notation}:
For a given space $\FL$ of closed curves
we denote from now on by  
 $\FL_{\mathrm{i}}$ the subset of all injective curves, 
  by $\FL_{\mathrm{r}}$ the regular curves, and 
   we write $\FL_{\mathrm{a}}$ for the arc length-parametrised curves
    contained in $\FL$, and we combine these indices $\mathrm{i,r}$
    and $\mathrm{a}$ as needed. 

To quantify injectivity of closed curves $\g$ we define the 
\emph{bi-Lipschitz 
constant} $\BiLip(\g)$ by means of
\begin{align*}\label{eq:bilip}
    \textstyle
    \BiLip(\gamma) := \inf_{\substack{x,y\in\R/\Z \\ x\ne y}} 
    \frac{|\gamma(y)-\gamma(x)|}{|x-y|_{\R/\Z}}.
\end{align*}

To 
apply the metric
existence result, \thref{Theorem_Existence_Metric_COMS}
of Section \ref{sec:2}, we 
formulate in the present section
a few general assumptions on the knot energy
$\EL$ so that the total energy $\phi$ in \eqref{eq:total-energy}
satisfies Assumptions
($\Phi$1)--($\Phi$5) of Section \ref{sec:2}. 
In order to verify ($\Phi$4), however,
we will need to consider a slightly smaller reflexive
Banach space $\CL$ that is compactly embedded in $W^{1+s,\rho}
(\R/\Z,\R^n)$ to finally prove the abstract existence
result of this section; see \thref{theorem:existenceCOMS}
below.

{\bf Assumptions on the knot energy $\EL$.}\,
Let $s\in (0,1)$ and $\rho\in (\frac1s,\infty)$.
\begin{enumerate}
\item[\rm (E1)]
$\EL$ is non-negative and of class $C^1$ on $W^{1+s,\rho}_{\ir}(\R/\Z,
\R^n)$
{\bf (differentiability)}.
\item[\rm (E2)] $\EL(\g)\le\liminf_{k\to\infty}\EL(\g_k)$ for all
$\g_k,\g\in W^{1+s,\rho}_\ir(\R/\Z,\R^n)$ with $\g_k\rightharpoonup \g$
as $k\to\infty$
{\bf (weak lower semi-continuity)}.
\item[\rm (E3)] For all $c_1,c_2>0$ there is a constant
$C=C(\EL,c_1,c_2)$ such that for all curves $\g\in C^1_{\mathrm i}(\R/\Z,\R^n)$
with $c_1^{-1}\le |\g'|\le c_1$ and $\EL(\g)\le c_2$ one has
$\BiLip(\g)\ge C$
{\bf (uniform control of bi-Lipschitz constant)}.
\end{enumerate}
Now we can state our abstract existence result for curves of maximal 
slope and solutions of the gradient flow equation.
\begin{theorem}
	        \label{theorem:existenceCOMS}
Let $s\in (0,1)$, $\rho\in (\frac1s,\infty)$ and $\theta,\kappa
\in (1,\infty)$. Suppose that $\CL\subset
W^{1+s,\rho}(\R/\Z,\R^n)$ is a compactly embedded reflexive
Banach space and that the energy $\EL$ satisfies Assumptions
{\rm (E1)--(E3)}, and $\g_0\in \sD(\phi)$, where
the total energy $\phi$ is defined as in \eqref{eq:total-energy}.
Then there exists $\bm u \in 
AC^{\theta}([0,\infty),(\CL,\Norm{\cdot}_{\CL}) )$ 
with $\bm u(0)=\gamma_0$ and 
			\begin{align}\label{eq:diff-inclusion}
 -D\phi[\bm u(t)] \in \FJ_{\CL,\theta}( \bm u'(t))\quad
 \textnormal{for a.e. $t>0$,}
                \end{align}
where $\FJ_{\CL,\theta}:\CL \mapsto 2^{\CL^*}$ 
is the $\theta$-duality mapping
defined in \eqref{eq:dual}. 
If, in addition, $\CL$ is strictly convex and has a 
G{\^a}teaux-differentiable norm on $\CL\setminus\{0\}$, then
          \begin{align}
\frac{\dif}{\dif t} \bm u(t) = 
-\FJ^{-1}_{\CL,\theta}\left( D\phi[\bm u(t)] \right)
\quad\textnormal{for a.e. $t>0$.}\label{EQN_h_8}
                \end{align}
   \end{theorem}
\begin{proof}
We proceed as follows. In the first four steps we verify
Assumptions (M), and ($\Phi$1)--($\Phi$5) of Section \ref{sec:2.1}
to obtain by virtue of the metric existence result, 
\thref{Theorem_Existence_Metric_COMS},
a curve of maximal slope for the total energy
$\phi$ with respect to its
local slope $|\partial\phi|$. Then in a final step we apply the more
specific results for Banach spaces presented  in Section \ref{sec:2.2}
to obtain the differential inclusion \eqref{eq:diff-inclusion}
and finally the gradient flow equation \eqref{EQN_h_8}.

{\bf Step 1: Assumptions (M), ($\Phi$2), ($\Phi$3) are satisfied.}\,
The Sobolev-Slobodecki\v{\i} space $W^{1+s,\rho}(\R/\Z,\R^n)$ as well
as the Banach space $\CL$ are complete metric spaces with the metric
induced by their respective norms. It follows from general Banach space
theory (see, e.g., \cite[Propositions 3.3 \& 3.5]{Brezis.2011})
that these norms are weakly lower semi-continuous, and that
the weak topology is Hausdorff, so that
Assumption (M) for the complete metric space
$\sS:=\CL$  
is satisfied (as well as for the larger space 
$W^{1+s,\rho}(\R/\Z,\R^n)$).
The total energy $\phi$ is coercive since $\EL$ and the norm of
$\Sigma$ are non-negative, which verifies Assumption ($\Phi$2).
Moreover, $W^{1+s,\rho}$ for $\rho>1$ is reflexive by
\thref{lem:A.1},
and $\CL$ is reflexive by assumption. So, the sequential
weak compactness of bounded sets follows from general theory again;
see, e.g., \cite[Theorem 3.17]{Brezis.2011}. In particular,
Assumption ($\Phi3$) holds here.

{\bf Step 2: Assumption  ($\Phi$1) holds true.}\,
To prove this claim we  first show that a sequence of bounded total
energy cannot have weak cluster points outside of the set of 
injective and regular curves. 
This way, the energy may not jump to infinity.
Afterwards, it is straight-forward to prove weak lower 
semi-continuity on $W^{1+s,\rho}_\ir(\R/\Z, \R^n)$.
        
Let $(\gamma_k)_{k\in\N} \subset W^{1+s,\rho}( \R/\Z,\R^n )$ be a 
sequence that converges weakly to a curve
$\gamma \in W^{1+s,\rho}(\R/\Z,\R^n)$
as $k\to\infty$,
and up to a further subsequence we may also assume
that
\begin{equation}\label{eq:28A}
\liminf_{k\to\infty} \phi(\gamma_k)<\infty,\quad\sup_{k\in\N}\phi(\g_k)<\infty,\quad\textnormal{and}\quad
\g_k\to\g \quad\textnormal{in $C^1$}, 
\end{equation}
where we used Arzela-Ascoli's theorem  in combination
with the embedding result, \thref{prop:embedding} (ii) 
in the appendix
for $k_1:=1$, $k_2:=1$,  $\rho_1:=\rho$, and any $\mu\in (0,s-\frac1{\rho})$, to obtain the $C^1$-convergence.
The finiteness of $\phi(\gamma_k)$ implies that 
$\gamma_k\in W^{1+s,\rho}_{\ir}( \R/\Z,\R^n )$ for all $k\in\N$ by
the very definition \eqref{eq:total-energy} of the total energy 
$\phi$, and further that
$
 \sup_{k\in\N} \Norm{\Sigma(\gamma_k)}_{W^{s,\rho}} <\infty.$
Because of the embedding $W^{s,\rho} \hookrightarrow C^0$ (see 
part (ii) of \thref{prop:embedding} for $k_1:=0,$ 
$\rho_1:=
\rho$, $k_2:=0$ and some $\mu\in (0,s-\frac1{\rho})$), 
we infer that $
\sup_{k\in\N} \sup_{x\in \R/\Z} |\log|\g_k'(x)|| <\infty$.
Consequently,  there exists a constant $c_1$ such that
\begin{equation}\label{EQN_1_2}
c_1^{-1}
  	\le |\g_k'(x)|\le c_1\quad\Foa x\in\R/\Z,
	\end{equation}
	which together with 
	the uniform total energy bound \eqref{eq:28A},
	in particular the resulting uniform 
	energy bound $c_2:=\sup_{k\in\N}\EL(\g_k)<\infty$,
	yields the uniform bound $C=C(\EL,c_1,c_2)$
	on the bi-Lipschitz constant
	by means of Assumption (E3):
$
  \inf_{k\in\N} \BiLip(\gamma_k) \ge C >0$.
By means of \eqref{EQN_1_2} and the convergence in 
\eqref{eq:28A}
we obtain that $\g$ is regular, and
for distinct parameters $x,y\in\R/\Z$
$$
|\g(x)-\g(y)|\overset{\eqref{eq:28A}}{=}
\lim_{k\to\infty}|\g_k(x)-\g_k(y)|\ge  
\inf_{k\in\N} \BiLip(\gamma_k)|x-y|_{\R/\Z}>0,
$$
so that $\g$ is injective, and therefore $\g\in W^{1+s,\rho}_{\ir}
(\R/\Z,\R^n)$. 

By Assumption (E2)  the energy $\EL$ is sequentially weakly
lower semi-continuous on $W^{1+s,\rho}_{\ir}(\R/\Z,\R^n)$. It
remains to prove the same for the regularising logarithmic strain
term, which also appears in the total energy $\phi$; cf.
\eqref{eq:total-energy}. Indeed, writing it out for $\g_k$,
\begin{align*}
\Norm{\Sigma(\gamma_k)}_{W^{s,\rho}}^{\kappa} = 
\textstyle
\paren[\Big]{\int_{\R/\Z} |\Sigma(\gamma_k)(x)|^{\rho} \;\dif x + \int_{\R/\Z} 
\int_{\R/\Z} \frac{|\Sigma(\gamma_k)(x)-\Sigma(\gamma_k)(y)|^\rho}{|x-y|^{1+s\rho}} 
\;\dif x \dif  y}^{\frac{\kappa}{\rho}},
		\end{align*}
		 we can use the $C^1$-convergence
		 in \eqref{eq:28A} to find pointwise
		convergence of  $\Sigma(\gamma_k)$  to $\Sigma(\gamma)$
		as $k\to\infty$, in order to apply 
		Fatou's Lemma to obtain
$
\liminf_{k\to\infty} \Norm{\Sigma(\gamma_k)}_{W^{s,\rho}}^{\kappa} 
\ge \Norm{\Sigma(\gamma)}_{W^{s,\rho}}^{\kappa}.
$
Since $\CL$ embeds continuously into $W^{1+s,\rho}( \R/\Z,\R^n)$, 
Assumption
($\Phi$1)  is fulfilled.

{\bf Step 3: Assumption ($\Phi$4) is satisfied.}\,
By definition of the total energy $\phi$ in \eqref{eq:total-energy} every sublevel set of $\phi$
is contained in $\CL_{\ir}$, which
according to \thref{lem:open}
in the appendix 
is an open subset of the full space $\CL$.
By \thref{Lemma_Log_Strain_is_Fréchet} and Assumption
(E1) the logarithmic strain $\Sigma$ and the energy $\EL$
are	 continuously Fr\'echet differentiable
on $W^{1+s,\rho}_{\ir}(\R/\Z,\R^n)$.
Furthermore, the choice $\kappa>1$ and 
\thref{lem:A.1} 
in the appendix for $k:=0$ imply that 
$\Norm{\cdot}_{W^{s,\rho}}^{\kappa}$ is also continuously
differentiable.
 Hence, the 
total energy $\phi$ is of class 
$C^1(W^{1+s,\rho}_{\ir}(\R/\Z,\R^n))$. 
Since the canonical embedding $\iota\colon\CL \to W^{1+s,\rho}(\R/\Z,\R^n)$ is linear and 
bounded and since $\CL_{\ir} \subset W^{1+s,\rho}_{\ir}(\R/\Z,\R^n)$, 
we infer that $\phi$ is of class $C^1(\CL_{\ir},\R)$.
It follows from \thref{Proposition_Metric_Quantities_in_Banach}
that the local slope takes the form 
$|\partial \phi|(\g) = \Norm{D\phi[\g]}_{\CL^{\ast}}$.
To be precise, denoting by $\iota^{\ast}$ the adjoint of $\iota$, 
we have 
$|\partial \phi|(\g) = \Norm{D(\phi\circ\iota)[\gamma]}_{\CL^{\ast}} = \Norm{\cdot}_{\CL^{\ast}} \circ \iota^{\ast} \circ D\phi \circ \iota[\gamma]$.
The map $\Norm{\cdot}_{\CL^{\ast}} \circ \iota^{\ast} \circ D\phi$ is continuous with respect to the norm-topology on 
 $W^{1+s,\rho}(\R/\Z,\R^n)$. 
Since $\CL$ embeds compactly, i.e.~since $\iota$ is a compact operator, effectively converting weakly 
convergent sequences in $\CL$
 into strongly convergent ones in $W^{1+s,\rho}(\R/\Z,\R^n)$, 
 we obtain the weak 
sequential continuity of $|\partial\phi|(\cdot)$ on $\CL_{\ir}$. 
In particular, it is also weakly sequentially lower 
semi-continuous, so that Assumption
($\Phi$4) is satisfied.

{\bf Step 4: Assumption  ($\Phi$5) holds.}\,
Let $\bm u \in AC(I,\CL)$ such that 
$\phi(\bm u(t))<\infty$ for all $t\in I$. 
By the previous step, the local slope $|\partial\phi|$
is weakly sequentially lower semi-continuous 
and therefore it is lower semicontinuous with respect to the 
strong topology. It follows that $|\partial\phi|\circ \bm u$ is 
lower semicontinuous and thus measurable. 
Since $\phi$ is continuously Fr{\'e}chet-differentiable
on $\CL_\ir$, it is locally Lipschitz-continuous, 
whence $\phi\circ \bm u$ is locally absolutely continuous,
and therefore differentiable a.e.~on $I$. Furthermore, since the
reflexive space $\CL$ satisfies the Radon-Nikodym property;
see \cite[Corollary 1.2.7]{Arendt.2011} $\bm{u}$ is also differentiable a.e.~on $I$.
Therefore, by the chain rule  and 
\thref{Proposition_Metric_Quantities_in_Banach}
\begin{align*}
\abs{\phi\circ \bm u(t)  & \textstyle
- \phi\circ\bm u (s)}
= 
\abs[\big]{\int_s^t \frac{\dif}{\dif r} \phi\circ \bm u(r) \;\dif r}
= \abs[\big]{\int_s^t \langle D\phi[\bm u(r)],\bm u'(r)
\rangle_{\CL^{\ast}\times \CL}\;\dif r}\\
&\textstyle
\le \int_s^t \Norm{D\phi[\bm u(r)]}_{\CL^{\ast}} \Norm{\bm u'(r)}_{
\CL} \;\dif r 
= \int_s^t |\partial \phi|(\bm u(r)) |\bm u'|(r) \;\dif r
		\end{align*}
for all $s,t\in I$ with $s<t$. This is the defining inequality
\eqref{EQN_h_1} for the strong upper gradient $g:=|\partial\phi|$.
Thus, Assumption ($\Phi$5) is verified.
	
{\bf Step 5: Proof of \eqref{eq:diff-inclusion} and
\eqref{EQN_h_8}.}\,
The previous steps have established the validity of Assumptions
(M), and ($\Phi$1)--($\Phi$5) so that \thref{Theorem_Existence_Metric_COMS} for the choice $\sS:=\CL$ 
yields a curve of
maximal slope for the total energy $\phi$ with 
respect to its local slope $|\partial\phi|$ starting
at $\g_0\in\sD(\phi)$. Since $\CL$ is reflexive it has the
Radon-Nikodym property \cite[Corollary 1.2.7]{Arendt.2011},
part (i) of \thref{Proposition_COMS_are_GFs}
implies the validity of the differential inclusion
\eqref{eq:diff-inclusion}. Under the stronger assumption that
$\CL$ is strictly convex and has a 
G{\^a}teaux-differentiable norm on $\CL\setminus\{0\}$, 
part (ii) of the same proposition yields the
gradient flow equation \eqref{EQN_h_8}.
\end{proof}

\begin{proposition}\label{prop:kappa1}
	Let $s\in (0,1)$, $\rho\in (\frac1s,\infty)$, $\theta \in(1,\infty)$, and $\kappa=1$. Suppose that $\CL\subset
	W^{1+s,\rho}(\R/\Z,\R^n)$ is a compactly embedded reflexive
	Banach space,  that the energy $\EL$ satisfies Assumptions
	{\rm (E1)--(E3)}, and $\g_0\in \sD(\phi)$, where
	the total energy $\phi$ is defined as in \eqref{eq:total-energy}.
	Then there exists a $\theta$-curve of maximal slope $\bm u \in 
	AC^{\theta}([0,\infty),(\CL,\Norm{\cdot}_{\CL}) )$ for $\phi$ with respect to its strong upper gradient $|\partial \phi|$ 
	starting at $\gamma_0$. 
	Furthermore, the local slope at $\g\in\CL$ takes the form
	\begin{align} \label{eq:local_slope_kappa=1}
				|\partial \phi|(\g) = \sup_{\substack{v\in\CL,\Norm{v}_{\CL}=1}} 
					\begin{cases}
						 \big(\langle D\EL[\g],v\rangle_{\CL^{\ast}\times \CL} - \Norm{D\Sigma[\g]v}_{\CA} \big)^{+}, \quad &\Sigma(\g)=0,\\[10pt]
						 \langle D\EL[\g],v \rangle_{\CL^{\ast}\times \CL} + \big\langle \frac{\FJ_{\CA,2}(\Sigma(\g))}{\Norm{\Sigma(\g)}_{\CA}},D\Sigma[\g]v\big\rangle_{\CA^{\ast}\times \CA}, \quad &\Sigma(\g) \ne 0,
					\end{cases}
	\end{align}
where $\CA:=W^{s,\rho}(\R/\Z)$ 
\end{proposition}
\begin{proof}	
	The assumptions (M) and $(\Phi1)$-$(\Phi3)$ follow exactly as in in the first two steps 
	of  \thref{theorem:existenceCOMS}.
	
	We now show that the local slope takes the  form 
	\eqref{eq:local_slope_kappa=1}. If $\Sigma(\g)\ne 0$, then $\phi = \EL + \Norm{\Sigma}_{\CA}$ is differentiable at $\g$ and we can apply \thref{Proposition_Metric_Quantities_in_Banach}. Notice that $D\Norm{\cdot}_{\CA}[\zeta] = \frac{\FJ_{\CA,2}(\zeta)}{\Norm{\zeta}_{\CA}}$ for every $\zeta\in \CA\setminus \{0\}$. Let us now assume that $\Sigma(\g)=0$. By differentiability, we have
	\begin{align*}
		&\phi(\eta) = \EL(\eta) + \Norm{\Sigma(\eta)}_{\CA}\notag\\
		& = \EL(\g) + \langle D\EL[\g],\eta-\g\rangle_{\CL^{\ast}\times \CL} + \Norm{\Sigma(\g) + D\Sigma[\g](\eta-\g)}_{\CA} 
		+ o\big( \Norm{\g-\eta}_{\CL} \big)\notag\\
		&= \phi(\g) + \langle D\EL[\g],\eta-\g\rangle_{\CL^{\ast}\times \CL} + \Norm{D\Sigma[\g](\eta-\g)}_{\CA} + o\big( \Norm{\g-\eta}_{\CL} \big)
	\end{align*}
as $\|\g-\eta\|_\CL \to 0$,	and therefore, 
	\begin{align*}
    \textstyle
	\frac{\big( \phi(\g) - \phi(\eta) \big)^{+}}{\Norm{\g-\eta}_{\CL}} &= \paren[\big]{\textstyle-\langle D\EL[\g],\frac{\eta-\g}{\Norm{\eta-\g}_{\CL}}\rangle_{\CL^{\ast}\times \CL} - \big\lVert D\Sigma[\g]\paren[\big]{\frac{\eta-\g}{\Norm{\eta-\g}_{\CL}}}\big\rVert_{\CA}}^{+} + o(1) \notag\\
    \textstyle
	&\le \sup_{\substack{v\in\CL,\Norm{v}_{\CL}=1}} \big(  \langle D\EL[\g],-v\rangle_{\CL^{\ast}\times \CL} - \Norm{D\Sigma[\g]v}_{\CA}  \big)^{+} + o(1)
	\end{align*}
    as $\|\g-\eta\|_\CL \to 0$.
	Choosing a maximising sequence $(v_n)_{n\in\N} \subset \CL$ 
	for the supremum and setting $\eta_n:= \g + \frac{v_n}{n}$, we infer that
	\begin{align*}
	\limsup_{\eta\to \g} \textstyle \frac{( \phi(\g) - \phi(\eta) )^{+}}{\Norm{\g-\eta}_{\CL}} &\ge \limsup_{n\to\infty} \big(\langle D\EL[\g],-v_n\rangle_{\CL^{\ast}\times \CL} - \Norm{D\Sigma[\g]v_n}_{\CA}  \big)^{+}\\
	&= \sup_{\substack{v\in\CL,\Norm{v}_{\CL}=1}} \big(  \langle D\EL[\g],-v\rangle_{\CL^{\ast}\times \CL} - \Norm{D\Sigma[\g]v}_{\CA}  \big)^{+}.
	\end{align*}
	Changing $v$ to $-v$ completes the representation formula 
	\eqref{eq:local_slope_kappa=1} for the local slope.
	
	Next, we prove that the local slope is sequentially weakly lower semicontinuous and thus satisfies Assumption $(\Phi4)$. Let $(\g_n)_{n\in\N} \subset \CL$ 
	that converges weakly to $\g\in\CL$. 
	Then, by compact embedding, $\g_n$ converges strongly 
	to $\g$ in $\mathcal{W}:=W^{1+s,\rho}(\R/\Z,\R^n)$. By Assumption (E1) and \thref{Lemma_Log_Strain_is_Fréchet} we have
	\begin{align} \label{eq:h:6}
	\lim_{n\to\infty} \paren[\big]{\Norm{\Sigma(\g_n) -\Sigma(\g)}_{\CA}  + \Norm{D\Sigma[\g_n]-D\Sigma[\g]}_{\LL(\WL,\CA)} + \Norm{D\EL[\g_n]-D\EL[\g]}_{\WL^{\ast}}}= 0.
	\end{align}
	We begin with the case where $\Sigma(\g)=0$. Let $v\in \CL$ with $\Norm{v}_{\CL}=1$. We have by \eqref{eq:h:6}
	\begin{align}
		\big(  \langle D\EL[\g],v\rangle_{\CL^{\ast}\times \CL} - \Norm{D\Sigma[\g]v}_{\CA}  \big)^{+} = \lim_{n\to\infty} \big(  \langle D\EL[\g_n],v\rangle_{\CL^{\ast}\times \CL} - \Norm{D\Sigma[\g_n]v}_{\CA}  \big)^{+}. \label{eq:h:1}
	\end{align}
	If $\Sigma(\g_n)=0$, it follows from \eqref{eq:local_slope_kappa=1} that
	\begin{align}
		\big(  \langle D\EL[\g_n],v\rangle_{\CL^{\ast}\times \CL} - \Norm{D\Sigma[\g_n]v}_{\CA}  \big)^{+} \le |\partial\phi|(\g_n). \label{eq:h:2}
	\end{align}
	If $\Sigma(\g_n)\ne 0$, we use the identity $\Norm{\FJ_{\CA,2}(\zeta)}_{\CA^{\ast}} = \Norm{\zeta}_{\CA}$ for all $\zeta\in \CA$ to 
	compute
	\begin{align}
		&\big(\langle D\EL[\g_n],v\rangle_{\CL^{\ast}\times \CL} -\Norm{D\Sigma[\g_n]v}_{\CA} \big)^{+}\nonumber\\
		&= 
		\textstyle
		\big(\langle D\EL[\g_n],v\rangle_{\CL^{\ast}\times \CL} - \frac{\Norm{\FJ_{\CA,2}(\Sigma(\g_n))}_{\CA^{\ast}}}{\Norm{\Sigma(\g_n)}_{\CA}}\Norm{D\Sigma[\g_n]v}_{\CA}\big)^{+}\nonumber \\
					&
					\textstyle
					\le \big(\langle D\EL[\g_n],v\rangle_{\CL^{\ast}\times \CL} + \langle \frac{\FJ_{\CA,2}(\Sigma(\g_n))}{\Norm{\Sigma(\g_n)}_{\CA}}, D\Sigma[\g_n]v\rangle_{\CA^{\ast}\times\CA}\big)^{+}
					\stackrelx{\smash{\eqref{eq:local_slope_kappa=1}}}{\le} |\partial \phi|(\g_n). \label{eq:h:3}
	\end{align}
	Combining \eqref{eq:h:1}, \eqref{eq:h:2} and \eqref{eq:h:3}, we obtain
	the inequality
$
        \textstyle
		(\langle D\EL[\g],v\rangle_{\CL^{\ast}\times \CL} - \Norm{D\Sigma[\g]v}_{\CA})^{+} \\ \le \liminf_{n\to\infty} |\partial\phi|(\g_n).
	$
	Taking the supremum over all $v\in \CL$ with $\Norm{v}_{\CL}=1$ on the left-hand side, we infer the lower semicontinuity.
	The case where $\Sigma(\g)\ne 0$ follows directly from the continuity properties \eqref{eq:h:6}, since $\Sigma(\g_n) \ne 0$ for all sufficiently large $n\in\N$. This proves that Assumption $(\Phi4)$ is satisfied.
	
Finally, we show that the local slope is a strong upper gradient, i.e.,
we verify Assumption~($\Phi \mathrm{5}$). 
Since $\EL$ and $\Sigma$ are continuously Fr{\'e}chet-differentiable, 
they are locally Lipschitz-continuous and so is $\phi$. 
In particular, $\phi\circ \bm{u}$ is absolutely continuous for every 
$\bm{u} \in AC\left( [a,b],\CL \right)$. 
It follows from \cite[Theorem 1.2.5]{AGS08},
\cite[Definition 1.2.2]{AGS08}, and \cite[Definition 1.2.1]{AGS08}
that $|\partial \phi|$ is a strong upper gradient for $\phi$. Therefore, Assumption $(\Phi5)$ is satisfied.
	The existence of a curve of maximal slope now follows as in step 5 of \thref{theorem:existenceCOMS}.
\end{proof}

\section{Gradient flows for various knot energies}
\label{sec:4}
To verify Assumptions (E1)--(E3) of \thref{theorem:existenceCOMS}  for the three energy families $E^{\alpha,p}$, $\intM^{(p,q)}$, and
$\TP^{(p,q)}$, the following two general results turn out 
to be useful. The first one yields sequential lower semi-continuity
for a multiple integral functional,
whereas the second guarantees uniform
control over the bi-Lipschitz constants.

\begin{lemma}
    \label{lemma:WeakSubcontinuity}
    Let $N \in \N $, $\rho \in (1, \infty)$, $s\in (\frac1{\rho},1)$,
     and
    \[
    \textstyle
\EL: W^{1+s, \rho}_{\ir}(\R/\Z, \R^n) \to [0, \infty], \,\,
\g \mapsto \idotsint_{(\R/\Z)^N} e(\g; x_1, \ldots, x_N) \d x_N \ldots \d x_1.
    \]
   Suppose  the integrand
    $e: W^{1+s, \rho}_{\ir}(\R/\Z, \R^n) \times (\R/\Z)^N \to [0, \infty]$ 
    satisfies
    \begin{enumerate}
        \item[\rm (e1)] $e(\g; x_1, \ldots, x_N) < \infty$ and
        \item[\rm (e2)] $C^1_{\ir}(\R/\Z, \R^n) \to [0, \infty), \,\,
	\g \mapsto e(\g; x_1, \ldots, x_N)$ is continuous
    \end{enumerate}
    for almost all $x_1, \ldots, x_N \in \R/\Z$.
    Then, the energy  $\EL$ satisfies Assumption {\rm (E2).}
\end{lemma}
\begin{proof}
Up to taking subsequences we may assume that $\g_k
\rightharpoonup\g$ in 
$W^{1+s,\rho}(\R/\Z,\R^n)$ and $\g_k\to\g$ in 
$C^1(\R/\Z,\R^n)$
as $k\to\infty$ (by the embedding result, \thref{prop:embedding}), and
that 
$
\lim_{k\to\infty}\EL(\g_k)=\liminf_{k\to\infty}\EL(\g_k)<\infty$.
The strong convergence in $C^1$ implies by virtue of
assumption (e2) the pointwise convergence
$
	         e(\g_k; x_1, \ldots, x_N) \to e(\g; x_1, \ldots, x_N) 
             $
 as $k\to\infty$ for a.e.~$x_1, \ldots, x_N \in \R/\Z$. 
    Now apply Fatou's lemma to conclude.
\end{proof}

\begin{lemma}
    \label{lemma:assumption2Arclength}
    Let $\EL: W^{1+s,\rho}(\R/\Z, \R^n) \to [0, \infty]$ be invariant 
    under reparametrisation and positively $d$-homogeneous
    for some $d\in\R$, and suppose
    that the following holds: 
    For every $M < \infty$, there is a constant $C=C(M,d,s,\rho) 
    > 0$ 
    such that for all $\g \in C^1_{\rm ia}(\R/\Z, \R^n)$ with 
    $\EL(\g) \le M$ one has $\BiLip\g\ge C$.
    Then, $\EL$ satisfies Assumption {\rm (E3).}
\end{lemma}
\begin{proof}
    Let $c_1, c_2 > 0$ and $\g\in C^1_{\ir}(\R/\Z, \R^n)$ such that 
    \[
        \label{lemma:assumption2Arclength:speed}
        \numberthis
        \textstyle
       c_1^{-1} \le \abs{\g'} \le c_1\quad \textnormal{on $\R/\Z$, \quad
       and}\quad
        \EL(\g) \le c_2.
    \]
    Then consider the arc length parametrisation 
    $\arclengthParam \in W_{\ia}^{1, \infty}(\R/\Z, \R^n)$ 
    of the rescaled curve $\tilde \g := \frac {\g} {\LL(\g)}$, 
    where $\LL(\g)>0$ denotes the  length of $\g$.
    Then $
        \EL(\arclengthParam)
        = \EL(\tilde \g)
        = \LL(\g)^{-d} \EL(\g)$.
    Using \eqref{lemma:assumption2Arclength:speed}, we may estimate the 
    first factor by $\tilde{c}_1 := \max \set{c_1^d, c_1^{-d}}$, 
    allowing us to set $M := \tilde{c}_1 c_2$ in order to
    use our assumption 
    to obtain a constant $b=b(c_1, c_2,d, s, \rho) > 0$ 
    independent of the specific curve $\g$ such that
    $\BiLip\Gamma\ge b$.
    Now, 
\cite[Lemma~B.1]{KnappmannSchumacherSteenebruggeEtAl:2021:AspeedpreservingHilbertgradientflowforgeneralizedintegralMengercurvature} in
combination with \eqref{lemma:assumption2Arclength:speed} yields that
$\BiLip\tilde{\g}\ge bc_1^{-1}$,
    and consequently 
    $$
    \textstyle
    \BiLip\g\ge bc_1^{-1}\LL(\g)\ge bc_1^{-2}=:\tilde{C}
    (c_1,c_2,d,s,\rho).
    $$
    Hence, Assumption (E3) is verified with the
    constant $C(\EL,c_1,c_2):=\tilde{C}$.
\end{proof}

With these ingredients we can now apply \thref{theorem:existenceCOMS}
to prove our central long-time existence result, \thref{thm:1.1},
for the total energy $\phi$ defined in \eqref{eq:total-energy}
for the knot energies $\EL\in\{E^{\alpha,p},\intM^{(p,q)},\TP^{(p,q)}\}
$
in the respective ranges of parameters $\alpha,p$ and $q$.

\begin{proof}[Proof of \thref{thm:1.1}]
In all three scenarios (i)--(iii)
the Banach space $\CB$ is chosen as
the energy space of the respective knot energy $\EL$, as
pointed out in \thref{rem:after-main-thm} of the introduction.
These Sobolev-Slobodecki\v{\i} spaces are all of the type
$W^{1+s,\rho}(\R/\Z,\R^n)$ for some $s\in (0,1)$ and $\rho\in
(\frac1s ,\infty)$ as considered in \thref{theorem:existenceCOMS}.
Moreover, the slightly smaller 
Sobolev-Slobodecki\v{\i} space $\CL$ is in each
of the cases (i)--(iii) a reflexive Banach space compactly 
embedded in $\CB$
according to 
\thref{lem:A.1} and \thref{prop:embedding}.
In addition, \thref{lem:A.1} for $k:=1$
guarantees that  the respective
norm $\|\cdot\|_{\CL}$ on the smaller Banach space $\CL\subset\CB$
is 
continuously Fr\'echet differentiable away from $0$,
which according to \thref{theorem:existenceCOMS}
yields that the differential inclusion
\eqref{eq:diff-inclusion} (if it holds at all)
reduces to the gradient flow equation
\eqref{EQN_h_8}, which corresponds to
\eqref{eq:gradflow} for each choice (i), (ii), and
(iii).  The $C^1$-regularity in time follows
from the continuity in time of the right-hand side
of \eqref{eq:gradflow}: The curve $\bm u$ is absolutely
continuous in time with values in $\CL_\ir$, and the
differential $D\phi$ is continuous on $\CL_\ir$, 
once Assumption (E1)  
is verified. 
In addition, \thref{lem:A.1} implies that the duality mapping 
$\FJ_{\CL,\theta}$ is
a homeomorphism from $\CL$ onto the dual space $\CL^*$.
To summarise, the right-hand side
of \eqref{eq:gradflow} is the composition of continuous mappings
and therefore continuous in time.

So, it remains to establish \eqref{eq:diff-inclusion}, and for that
it suffices to check that the respective knot energy $\EL$
satisfies Assumptions (E1)--(E3) in each case (i)--(iii). 

(i)\, O'Hara's knot energies $E^{\alpha,p}$ are non-negative
by definition, and they are continuously
differentiable on $W_{\ir}^{1+\frac{\alpha p-1}{2p},2p}(\R/\Z,\R^n)$
in the range of parameters $\alpha,p$ given in \thref{thm:1.1}.
This is shown in \thref{theorem:C1OHara} in Section
\ref{sec:5}, thus verifying 
Assumption (E1) for $s:=\frac{\alpha p-1}{2p}$
and $\rho:=2p$. 

The integrand of $E^{\alpha,p}$ in \eqref{eq:ohara}
obviously satisfies
the assumptions (e1) and (e2) of
\thref{lemma:WeakSubcontinuity} for $N=2$, $s:=\frac{\alpha p-1}{2p}$
and $\rho:=2p$ (which implies $s>\frac1\rho$ since $\alpha p >2$).
Therefore, Assumption (E2) is verified.

To verify Assumption (E3) for $E^{\alpha,p}$
assume first that $\g\in C^1_{\ia}(\R/\Z,\R^n)$ 
with $E^{\alpha,p}(\g)\\\le M$, which by means of
\cite[Theorem~1.1]{Blatt:2012:BoundednessandRegularizingEffectsofOHarasKnotEnergies}\footnote{Note 
again the unfortunate misprint in that reference.}
implies that $\g$ is of class $W^{1+\frac{\alpha p-1}{2p},2p}$. 
Then, we may approximate $\g$ with respect to the $\SobSpace$-norm by 
curves
$\g_k \in C^\infty_{\ia}(\R/\Z, \R^n)$ (see
\cite[Theorem 4.2]{lagemann-vdm_2022}),
 and use the continuity of $\Eap$ implied by 
the upcoming \thref{theorem:C1OHara} to assume without loss of generality 
that $\Eap(\g_k) \le \Eap(\g) + 1\le M+1$ for all $k\in\N$. 
Now, to these smooth arc length-parametrised approximants we can apply 
\cite[Theorem~2.3]{OHara:1992:Familyofenergyfunctionalsofknots} 
to obtain a uniform bi-Lipschitz constant, 
at least for the restricted parameter range 
$\alpha\in (\frac{\alpha}p,2]$. 
But even if $\a > 2$, we may use O'Hara's work.
    For curves $\eta \in C^1(\R/\Z,\R^n)$ parametrised by arc length, 
one has $\abs{\Delta \eta} \le \intrinsicDistance[\eta] \le \frac 1 2$. Consequently, both,
the integrand of $\Eap$ and the energy  $\Eap$ itself
are non-decreasing in $\a$.
    This means that if $\Eap(\eta) \le M + 1$ for $\a \ge 2$, 
so is $\Eap[2,p]$ and we may use the bi-Lipschitz constant obtained for that energy.    
    O'Hara's bi-Lipschitz estimate yields a constant $K = K(\a, p, M+1)$ independent of 
$\g$ such that
    \[
        \textstyle
        K \abs{x-y}_{\R/\Z} \le \abs{\g_k(x)-\g_k(y)} \yrightarrow{k \to \infty} \abs{\g(x)-\g(y)}.
    \]
    This, together with the energy's invariance under reparametrisation and its
 positive $(2-\a p)$-homogeneity  implies that $\Eap$ satisfies the suppositions 
of \thref{lemma:assumption2Arclength} and thus also Assumption (E3).

(ii)\,
Assumption (E1) is satisfied by means of
\cite[Theorem~3]{BlattReiter:2015:TowardsaregularitytheoryforintegralMengercurvature}.
The integrand of $\intM^{(p,q)}$,
    \[
\textstyle e(\g; x_1, x_2, x_3) 
:= \frac {\abs{(\g(x_1) - \g(x_2)) \wedge 
(\g(x_3) - \g(x_2))}^q} {\abs{\g(x_1) - \g(x_2)}^p 
\abs{\g(x_2) - \g(x_3)}^p \abs{\g(x_1) - \g(x_3)}^p} 
\abs{\g'(x_1)} \abs{\g'(x_2)} \abs{\g'(x_3)}
    \]
    is non-negative and continuous in $\g$ with respect to the 
    $C^1$-topology for all pairwise distinct $x_1, x_2, x_3 \in \R/\Z$.
    Thus, Assumption (E2) is verified for $\intM^{(p,q)}$
    by means of
    \thref{lemma:WeakSubcontinuity} for $N:=3$, $\rho:=
    q$ and $s=\frac{3p-2}q-2.$

The energy $\intM^{(p,q)}$
is invariant under reparametrisations and
positively $(3+2q-3p)$-homogeneous, which we can combine
with 
\cite[Proposition~2.1]{BlattReiter:2015:TowardsaregularitytheoryforintegralMengercurvature} and
\thref{lemma:assumption2Arclength} to conclude that Assumption
(E3) holds as well.

(iii)\,
    The $C^1$-regularity of the tangent-point energies
    $\TP^{(p,q)}$ was first stated in 
    \cite[Remark~3.1]{BlattReiter:2015:Regularitytheoryfortangentpointenergiesthenondegeneratesubcriticalcase}, for a 
    proof see 
    \cite[Satz~7.4]{Wings:2018:StetigeDifferenzierbarkeittangentenpunktartigerKnotenenergien} or 
    \cite{steenebruegge_2022},
    so, Assumption (E1) holds for $\TP^{(p,q)}$

Applying \thref{lemma:WeakSubcontinuity} with 
$N:=2$, $\rho:=q$, and $s:=\frac {p-1} {q}-1$ 
to the integrand of $\TP^{(p,q)}$,
$
        e(\g;x_1,x_2) :=
{\abs{P^\perp_{\g'(x_1)}(\g(x_2) - \g(x_1))}^q}
{\abs{\g(x_2) - \g(x_1)}^{-p}} \abs{\g'(x_2)} \abs{\g'(x_1)}$
    we infer Assumption (E2) for $\TP^{(p,q)}$.

Finally, $\TP^{(p,q)}$ is invariant under reparametrisations
and positively $(q-p+2)$-homogeneous. Therefore, we can
combine
\cite[Proposition~2.7]{BlattReiter:2015:Regularitytheoryfortangentpointenergiesthenondegeneratesubcriticalcase}
with \thref{lemma:assumption2Arclength} for 
$s:=\frac{p-1}q-1$ 
and $\rho:=q$
to conclude that Assumption (E3) is satisfied for $\TP^{(p,q)}$
as well.
\end{proof}
 
\begin{proof}[Proof of \thref{cor:1.2}]
For simplicity we set $\CL:=\CL_\eps$.
    Note that for our choice of $\CL$, \eqref{eq:gradflow} is in fact well-defined, as by \thref{lem:A.1} the duality mapping $\FJ_{\CL, \theta}$ is a homeomorphism between $\CL$ and $\CL^*$ (and in particular single-valued).
        Combining the continuity of $\FJ_{\CL, \theta}^{-1}$ with the continuity of $\bm u$ and $D\phi$ (see the proof of \thref{thm:1.1} for the latter)
    we obtain that $\bm u'$ is almost everywhere equal to the continuous function $\bm v: [0, \infty) \to \CL, t \mapsto - \FJ_{\CL, \theta}^{-1}(D\phi[\bm u(t)])$.
    It remains to establish that $\bm u'(t)$ exists for all $t$ and is equal to $\bm v(t)$.
    Since $\bm u$ is absolutely continuous on compact intervals and $\CL$ is reflexive, we have 
    the  fundamental theorem of calculus \cite[Corollary~1.2.7, Definition~1.2.5, Proposition~1.2.3]{Arendt.2011}, so that
    $
        \textstyle \frac 1 h (\bm u(t + h) - \bm u(t)) = \frac 1 h \int_t^{t+h} \bm u'(r) \d r = \frac 1 h \int_t^{t+h} \bm v(r) \d r.
  $
    As $\bm v$ is continuous, we can immediately infer that the right-hand side converges to $\bm v(t)$ as $h \to 0$ and hence $\bm u$ is of class $C^1([0,\infty),\CL)$.

We continue by proving that the energy is non-increasing along the flow.
According to \thref{lem:A.1} the Banach space $\CL$ is reflexive in each of the
three cases (i)--(iii) of  \thref{thm:1.1}. The $\theta$-duality mapping
$\FJ_{\CL,\theta}:\CL\to 2^{\CL^*}$ is a duality map \emph{with weight}
$\varphi(s):=s^{\frac{\theta}{\beta}}$ in the language of
\cite[Definition~I.4.1]{Cioranescu:1990:GeometryofBanachSpacesDualityMappingsandNonlinearProblems}.
Therefore, we may apply \cite[Corollary~II.3.5]{Cioranescu:1990:GeometryofBanachSpacesDualityMappingsandNonlinearProblems}
to identify the inverse $\FJ^{-1}_{\CL,\theta}$ with the duality mapping on $\CL^*$
with weight $\varphi^{-1}(s)=s^{\frac{\beta}{\theta}}$, which is just
$\FJ_{\CL^*,\beta}$.
Thus, by the chain rule and \eqref{eq:gradflow},
\begin{align*}
\textstyle
\frac {\d} {\d t} \phi(\bm u(t)) & = \langle  D\phi[\bm u(t)],\bm u'(t)\rangle_{\CL^*\times \CL}
\smash{\overset{\eqref{eq:gradflow}}{=}}
-\langle  D \phi [\bm u(t)], \FJ_{\CL, \theta}^{-1}(D \phi[\bm u(t)]) \rangle_{\CL^*\times \CL}
\notag\\
\textstyle
&
=-\langle  D \phi [\bm u(t)], \FJ_{\CL^*, \beta}(D \phi[\bm u(t)])
\rangle_{\CL^{*}\times \CL^{**}}
=-\|D\phi[\bm u(t)]\|_{\CL^*}^\beta\le 0
\end{align*}
for all $t\ge 0$.
Hence $\phi\circ \bm u$ is non-increasing, so that \cite[Proposition~1.4.1]{AGS08}
implies that $\bm u$ is a $\theta$-curve of maximal slope with respect to the weak upper gradient\footnote{As discussed in the beginning of Section \ref{sec:2.1}, in our setting our notion of a
curve of maximal slope is equivalent to the one presented in \cite{AGS08}.}
$\norm[\CL^*]{D\phi[\bm u(t)]}$.
In fact, by Step 4 of \thref{theorem:existenceCOMS} whose prerequisites were verified in  the proof of \thref{thm:1.1}, it is even a strong upper gradient.
Finally, $\CL$ continuously embeds into $C^1(\R/\Z,\R^n)$ in all three cases 
of \thref{thm:1.1}, and thus $\bm u$ is a $C^1$-isotopy.
    All curves $\bm u(t)$ are embedded because $\phi (\bm u(t))<\infty$ for all
    $t\ge 0$  and so, by \cite{reiter_2005,blatt_2009a}, $[\bm u(t)] = [\g_0]$
    for all $t\ge 0$, i.e.~the knot class is preserved along the flow.
\end{proof}

The following corollary shows that the gradient flows obtained in \thref{thm:1.1} 
for each $\eps >0$ admit a converging subsequence as $\epsilon \to 0$.
\begin{corollary} \label{cor:eps_to_0}
	Let $\theta, \kappa \in (1,\infty)$, $\phi$ and the spaces  $\CB$ and
	$\CL_\epsilon$ for $\epsilon>0$  be as in \thref{thm:1.1}, and $\gamma_0 \in \CB_{\ir}$. Then for any sequence $(\gamma_{0,\epsilon})_\epsilon\subset
	\CL_\epsilon$ with
	\begin{align}
    \textstyle
		\gamma_{0,\epsilon} \xrightharpoonup[(\epsilon\to 0)]{} \gamma_0  \text{ in } \CB, \quad \text{ and } \quad \phi(\gamma_{0,\epsilon}) \xrightarrow[(\epsilon\to 0)]{} \phi(\gamma_0)\label{anfangskonvergenz_cor}
	\end{align}
	and for  corresponding solutions $\bm{u}_{\epsilon} \in C^1( [0,\infty),(\CL_{\epsilon}, \Norm{\cdot}_{\CL_{\epsilon}}) )$ 
 of the gradient flow equation \eqref{eq:gradflow} with $\bm{u}_{\epsilon}(0)=\gamma_{0,\epsilon}$,  there exists a subsequence $\epsilon_k \to 0$ and a curve $\bm{u}^{\ast} \in AC^{\theta}\left( [0,\infty),(\CB,\Norm{\cdot}_{\CB}) \right)$ such that $\bm{u}^{\ast}(0)=\gamma_0$,	 
	$\bm{u}_{\epsilon_k}(t) \xrightharpoonup[(k\to \infty)]{} \bm{u}^{\ast}(t)$ in $\CB$ and $\phi(\bm{u}^{\ast}(t))\le \phi(\gamma_0)$ for all $t\ge 0$.
\end{corollary}

\begin{remark}
	Since $\phi \in C^0\paren[\big]{W^{1+s,\rho}_{\ir}\paren{\R/\Z,\R^n}}$, we can choose convolutions $\gamma_{0,\epsilon} \in C^{\infty}\left( \R/\Z,\R^n \right)$ with $\Norm{\gamma_0-\gamma_{0,\epsilon}}_{W^{1+s,\rho}\left( \R/\Z,\R^n \right)} <\epsilon$ to secure Assumption \eqref{anfangskonvergenz_cor}. Furthermore, since the weak convergence in $\CB= W^{1+s,\rho}\left( \R/\Z,\R^n \right)$ implies strong convergence in $C^1\left( \R/\Z,\R^n \right)$ and since $C^1_{\ir} \subset C^1$ is an open subset, see \thref{lem:open}, it follows from $\eqref{anfangskonvergenz_cor}$ and $\g_0 \in \CB_{\mathrm{ir}} \subset C^{1}_{\mathrm{ir}}$ that for all sufficiently small $\eps>0$ the curves $\gamma_{0,\eps}$ are injective and regular. Therefore, we may use \thref{thm:1.1} to secure the existence of solutions of \eqref{eq:gradflow} starting at $\gamma_{0,\eps}$. 
\end{remark}

\begin{proof}[Proof of \thref{cor:eps_to_0}]
	The claim is a consequence of \thref{prop:eps_to_0}, whose prerequisites we now check. As was shown in step 1 of the proof of \thref{theorem:existenceCOMS}, $(\CB,\Norm{\cdot}_{\CB})$ with its weak topology satisfies Assumption (M). Moreover, for all $\gamma \in W^{1+s+\epsilon}\left( \R/\Z,\R^n \right)$ we have 
	$
		\Norm{\gamma}_{W^{1+s,\rho}\left( \R/\Z,\R^n \right)} \le \Norm{\gamma}_{W^{1+s+\epsilon, \rho}\left( \R/\Z,\R^n \right)}.
	$ 
	In addition, as was shown in the proof of \thref{thm:1.1}, $\phi$ satisfies the Assumptions {\rm (E1)--(E3)} und thus, by steps 1 and 2 of the proof of \thref{theorem:existenceCOMS}, it also satisfies the Assumptions {\rm ($\Phi$1)--($\Phi$3)} with $(\CB,\Norm{\cdot}_{\CB})$ as the underlying metric space. Furthermore, since we are only interested in small $\epsilon>0$, in view of \eqref{anfangskonvergenz_cor} we may assume that $\sup_{\epsilon>0} \phi(\gamma_{0,\epsilon})<\infty$ and $\sup_{\epsilon>0} \Norm{\gamma_{0,\epsilon}}_{\CB} <\infty$. 
   Moreover, by \thref{cor:1.2}, $\bm{u}_{\eps}$ is a $\theta$-curve of maximal slope for $\phi$ with respect to the strong upper gradient $|\partial\phi|_{\CL_{\eps}} = \Norm{D\phi}_{\CL^{\ast}_{\eps}}$ starting at $\gamma_{0,\epsilon}$.
	
	Identifying  $(\sS_0,d_0)$ with $(\CB,\Norm{\cdot}_{\CB})$ let $\sigma$ be
 the weak topology on $\CB$. Furthermore, for $\epsilon>0$ let $(\sS_{\epsilon},d_{\epsilon}):=(\CL_{\epsilon},\Norm{\cdot}_{\CL_{\epsilon}})$, $g_{\epsilon}:=|\partial \phi|_{\CL_{\epsilon}}$, and $u_{0,\epsilon}:=\gamma_{0,\epsilon}$. Then these satisfy all the assumptions of \thref{prop:eps_to_0} with $c_0=1$ which concludes the proof. Notice that since $\phi(\bm{u}^{\ast}(t))\ge 0$ and $\phi(\gamma_0)<\infty$, it follows that $\bm{u}^{\ast}$ is actually absolutely continuous and not only locally absolutely continuous.
\end{proof}

\section{O'Hara's knot energies are continuously
differentiable}
\label{sec:5}
Let us rewrite $E^{\alpha,p}$ as
\begin{equation}\label{eq:ohara-integrand}
    \textstyle
    \Eap(\g) := \iint_{(\R/\Z)^2} \eap(\g; x, y) \abs{\g'(x)} \abs{\g'(y)} \d x \d y, \quad \ea(\g; x, y) = \frac {1} {\abs{\Delta \g}^\a} - \frac {1} {\intrinsicDistance^\a},
\end{equation}
where $\Delta \g:= \Delta \g(x,y) := \g(x) - \g(y)$ and 
$\intrinsicDistance := \intrinsicDistance(x,y)$ is the 
intrinsic distance between $\g(x)$ and $\g(y)$, 
i.e.\ the length of the shortest arc of $\g$ connecting 
$\g(x)$ and $\g(y)$.
Note that O'Hara defined the energy for curves parametrised 
by arc length only and had 
$\abs{x-y}_{\R/\Z}^{-\a}$ as the second term.
Using $\intrinsicDistance^{-\a}$ instead
is a sensible generalisation as the 
energy is then invariant under reparametrisations, 
cf.\  \cite[Remark~4.1.1(6)]{ohara_2003}.

Our method of proof for continuous differentiability of
$E^{\a,p}$ is inspired by
\cite[Section~3]{ReiterSchumacher:2021:SobolevGradientsfortheMobiusEnergy}.
There are several arguments which carry over completely; we include 
these in our proof for the reader's convenience.
  Several technical results needed in the proof
  of \thref{theorem:C1OHara} below, namely Lemmata
  \ref{lemma:geometricNormC1}--\ref{lemma:mainTechnicalEstimateOHara}
  are deferred to the end of this section.

\begin{theorem}
    \label{theorem:C1OHara}
    Let $p \ge 1$, $\a > 0$, $2 < \a p < 2p + 1$.
    Then, $\Eap$ is continuously Fréchet-differentiable on $\SobSpaceir(\R/\Z, \R^n)$.
\end{theorem}
In the following, we will use the \newterm{derivative with respect to arc length} $D_\g \eta (u) := \frac {\eta'(u)} {\abs{\g'(u)}}$ and the interval $I_\g(x,y)$ parametrising the arc of $\g$ where $\intrinsicDistance(x,y)$ is attained.
To be more precise, $I_\g(x,y)$ is the interval containing $x$ and one of $y-1, y, y+1$ such that 
$
    \intrinsicDistance(x,y) = \int_{I_\g(x,y)} \abs{\g'(t)} \d t$.
This is well-defined whenever $\intrinsicDistance(x,y) < \frac L 2$, i.e.\ for almost all $x,y \in \R$.
When there is no risk of confusion, we omit the arguments $x$ and $y$.
Lastly, we need the \newterm{minimal velocity} $v_\g := \essinf_{x \in \R/\Z} \abs{\g'(x)}$.
\begin{proof}[Proof of \thref{theorem:C1OHara}]
    For $p=1$, this was already proved in \cite[Proposition~2.1]{BlattReiter:2013:StationarypointsofOHarasknotenergies}, so we may assume $p>1$.

    We use the chain rule (see e.g.\ \cite[Proposition~4.10]{Zeidler:1993:NonlinearfunctionalanalysisanditsapplicationsFixedpointtheorems} for the general Banach space version) to prove our statement.
    As outer function, we choose a geometric $L^p$-norm additionally depending on $\g$:
   \begin{align*} 
       \norm[L^p_\cdot((\R/\Z)^2)]{\cdot}^p \colon &
       \SobSpaceir(\R/\Z, \R^n) \times L^p((\R/\Z)^2) \to \R,\\
       &(\g,g) \mapsto \textstyle \iint_{(\R/\Z)^2} \abs{g(x,y)}^p \abs{\g'(x)} \abs{\g'(y)} \d y \d x .
   \end{align*} 
    We consider the integrand $\ea$ 
    in \eqref{eq:ohara-integrand}
    as a mapping from $\smash{\SobSpaceir}$ to $L^p((\R/\Z)^2)$ for the inner function.
    Then, $\Eap(\g) = \Lpg{\cdot}^p \circ \ea (\g)$ and it suffices to show that both functions are $C^1$.
    \thref{lemma:geometricNormC1} together with the embedding
    result, part (ii) of \thref{prop:embedding},
    take care of the outer function, so we only need to look at $\ea$.
    In order to do this, define
    \begin{align*}
        F_k(\g; \eta_1, \ldots, \eta_k) &:= \delta^k(\g \mapsto \ea(\g)(x,y))(\eta_1, \ldots, \eta_k),\\
        G_k(\g; \eta_1, \ldots, \eta_k) &:= \textstyle\int_{I_\g}\delta^k(\g \mapsto \abs{\g'(t)})(\eta_1, \ldots, \eta_k)  \d t.
    \end{align*}
    Here, $\delta^k$ denotes the \newterm{$k$-th variation}.
    As for all $x \in \R/\Z$ there exists exactly one $y \ne x \in \R/\Z$ such that $I_\g(x,y)$ is not well-defined, we work on
    \[
        \Sigma := \bigl\{(x,y) \in (\R/\Z)^2 \ \vert \ x \ne y \text{ and } \intrinsicDistance(x,y) < \tfrac {\LL(\g)} {2} \bigr\}
    \]
    which is open (since $\g \in \smash{\SobSpace} \hookrightarrow C^1$) and is the same as $(\R/\Z)^2$ up to a set of measure $0$.
    For all $\g \in \SobSpaceir$ and $(x,y) \in \Sigma$, there is an open neighbourhood $U(x,y) \subseteq \SobSpaceir(\R/\Z, \R^n)$ of $\g$ such that $\eta \mapsto I_\eta(x,y)$ is constant on $U(x,y)$.
    This means that the $G_k$ exist for all such $(x,y)$ and a simple calculation shows that
    \begin{align*}
        G_1(\g; \eta_1) &= \textstyle\int_{I_\g} \langle D_\g \g, D_\g \eta_1 \rangle \abs{\g'} \d \tau,\\
        G_2(\g; \eta_1, \eta_2) &= \textstyle\int_{I_\g} (\langle D_\g \eta_1, D_\g \eta_2\rangle - \langle D_\g \g, D_\g \eta_1 \rangle \langle D_\g\g, D_\g \eta_2 \rangle) \abs{\g'} \d t.
    \end{align*}
    To shorten notation, we left out the $t$-dependencies in these terms and will also do this in the following when there is no risk of confusion.
    Recall that $\intrinsicDistance(x,y) = 
    \int_{I_\g} \abs{\g'(t)} \d t$.
    For fixed $\eta$ and $\abs{\tau} < 1$ sufficiently small, the estimate $v_{\g + \tau \eta} \ge \frac 1 2 v_\g$ holds and so
    \begin{align*}
        \abs{D_{\g + \tau \eta} \tilde \eta (t)}
        &=\textstyle \abs*{\frac {\tilde \eta'(t)} {\abs{\g'(t) + \tau \eta'(t)}}}
        \le \abs{D_{\g} \tilde \eta} \abs*{\frac {\g'(t)} {\abs{\g'(t) + \tau \eta'}}}
        \le \abs{D_{\g} \tilde \eta} \norm[L^\infty]{\g'} \frac 2 {v_\g}.
    \end{align*}
    This, together with \thref{lemma:boundGeometricDerivative}, enables us to find integrable majorants for the $\tau$-derivatives of $\abs{\g'_\tau}:=\abs{\g'(t) + \tau \eta'(t)}$: 
    They only consist of sums of inner products of $D_{\g_\tau} \g_\tau$, $D_{\g_\tau} \eta$  and $\abs{\g_\tau'}$ (whose derivatives again fit the pattern).
    Consequently, $\delta^k d_\g = G_k(\g)$ and we may calculate for $(x,y) \in \Sigma$, that 
    \[
        \textstyle
        F_1(\g; \eta_1) = \a \paren[\Big]{\frac 1 {\intrinsicDistance^{\a + 1}} G_1(\g; \eta_1) - \frac 1 {\abs{\Delta \g}^{\a + 2}} \langle \Delta \g, \Delta \eta_1 \rangle}
    \]
    and 
    \begin{align*}
        F_2(\g; \eta_1, \eta_2) = & \textstyle
        \a (\a + 2) \frac 1 {\abs{\Delta\g}^{\a + 4}}\langle \Delta \g, \Delta \eta_1\rangle \langle \Delta \g, \Delta \eta_2\rangle - \a \frac {\langle \Delta \eta_1, \Delta \eta_2\rangle} {\abs{\Delta \g}^{\a + 2}}\\
        &\textstyle - \a (\a + 1) \frac 1 {\intrinsicDistance^{\a + 2}} G_1(\g; \eta_1) G_1(\g; \eta_2) + \a \frac 1 {\intrinsicDistance^{\a + 1}} G_2(\g; \eta_1, \eta_2).
    \end{align*}
    Up until now, these are only pointwise limits and we still need to show that $F_1$ is indeed the Fréchet-derivative of $\ea$ and also continuous with respect to $\g$.
    Let us first show that $F_1$ is a valid candidate for a derivative.
    \begin{claim}
        \label{theorem:C1OHara:claim:boundFirstVariation}
        There is $C=C(\g)>0$ which continuously depends on $\g$ such that $\norm[L^p]{F_1(\g; \eta)}\\\le C \norm[\SobSpace]{\eta}$ for all $\eta \in \SobSpace(\R/\Z, \R^n)$.
    \end{claim}
    In order to show this, decompose $F_1$ as
    \[
        \textstyle
        \frac {\a} {\intrinsicDistance^{\a + 2}} \paren{ \intrinsicDistance G_1(\g;\eta_1) - \langle\Delta \g, \Delta \eta_1 \rangle } - \a \paren[\Big]{\frac 1 {\abs{\Delta \g}^{\a + 2}} - \frac 1 {\intrinsicDistance^{\a + 2}}} \langle \Delta \g, \Delta \eta_1\rangle.
    \]
    \thref{lemma:boundedDifference} gives us a fitting upper bound for the first term, \thref{lemma:multilinearIntegral} gives one for the second term when we choose $\varphi = 0$, $\psi = \a$, $\tilde \eta_1 := \g$ and $\tilde \eta_2 := \eta_1$ (note that $L_1 = L_2 = \frac {\Delta} {\intrinsicDistance}$). 
    
    We will later need a similar bound for $F_2$.
    \begin{claim}
        \label{theorem:C1OHara:claim:boundSecondVariation}
        There is $\Xi = \Xi(\g) > 0$ depending continuously on $\g$ such that $\norm[L^p]{F_2(\g; \eta_1, \eta_2)} \\\le \Xi \norm[\SobSpace]{\eta_1} \norm[\SobSpace]{\eta_2}$.
    \end{claim}
    The central ingredient is once again the right decomposition.
    Rewrite $F_2$ as
    \begin{align}
        \label{theorem:C1OHara:eq:F21}
        &\textstyle\a (\a + 2) \langle\Delta \g, \Delta \eta_1\rangle \langle\Delta \g, \Delta \eta_2\rangle \paren[\Big]{\frac 1 {\abs{\Delta \g}^{\a + 4}} - \frac 1 {\intrinsicDistance^{\a + 4}}}\\
        \label{theorem:C1OHara:eq:F22}
        &\textstyle- \a \langle \Delta \eta_1, \Delta \eta_2 \rangle \paren[\Big]{\frac 1 {\abs{\Delta \g}^{\a + 2}} - \frac 1 {\intrinsicDistance^{\a + 2}}} \\
        \label{theorem:C1OHara:eq:F23}
        &\textstyle+ \frac {\a} {\intrinsicDistance^{\a + 2}} \paren*{G_1(\g; \eta_1) G_1(\g; \eta_2) + \intrinsicDistance G_2(\g; \eta_1, \eta_2) - \langle \Delta \eta_1, \Delta \eta_2\rangle}\\
        \label{theorem:C1OHara:eq:F24}
        &\textstyle+ \a (\a + 2) \frac 1 {\intrinsicDistance^{\a + 4}} \paren{\langle \Delta \g, \Delta \eta_1\rangle \langle \Delta \g, \Delta \eta_2\rangle - \intrinsicDistance G_1(\g; \eta_1) \intrinsicDistance G_1(\g; \eta_2)}.
    \end{align}
    We can use \thref{lemma:multilinearIntegral} with $\psi = \a + 2$, $\varphi = 0$, $\tilde \eta_1 = \tilde \eta_3 = \g$, $\tilde \eta_2 = \eta_1$ and $\tilde \eta_4 = \eta_2$ to deal with \eqref{theorem:C1OHara:eq:F21} and the same Lemma with $\psi = \a$, $\varphi = 0$, $\tilde \eta_1 =\eta_1$ and $\tilde \eta_2 = \eta_2$ to find an upper bound for \eqref{theorem:C1OHara:eq:F22}.
    In order to take care of \eqref{theorem:C1OHara:eq:F23}, let us define $H := \int_{I_\g} \langle D_\g \eta_1, D_\g\eta_2\rangle \abs{\g'} \d t$.
    Then,
    $
        \eqref{theorem:C1OHara:eq:F23} 
	\textstyle
        = \frac {\a} {\intrinsicDistance^{\a + 2}} \paren*{\intrinsicDistance H - \langle \Delta \eta_1, \Delta \eta_2\rangle} + \frac {\a} {\intrinsicDistance^{\a+2}} \paren[\big]{G_1(\g;\eta_1)G_1(\g;\eta_2) - \intrinsicDistance\paren{H - G_2(\g; \eta_1, \eta_2)}}.
    $
    The first term is again bounded above via \thref{lemma:boundedDifference}, so let us take a look at the second one.
    Define $\varphi_i := \langle D_\g \g, D_\g \eta_i \rangle$.
    Then, $H - G_2(\g; \eta_1, \eta_2) = \int_{I_\g} \varphi_1 \varphi_2 \abs{\g'} \d t$ and thus the second term is equal to
    \begin{align*}
        &\textstyle\frac {\a} {\intrinsicDistance^{\a+2}} \paren[\big]{\int_{I_\g} \varphi_1(s) \abs{\g'(s)} \d s \int_{I_\g} \varphi_2(t) \abs{\g'(t)} \d t - \iint_{I_\g^2} \varphi_1(t)\varphi_2(t) \abs{\g'(s)} \abs{\g'(t)} \d t \d s}\\
        =&\textstyle \frac {\a} {2 \intrinsicDistance^{\a+2}} \bigl(\iint_{I_\g^2} (\varphi_1(s) - \varphi_1(t))\varphi_2(t) \abs{\g'(s)} \abs{\g'(t)} \d t \d s \\
        &\textstyle + \iint_{I_\g^2} (\varphi_1(t) - \varphi_1(s))\varphi_2(s) \abs{\g'(s)} \abs{\g'(t)} \d t \d s\bigr)\\
        =&\textstyle - \frac {\a} {2 \intrinsicDistance^{\a+2}} \iint_{I_\g^2} \Delta \varphi_1(s,t) \Delta \varphi_2(s,t) \abs{\g'(s)} \abs{\g'(t)} \d t \d s
    \end{align*}
    If we can estimate this by terms controlled via \thref{lemma:mainTechnicalEstimateOHara}, we have an upper bound for the $L^p$-norm of \eqref{theorem:C1OHara:eq:F23}.
    To achieve this, first apply the Hölder inequality (with $p=q=2$) and our usual upper bound for the line elements to reduce the problem to bounding
    \begin{align*}
        &\textstyle\frac 1 {\intrinsicDistance^{\a + 2}} \iint_{I_\g^2} \abs{\Delta \varphi_i(s,t)}^2 \d t \d s\\
        =&\textstyle \frac 1 {\intrinsicDistance^{\a + 2}} \iint_{I_\g^2} \abs{\langle D_\g \g(t), D_\g \eta_i(t)\rangle - \langle D_\g \g(s), D_\g \eta_i(s)\rangle}^2 \d t \d s\\
        =&\textstyle \frac 1 {\intrinsicDistance^{\a + 2}} \iint_{I_\g^2} \abs{\langle D_\g \g(t), D_\g \eta_i(t) - D_\g \eta_i(s)\rangle + \langle D_\g \g(t) - D_\g \g(s), D_\g \eta_i(s)\rangle}^2 \d t \d s\\
        \le&\textstyle \frac 2 {\intrinsicDistance^{\a + 2}} \iint_{I_\g^2} \abs{\langle D_\g \g(t), D_\g \eta_i(t) - D_\g \eta_i(s)\rangle}^2 + \abs{\langle D_\g \g(t) - D_\g \g(s), D_\g \eta_i(s)\rangle}^2 \d t \d s\\
        \le&\textstyle 2 \norm[L^\infty]{D_\g \g}^2 \frac 1 {\intrinsicDistance^{\a + 2}} \iint_{I_\g^2} \abs{\Delta D_\g \eta_i(s,t)}^2 \d t \d s + 2 \norm[L^\infty]{D_\g \eta_i}^2 \frac 1 {\intrinsicDistance^{\a + 2}} \iint_{I_\g^2} \abs{\Delta D_\g \g (s,t)}^2 \d t \d s.
    \end{align*}
    Each summand in the last line fits the pattern of \thref{lemma:mainTechnicalEstimateOHara}, so we are finished with bounding \eqref{theorem:C1OHara:eq:F23}.
    
    In dealing with \eqref{theorem:C1OHara:eq:F24}, again the important technique is creative rewriting.
    In this case, we sum up two terms which differ only by exchanging $\eta_1$ and $\eta_2$:
    \begin{align*}
        &\textstyle\frac {\a (\a + 2)} {2} \paren*{ \langle \frac {\Delta \g} {\intrinsicDistance}, \frac {\Delta \eta_1} {\intrinsicDistance}\rangle + \frac {G_1(\g; \eta_1)} {\intrinsicDistance}} \cdot \frac 1 {\intrinsicDistance^{\a + 2}} \paren{\langle \Delta \g, \Delta \eta_2\rangle - \intrinsicDistance G_1(\g;\eta_2)}\\
        +&\textstyle\frac {\a (\a + 2)} {2} \paren*{ \langle \frac {\Delta \g} {\intrinsicDistance}, \frac {\Delta \eta_2} {\intrinsicDistance}\rangle + \frac {G_1(\g; \eta_2)} {\intrinsicDistance}} \cdot \frac 1 {\intrinsicDistance^{\a + 2}} \paren{\langle \Delta \g, \Delta \eta_1\rangle - \intrinsicDistance G_1(\g;\eta_1)}\\
        =&\textstyle \a (\a + 2) \frac 1 {\intrinsicDistance^{\a + 4}} \paren[\big]{ \langle \Delta \g, \Delta \eta_1\rangle \langle \Delta \g, \Delta \eta_2\rangle - \intrinsicDistance^2 G_1(\g;\eta_1) G_2(\g;\eta_2)} = \eqref{theorem:C1OHara:eq:F24}.
    \end{align*}
    The first factor of each summand on the left-hand side is bounded above by $2 \frac {\norm[L^\infty]{\eta_i'}} {v_\g}$, because of the upcoming
   estimate \eqref{lemma:multilinearIntegral:eq:estDiffQuot} and the fact that
    $
        \textstyle
        \abs{G_1(\g;\eta_i)} = \abs[\Big]{\int_{I_\g} \langle \g'(t), \frac {\eta_i'(t)} {\abs{\g'(t)}} \rangle \d t} \\\le \frac {\norm[L^\infty]{\eta_i'}} {v_\g} \intrinsicDistance.
    $
    The second factor of each summand on 
    the left-hand side is $L^p$-bounded because of \thref{lemma:boundedDifference} and thus, also \eqref{theorem:C1OHara:eq:F24} is $L^p$-bounded.
    
    Now let us really prove that $F_1$ is the Fréchet-derivative.
    \begin{claim}
        \label{theorem:C1OHara:claim:frechet}
        $F_1(\g, \eta)$ is the Fréchet-derivative of $\eap(\g)$ in direction $\eta \in \SobSpace(\R/\Z,\R^n)$.
    \end{claim}
    From now on, fix $\g \in \SobSpaceir(\R/\Z, \R^n)$ and $\epsilon > 0$ such that for all $\eta$ with $\norm[\SobSpace]{\eta} \\< \epsilon$ one has not only that
    $\g + \eta\in
    \smash{\SobSpaceir}(\R/\Z,\R^n)$, but also  the following:
    \begin{align*}
        \label{theorem:C1OHara:eq:conditionBoundSecondVariation}
        \numberthis
        \Xi(\g + \eta) &< \textstyle 2 \Xi(\g)\\
        \label{theorem:C1OHara:eq:conditionLipschitzBound}
        \numberthis
        \norm[L^\infty]{\eta'} &< \textstyle \min\set*{\frac {\BiLip(\g)} 2, \frac {v_\g} {12}, \LL(\g)}.
    \end{align*}
    Furthermore, let $\g_t := \g + t \eta_1$ and decompose $\Sigma = U(\eta_1) \cup V(\eta_1)$ with
    \begin{align*}
        U(\eta_1) &:= \set[\text{for all $t \in {[0,1]}$, } I_{\g_t}(x,y) = I_{\g}(x,y)]{(x,y) \in \Sigma} \text{ and }\\
        V(\eta_1) &:= \set[\text{there exists } t \in {[0,1]} \text{ such that } I_{\g_t}(x,y) \ne I_{\g}(x,y)]{(x,y) \in \Sigma}.
    \end{align*}
    On $U(\eta_1)$, $\ea(\g_t)$ is differentiable with respect to $t$ because $\intrinsicDistance[\g_t]$ is differentiable as long as $I_{\g_t}$ is fixed.
    By Taylor's theorem, we have for $(x,y) \in U(\eta_1)$
    \begin{align*}
        \textstyle
        \abs{\ea(\g + \eta_1) - \ea(\g) - F_1(\g; \eta_1)}
        = \abs[\big]{\int_0^1 (1-t) F_2(\g_t; \eta_1, \eta_1) \d t},
    \end{align*}
    so taking the $L^p$-norm on $U(\eta_1)$ yields, with Jensen's inequality, Tonelli's variant of Fubini's theorem, 
    \thref{theorem:C1OHara:claim:boundSecondVariation} and \eqref{theorem:C1OHara:eq:conditionBoundSecondVariation}:
    \begin{equation*}
        \numberthis
        \label{theorem:C1OHara:claim:frechet:U}
        \begin{split}
            &\textstyle\iint_{U(\eta_1)} \abs{\ea(\g + \eta) - \ea(\g) - F_1(\g; \eta_1)}^p \d y \d y\\
            \le &\textstyle \iint_{U(\eta_1)} \paren[\big]{\int_0^1 (1-t) \abs{F_2(\g_t; \eta_1, \eta_1)} \d t}^p \d y \d x\\
            \le &\textstyle\int_0^1 (1-t)^p \iint_{U(\eta_1)} \abs{F_2(\g_t; \eta_1, \eta_1)}^p \d y \d x \d t
            \le \int_0^1 (1-t)^p \Xi(\g_t)^{p} \norm[\SobSpace]{\eta_1}^{2p} \\
            \le &\textstyle 2^{p} \Xi(\g)^{p} \int_0^1 (1-t)^p \d t  \norm[\SobSpace]{\eta_1}^{2p}
            \le 2^p \Xi(\g)^{p} \norm[\SobSpace]{\eta_1}^{2p}
        \end{split}
    \end{equation*}
    Instead of trying to show that the same holds true on the ``bad'' set $V(\eta_1)$, we will show in \thref{theorem:C1OHara:claim:integrandLipschitz} that there, $\ea(\g)(x,y)$ is locally Lipschitz continuous with respect to $\g$ and has a Lipschitz constant that is uniform in $(x,y)$.
    If that is the case,
    \[
        \textstyle
        \abs{F_1(\g; \eta_1)(x,y)}
        = \lim_{h \to 0} \abs[\big]{\frac {\ea(\g + h \eta_1)(x,y) - \ea(\g)(x,y)} {h}}
        \le L_{\ea} \norm[\SobSpace]{\eta_1}
    \]
    and consequently
    \begin{align*}
        &\textstyle\iint_{V(\eta_1)} \abs{\ea(\g + \eta_1) - \ea(\g) - F_1(\g; \eta_1)}^p \d y \d x\\
        \le &\textstyle C(p) \iint_{V(\eta_1)} \abs{\ea(\g + \eta_1) - \ea(\g)}^p + \abs{F_1(\g; \eta_1)}^p \d y \d x\\
        \le &\textstyle C(p) \iint_{V(\eta_1)}  2 L_{\ea}^p \norm[\SobSpace]{\eta_1}^p \d y \d x.
    \end{align*}
    If we can furthermore show that $\abs{V(\eta_1)} \le C(\g) \norm[\SobSpace]{\eta_1}$, see \thref{theorem:C1OHara:claim:Vsmall}, we obtain
    \[
        \textstyle
        \iint_{V(\eta_1)} \abs{\ea(\g + \eta_1) - \ea(\g) - F_1(\g; \eta_1)}^p \d y \d x
        \le 2 C(p)C(\g) L_{\ea}^p \norm[\SobSpace]{\eta_1}^{p+1}
    \]
    and can use this together with \eqref{theorem:C1OHara:claim:frechet:U} to show that
    $
        \textstyle
        \norm[L^p]{\ea(\g + \eta_1) - \ea(\g) - F_1(\g; \eta_1)} \le \widetilde C(\g) \norm[\SobSpace]{\eta_1}^{1 + \frac 1 p},
    $
    which is enough for Fréchet-differentiability.
    
    \begin{claim}
        \label{theorem:C1OHara:claim:integrandLipschitz}
        There is a constant $L_{\ea}=L_{\ea}(\a, \g, n, p)$ such that
        \[
            \textstyle
            \norm[L^\infty]{\ea(\g + \eta) - \ea(\g)} \le L_{\ea} \norm[\SobSpace]{\eta}
        \]
        for all $\norm[\SobSpace]{\eta} < \epsilon$.
        Here,
       \begin{equation}\label{L} 
            \textstyle
            L_{\ea}(\g) = C_E(\a, n ,p) \paren[\Big]{\a \paren[\Big]{\frac {\LL(\g)} {3\norm[L^\infty]{\g'}}}^{-\a}\paren*{\frac {\BiLip(\g)} 2}^{-\a -1} + \a \paren*{\frac {\LL(\g)} {6}}^{-\a -1} \frac 2 {v_\g} \LL(\g)}.
       \end{equation} 
    \end{claim}
    The key ingredient is the local Lipschitz continuity of the intrinsic distance $\intrinsicDistance$ as a mapping from $\SobSpaceir(\R/\Z,\R^n)$ to $L^\infty((\R/\Z)^2)$.
    Since $\Sigma$ has full measure, it suffices to look at $(x,y) \in \Sigma$ and differentiate between two cases:
    The first case is $I_\g = I_{\tilde \g}$ (setting $\tilde \g= \g + \eta$ with $\norm[\SobSpace]{\eta} < \epsilon$).
    Then we may simply compute
    \[
        \textstyle
        \abs{\intrinsicDistance[\tilde \g](x,y) - \intrinsicDistance(x,y)} 
        = \abs[\big]{\int_{I_\g} \abs{\tilde \g'(t)} - \abs{\g'(t)} \d t}
        \le \abs{I_{\g}} \norm[L^\infty]{\eta'}
        \le \norm[L^\infty]{\eta'}.
    \]
    The other case is when the shortest connections do not match, i.e.\ $I_\g \ne I_{\tilde \g}$.
    In this case, we need more precise control of the integrands, so let us first find bounds for them.
    We know that $\abs{\abs{\tilde \g'(t)} - \abs{\g'(t)}} \le \norm[L^\infty]{\eta'}$ and so
    \[
        \textstyle
        \abs{\tilde \g'(t)} \ge \abs{\g'(t)} - \norm[L^\infty]{\eta'} \ge \abs{\g'(t)} - \frac {\abs{\g'(t)}} {v_\g} \norm[L^\infty]{\eta'}.
    \]
    Performing 
    the same estimates for an upper bound, we arrive at the fact that
    \[
        \textstyle
        \abs{\g'(t)}(1- \frac 1 {v_\g} \norm[L^\infty]{\eta'}) \le \abs{\tilde \g'(t)} \le \abs{\g'(t)}(1 + \frac 1 {v_\g} \norm[L^\infty]{\eta'}).
    \]
    Then, we use the fact that the connection between $\tilde \g(x)$ and $\tilde \g(y)$ via $I_\g$ is longer than the one via ${I_{\tilde \g}}$ to estimate
    \[
        \label{theorem:C1OHara:eq:estimateIntrinsicDistance}
        \numberthis
        \textstyle
        \int_{I_{\tilde \g}} \abs{\g'(t)}(1-\frac 1 {v_\g}\norm[L^\infty]{\eta'}) \d t 
        \le \int_{I_{\tilde \g}} \abs{\tilde \g'(t)} \d t 
        \le \int_{I_\g} \abs{\tilde \g'(t)} \d t 
        \le \int_{I_\g} \abs{\g'(t)} (1 + \frac 1 {v_\g}\norm[L^\infty]{\eta'}) \d t 
    \]
    which implies
    \[
        \label{theorem:C1OHara:eq:estimateIntrinsicDistance2}
        \numberthis
        \textstyle
         \int_{I_{\tilde \g}} \abs{\g'(t)}
         \le \int_{I_\g} \abs{\g'(t)} \d t + \int_{I_\g \cup I_{\tilde \g}} \frac 1 {v_\g}\norm[L^\infty]{\eta'} \abs{\g'(t)} \d t
         = \int_{I_\g} \abs{\g'(t)} \d t +  \frac 1 {v_\g}\norm[L^\infty]{\eta'} \LL(\g).
    \]
    Consequently, this time using the fact that ${I_{\tilde \g}}$ parametrises the longer connection between $\g(x)$ and $\g(y)$, we have
    \begin{align*}
        \textstyle
        \abs[\big]{\int_{I_{\tilde \g}} \abs{\g'(t)} - \int_{I_{\g}} \abs{\g'(t)} \d t}
        = \int_{I_{\tilde \g}} \abs{\g'(t)} \d t - \int_{I_\g} \abs{\g'(t)} \d t
        \stackrelx{\smash{\eqref{theorem:C1OHara:eq:estimateIntrinsicDistance2}}}{\le}
        \frac 1 {v_\g} \norm[L^\infty]{\eta'} \LL(\g).
    \end{align*}
    Furthermore, we can use the definition of $\tilde \g$ 
    and the fact that $1 \le \frac {\abs{\g'(t)}} {v_\g}$ to obtain
    \[
        \textstyle
\big|        \int_{I_{\tilde \g}} \abs{\tilde \g'(t)} - \abs{\g'(t)} \d t\big|
        \le \int_{I_{\tilde \g}} \abs{\abs{\tilde \g'(t)} - \abs{\g'(t)}} \d t
        \le \int_{I_{\tilde \g}} \norm[L^\infty]{\eta'} \frac 1 {v_\g} \abs{\g'(t)} \d t
        \le  \frac 1 {v_\g} \norm[L^\infty]{\eta'} \LL(\g).
    \]
    Combining the last two estimates yields
    \[
        \label{theorem:C1OHara:eq:lipschitzIntrinsicDistance}
        \numberthis
        \begin{split}
            \abs{\intrinsicDistance[\tilde \g](x,y) - \intrinsicDistance(x,y)}
            &= \textstyle \abs[\big]{\int_{I_{\tilde \g}} \abs{\tilde \g'(t)} \d t - \int_{I_\g} \abs{\g'(t)} \d t}\\
            &\le \textstyle \abs[\big]{\int_{I_{\tilde \g}} \abs{\tilde \g'(t)} - \abs{\g'(t)} \d t} + \abs[\big]{\int_{I_{\tilde \g}} \abs{\g'(t)} \d t - \int_{I_\g} \abs{\g'(t)} \d t}\\
            &\le \textstyle \frac 2 {v_\g} \norm[L^\infty]{\eta'} \LL(\g),
        \end{split}
    \]
    so we have proven the local Lipschitz property of $\intrinsicDistance$.
    
    To apply this result to $\ea$, we first provide 
    a simple local Lipschitz estimate for $x \mapsto x^{-\a}$.
    Assume $x,x+h>0$, then
    \[
        \label{theorem:C1OHara:eq:lipschitzxa}
        \numberthis
        \begin{split}
            \textstyle\abs{x^{-\a} - (x+h)^{-\a}}
            &= \textstyle \abs[\big]{h \int_0^1 - \a (x+th)^{-\a - 1} \d t}
            \le \abs{\a} \max\set{\abs{x}^{-\a - 1}, \abs{x+h}^{-\a - 1}} \abs{h}
            .
        \end{split}    
    \]
    Rewriting $\ea$ a bit, we obtain
    \begin{align*}
        \textstyle
        \abs{\ea(\tilde \g)(x,y)-\ea(\g)(x,y)}
        &= \textstyle \abs[\Big]{\frac 1 {\abs{\Delta \tilde \g}^\a} - \frac 1 {\intrinsicDistance[\tilde \g]^\a} - \frac 1 {\abs{\Delta \g}^\a} + \frac 1 {\intrinsicDistance[\g]^\a}}\\
        \label{theorem:C1OHara:eq:rhsLipschitzea}
        \numberthis
        &\le \textstyle \abs{x-y}_{\R/\Z}^{-\a} \abs[\Big]{\paren*{\frac {\abs{x-y}_{\R/\Z}} {\abs{\Delta \tilde \g}}}^\a - \paren*{\frac {\abs{x-y}_{\R/\Z}}{\abs{\Delta \g}}}^\a} + \abs[\Big]{\frac {1}{\intrinsicDistance[\g]^\a}- \frac {1}{\intrinsicDistance[\tilde \g]^\a}}. 
    \end{align*}
    In order to use \eqref{theorem:C1OHara:eq:lipschitzxa}, we need to make sure that the $\tilde \g$-terms do not veer too far from their $\g$-counterparts.
    For the $\Delta$-terms, consider that for all $k \in \Z$, we have
    \begin{align*}
        \textstyle
        \abs[\big]{\frac {\Delta \g} {\abs{x-y}_{\R/\Z}} - \frac {\Delta \tilde \g} {\abs{x-y}_{\R/\Z}}}
        &= \textstyle \abs[\big]{\int_0^1 \tilde \g'(x+t(y + k -x)) - \g'(x+t(y + k -x)) \d t} \frac {\abs{y+k - x}} {\abs{y-x}_{\R/\Z}}\\
        &\le \textstyle \norm[L^\infty]{\eta'} \frac {\abs{y+k - x}} {\abs{y-x}_{\R/\Z}}.
    \end{align*}
    Taking the minimum over all $k$, the fraction on the right-hand side becomes $1$ and so we obtain $\norm[L^\infty]{\eta'}$ as an upper bound.
    By \eqref{theorem:C1OHara:eq:conditionLipschitzBound}, this means that $\abs*{\frac {\Delta \tilde \g} {\abs{x-y}_{\R/\Z}}} \ge \abs*{\frac {\Delta \g} {\abs{x-y}_{\R/\Z}}} - \frac 1 2 \BiLip(\g) \ge \frac 1 2 \BiLip(\g)$, so the first part of \eqref{theorem:C1OHara:eq:rhsLipschitzea} is bounded above by 
    \[
        \label{theorem:C1OHara:eq:boundExtrinsicPartEA}
        \numberthis
        \textstyle
        \abs{x-y}_{\R/\Z}^{-\a} \a \paren*{\frac {\BiLip(\g)} 2}^{-\a -1} \norm[L^\infty]{\eta'}.
    \]
    In order to apply \eqref{theorem:C1OHara:eq:lipschitzxa} to the $\intrinsicDistance$-terms of \eqref{theorem:C1OHara:eq:rhsLipschitzea}, we need a lower bound for $\intrinsicDistance[\tilde \g](x,y)$.
    To achieve this, assume for the moment that $\intrinsicDistance(x,y) \ge \frac {\LL(\g)} 3$, which we will prove  in \eqref{theorem:C1OHara:eq:lowerBoundsDistancesBadSet}.
    Then, by \eqref{theorem:C1OHara:eq:lipschitzIntrinsicDistance} and \eqref{theorem:C1OHara:eq:conditionLipschitzBound},
    \begin{align*}
        \label{theorem:C1OHara:eq:lowerBoundDTildeg}
        \textstyle
        \intrinsicDistance[\tilde \g](x,y)
        &= \textstyle \intrinsicDistance(x,y) - \paren{\intrinsicDistance(x,y) - \intrinsicDistance[\tilde \g](x,y)}
        \ge \intrinsicDistance(x,y) - \frac 2 {v_\g} \norm[L^\infty]{\eta'} \LL(\g)\\
        &\ge \textstyle \frac {\LL(\g)} {3} - \frac 2 {v_\g} \norm[L^\infty]{\eta'} \LL(\g)
       \ge \frac {\LL(\g)} {6}.
    \end{align*}
    Applying \eqref{theorem:C1OHara:eq:lipschitzxa} and \eqref{theorem:C1OHara:eq:lipschitzIntrinsicDistance}, we obtain that the second part of \eqref{theorem:C1OHara:eq:rhsLipschitzea} is bounded above by
    \[
        \label{theorem:C1OHara:eq:boundIntrinsicPartEA}
        \numberthis
        \textstyle
        \a \paren[\Big]{\frac {\LL(\g)} {6}}^{-\a -1} \frac 2 {v_\g} \norm[L^\infty]{\eta'} \LL(\g).
    \]
    
    The last thing we need for Lipschitz continuity of $\ea$ on $V(\eta_1)$ is that $x$ and $y$ cannot get too close.
    The tuple $(x,y)$ is in $V(\eta_1)$ if and only if $I_{\g}(x,y) \ne I_{\g_t}(x,y)$, so it suffices to establish that this cannot happen when $\abs{x-y}_{\R/\Z}$ is small.
    
    Assume $\intrinsicDistance(x,y) \le \frac {\LL(\g)} {3}$.
    We will show that this is impossible thus establishing
    \eqref{theorem:C1OHara:eq:lowerBoundsDistancesBadSet} below, since
    \[
        \textstyle
        \abs{x-y}_{\R/\Z} 
        = \abs[\big]{\int_x^y \frac {\abs{\g'(t)}} {\abs{\g'(t)}} \d t}
        \ge \frac 1 {\norm[L^\infty]{\g'}}\abs[\big]{\int_x^y \abs{\g'(t)} \d t}
        \ge \frac 1 {\norm[L^\infty]{\g'}} \intrinsicDistance(x,y).
    \]
         Under our assumption
    and by means of
    \eqref{theorem:C1OHara:eq:estimateIntrinsicDistance}
 (recall that $\tilde \g = \g + \eta$) and \eqref{theorem:C1OHara:eq:conditionLipschitzBound},
    \begin{align*}
        \textstyle
        \int_{I_\g} \abs{\g'(t) + \eta'(t)} \d t 
        \le \int_{I_\g} \abs{\g'(t)} \d t \paren[\big]{1 +\frac {\norm[L^\infty]{\eta'}} {v_\g}}
        \le \frac {\LL(\g)} {3} \paren[\big]{1 +\frac {\norm[L^\infty]{\eta'}} {v_\g}}
        < \frac {\LL(\g)} {2}.
    \end{align*}
    Thus, $I_\g = I_{\tilde \g}$ for all $\tilde \g \in B_\epsilon(\g)$, so $(x,y) \notin V(\eta_1)$ a contradiction. Thus,
    \[
        \label{theorem:C1OHara:eq:lowerBoundsDistancesBadSet}
        \numberthis
        \textstyle
        \intrinsicDistance(x,y) > \frac {\LL(\g)} {3} \text{ and } \abs{x-y}_{\R/\Z} > \frac {\LL(\g)} {3 \norm[L^\infty]{\g'}} \text{ for all } (x,y) \in V(\eta_1).
    \]
    Combining the 
    upper bounds \eqref{theorem:C1OHara:eq:boundExtrinsicPartEA} 
    and \eqref{theorem:C1OHara:eq:boundIntrinsicPartEA} for the 
    right-hand side of \eqref{theorem:C1OHara:eq:rhsLipschitzea}
    with \eqref{theorem:C1OHara:eq:lowerBoundsDistancesBadSet}, we obtain
    \begin{align*}
        & \textstyle \abs{x-y}_{\R/\Z}^{-\a} \a \paren*{\frac {\BiLip(\g)} 2}^{-\a -1} + \a \paren*{\frac {\LL(\g)} {6}}^{-\a -1} \frac 2 {v_\g} \LL(\g)\\
        \overset{\smash{\eqref{theorem:C1OHara:eq:lowerBoundsDistancesBadSet}}}{<} & \textstyle \a \paren*{\frac {\LL(\g)} {3\norm[L^\infty]{\g'}}}^{-\a}\paren*{\frac {\BiLip(\g)} 2}^{-\a -1} + \a \paren*{\frac {\LL(\g)} {6}}^{-\a -1} \frac 2 {v_\g} \LL(\g)
        =: \widetilde{L}_{\ea}(\g)
    \end{align*}
    as Lipschitz constant for $\g \mapsto \ea(\g)$ as a mapping from $W^{1,\infty}$ to $L^\infty$.
    Finally, using the embedding from \thref{prop:embedding} (ii) to estimate $\norm[L^\infty]{\eta'}$ by $C_E(\a, n, p)\norm[\SobSpace]{\eta}$ (with $k_1=k_2 = 1$, $s_1 = \frac {\a p - 1}{2p} -1$, $\rho_1 = 2p$ and some 
    $\mu\in (0,s_1-\frac1{\rho_1})$) yields
    $L_{\ea}(\a, \g, n, p) = C_E(\a, n ,p) \widetilde{L}_{\ea}(\g)$.
    
    To wrap up the proof of \thref{theorem:C1OHara:claim:frechet}, we need that $V(\eta_1)$ is small.    
    \begin{claim}
        \label{theorem:C1OHara:claim:Vsmall}
        $\abs{V(\eta_1)} \le \frac {6} {\BiLip(\g)} \norm[L^\infty]{\eta_1'}.$
    \end{claim}
    For $(x,y) \in V(\eta_1)$, there is $t \in [0,1]$ such that the intrinsic distance is parametrised over ${I_{\tilde \g}}$ instead of $I_\g$, so let us take a closer look at the corresponding integral.
    The fundamental theorem of calculus yields
    $
        \textstyle
        \int_{I_{\tilde \g}} \abs{\g_t'(s)} \d s
        = \int_{I_{\tilde \g}} \abs{\g'(s)} + t \int_0^1 \langle D_{\g_{\tau t}} \g_{\tau t}(s), \eta_1'(s)\rangle \d \tau \d s
    $
    and so, for the fitting $t$,
    \begin{align*}
        \textstyle
        \intrinsicDistance[\g_t](x,y)
        &= \textstyle \int_{I_{\tilde \g}} \abs{\g_t'(s)} \d s
        \ge \int_{I_{\tilde \g}} \abs{\g'(s)}  - t \abs[\big]{\int_0^1 \langle D_{\g_{\tau t}} \g_{\tau t}(s), \eta_1'(s)\rangle \d \tau} \d s \\
        &\ge \textstyle \int_{I_{\tilde \g}} \abs{\g'(s)} \d s - \norm[L^\infty]{\eta_1'}
        \ge \frac {\LL(\g)} {2} - \norm[L^\infty]{\eta_1'}.
    \end{align*}
    Note that for each $(x,y) \in V(\eta_1)$, there is exactly one $\tilde x = \tilde x (x,t)$ such that $\intrinsicDistance[\g_t](x,\tilde x) = \frac 1 2 \LL(\g_t)$ and because it does not matter which arc of $\g_t$ we travel through to get from $\g_t(x)$ to $\g_t(\tilde x)$, $\intrinsicDistance[\g_t](x,\tilde x) = \intrinsicDistance[\g_t](x, y) + \intrinsicDistance[\g_t](\tilde x, y)$.
    Thus, $ \intrinsicDistance[\g_t](x,y) = \frac 1 2 \LL(\g_t) - \intrinsicDistance[\g_t](\tilde x, y)$.
    Consequently,
    $
        \textstyle
        \norm[L^\infty]{\eta_1'} 
        \ge \frac 1 2 \LL(\g) - \intrinsicDistance[\g_t](x,y)
        = \frac 1 2 \paren{\LL(\g) - \LL(\g_t)} + \intrinsicDistance[\g_t](\tilde x, y)
    $
    and so
    \[
        \textstyle
        \intrinsicDistance[\g_t](\tilde x, y)
        \le \norm[L^\infty]{\eta_1'} + \frac 1 2 \paren{\LL(\g_t) - \LL(\g)}
        = \norm[L^\infty]{\eta_1'} + \frac 1 2 \int_0^1 \abs{\g_t'(s)} - \abs{\g'(s)} \d s
        \le \frac 3 2 \norm[L^\infty]{\eta_1'}.
    \]
    By \eqref{theorem:C1OHara:eq:conditionLipschitzBound}, $\intrinsicDistance[\g_t](\tilde x, y) \ge v_{\g_t} \abs{\tilde x - y}_{\R/\Z} \ge \frac 1 2 \BiLip(\g) \abs{\tilde x - y}_{\R/\Z}$, and therefore
    \[
        \label{theorem:C1OHara:eq:distanceTildexy}
        \numberthis
        \textstyle
        \abs{\tilde x - y}_{\R/\Z} \le \frac {3} {\BiLip(\g)}
	\norm[L^\infty]{\eta_1'}
    \]
    for $(x,y) \in V(\eta_1)$, $\tilde x=\tilde x(x)$.
    Let us set $V_1(x) := \set[(x,y) \in V(\eta_1)]{y \in [0,1]}$ and $B_r^{\R/\Z}(x) := \set[\abs{x-y}_{\R/\Z} < r]{y \in [0,1]}$.
    Note that $\abs{B_r^{\R/\Z}(x)} = \abs{B_r^{\R/\Z}(\frac 1 2)} \le \abs{B_r(\frac 1 2)} = 2r$.
    We may use this in combination with \eqref{theorem:C1OHara:eq:distanceTildexy} to estimate 
    \begin{align*}
       \textstyle
       \abs{V(\eta_1)}
       &= \textstyle \int_0^1 \abs{V_1(x)} \d x
       \le \int_0^1 \abs{B^{\R/\Z}_{\frac {3} {\BiLip(\g)}\norm[L^\infty]{\eta_1'}}(\tilde x(x))} \d x
       = \int_0^1 \frac {6} {\BiLip(\g)} \norm[L^\infty]{\eta_1'} \d x\\
       &= \textstyle \frac {6} {\BiLip(\g)} \norm[L^\infty]{\eta_1'}.
    \end{align*}

    Next, we prove that $F_1$ is continuous in $\g$.
    \begin{claim}
        \label{theorem:C1OHara:claim:C1}
        The mapping
        \[
            \textstyle
            \SobSpaceir(\R/\Z, \R^n) \to \LL\paren[\Big]{\SobSpace(\R/\Z, \R^n), L^p((\R/\Z)^2)}, \g \mapsto F_1(\g, \cdot)
        \]
        is continuous. 
    \end{claim}
    It suffices to show that $\norm[L^p((\R/\Z)^2)]{F_1(\g + \eta_1;\eta_2) - F_1(\g; \eta_2)} \le \norm[\SobSpace]{\eta_2} o(1)$ as $\eta_1 \to 0$, the Landau symbol being uniform in $\eta_2$.
    Let us once again split the domain of integration into $U(\eta_1)$ and $V(\eta_1)$.
    On the former, we can use that $t \mapsto F_1(\g_t; \eta_2)$ is differentiable, as well as Jensen's inequality, Tonelli's variant of Fubini's theorem, \thref{theorem:C1OHara:claim:boundSecondVariation} and \eqref{theorem:C1OHara:eq:conditionBoundSecondVariation}:
    \begin{align*}
        &
        \textstyle
        \iint_{U(\eta_1)} \abs{F_1(\g + \eta_1; \eta_2) - F_1(\g, \eta_2)}^p \d y \d x
        = \iint_{U(\eta_1)} \abs{\int_0^1 F_2(\g_t; \eta_2; \eta_1)\d t}^p \d y \d x\\ \textstyle
        \le & \textstyle \iint_{U(\eta_1)} \int_0^1 \abs{F_2(\g_t; \eta_2; \eta_1)}^p \d t \d y \d x
        = \int_0^1 \iint_{U(\eta_1)} \abs{F_2(\g_t; \eta_2; \eta_1)}^p  \d y \d x \d t\\ \textstyle
        \le & \textstyle \Xi^p(\g_t) 
	\norm[\SobSpace]{\eta_1}^p \norm[\SobSpace]{\eta_2}^p \d t
        \le (2 \Xi(\g))^p \norm[\SobSpace]{\eta_1}^p \norm[\SobSpace]{\eta_2}^p.
    \end{align*}
    On $V(\eta_1)$ we may use \thref{theorem:C1OHara:claim:integrandLipschitz} and \thref{theorem:C1OHara:claim:Vsmall}.
    Note that because of 
    \eqref{theorem:C1OHara:eq:conditionLipschitzBound} 
    and \eqref{L}, 
    $L_{\ea}(\g + \eta) \le C L_{\ea}(\g)$ for all $\norm[\SobSpace]{\eta} < \epsilon$ and we can thus estimate
    \begin{align*}
    & \textstyle \iint_{V(\eta_1)} \abs{F_1(\g + \eta_1; \eta_2) - F_1(\g, \eta_2)}^p \d y \d x \\
        \le & \textstyle
	C(p) \iint_{V(\eta_1)} \abs{F_1(\g + \eta_1; \eta_2)}^p + \abs{F_1(\g, \eta_2)}^p \d y \d x\\
        \le & \textstyle
	C(p) (C^p+1) L_{\ea}^p \frac 6 {\BiLip(\g)} \norm[L^\infty]{\eta_1'} \norm[\SobSpace]{\eta_2}^p.
    \end{align*}
\end{proof}

\begin{lemma}
    \label{lemma:geometricNormC1}
    Let $p>1$.
    Then, the map
    \begin{align*}
        \textstyle 
        \norm[L^p_\cdot((\R/\Z)^2)]{\cdot}^p \colon W^{1, \infty}(\R/\Z, \R^n) \times L^p((\R/\Z)^2) &\to  \R,\\
        (\g,g) &\mapsto \textstyle \iint_{(\R/\Z)^2} \abs{g(x,y)}^p \abs{\g'(x)} \abs{\g'(y)} \d y \d x
    \end{align*}
    is continuously differentiable.
\end{lemma}
\begin{proof}
    We 
    prove continuous Gateaux-differentiability, which implies Fréchet-dif\-feren\-tia\-bi\-li\-ty (see e.g.\ \cite[Proposition~4.8]{Zeidler:1993:NonlinearfunctionalanalysisanditsapplicationsFixedpointtheorems}).
  Let $\g, \tilde \g \in W^{1,\infty}(\R/\Z, \R^n)$ and $g, \tilde g \in L^p((\R/\Z)^2)$.
    Then,
    \begin{align*}
        &\textstyle  \lim\limits_{h \to 0}\frac1{h}\big( \Lpg[\g + h \tilde \g]{g + h \tilde g}^p - \Lpg{g}^p\big)\\
        = &\textstyle  \lim\limits_{h \to 0} \frac 1 {h} \iint_{(\R/\Z)^2} \abs{(g + h \tilde g)(x,y)}^p \abs{\g'(x) + h \tilde \g'(x)} \abs{\g'(y) + h \tilde \g'(y)}\\
        &\textstyle \adjustedalignment{\textstyle \lim\limits_{h \to 0} \frac 1 {h} \iint_{(\R/\Z)^2}}{-} \abs{g(x,y)}^p \abs{\g'(x)} \abs{\g'(y)} \d y \d x \\
        = &\textstyle  \lim\limits_{h \to 0}  \iint_{(\R/\Z)^2} \int_0^1 \frac 1 {h} \frac {\d} {\d t} \paren{\abs{(g + t h \tilde g)(x,y)}^p \abs{\g'(x) + t h \tilde \g'(x)} \abs{\g'(y) + t h \tilde \g'(y)}} \d t \d y \d x. 
    \end{align*}
    Calculating and then bounding the derivative in the integrand, we obtain for $\abs{h} < 1$:
    \begin{align*}
        &\textstyle\frac 1 {\abs{h}} \Bigl\lvert \abs{g + th \tilde g}^{p-2} \langle g + t h \tilde g, h \tilde g\rangle \abs{(\g' + th \tilde \g')(x)} \abs{(\g' + th \tilde \g')(y)}\\
        &\textstyle \phantom{\frac 1 {\abs{h}}}+ \abs{g + th \tilde g}^p \sum_{(a,b) \in \set{(x,y), (y,x)}} \frac {\langle \g' + t h \tilde \g', h \tilde \g'\rangle} {\abs{\g' + t h \tilde \g'}} (a) \abs{\g' + t h \tilde \g'}(b) \Bigr\rvert\\
        \le &\textstyle \abs{g + t h \tilde g}^{p-1} \abs{\tilde g} (\norm[L^\infty]{\g'} + \norm[L^\infty]{\tilde \g'})^2 + 2  \abs{g + th \tilde g}^p \norm[L^\infty]{\tilde \g'} (\norm[L^\infty]{\g'} + \norm[L^\infty]{\tilde \g'})
    \end{align*}
    Estimating $\abs{g + t h \tilde g}^{\tilde p} \le C(\tilde p) (\abs{g}^{\tilde p} + \abs{\tilde g}^{\tilde p})$ and using Young's inequality with exponents $\frac {p} {p-1}$ and $p$ on the first term, we find an integrable upper bound
    \[
        \textstyle
        C(p) \paren{\abs{g}^p + \abs{\tilde g}^p}  \paren{1+ \norm[L^\infty]{\g'} + \norm[L^\infty]{\tilde \g'}} \paren {\norm[L^\infty]{\g'} + \norm[L^\infty]{\tilde \g'}}.
    \]
    By Lebesgue's theorem, we may interchange integration and taking the limit $h \to 0$ and thus arrive at
    \begin{align*}
        & \textstyle
        \delta \Lpg[\g]{g}^p(\tilde \g, \tilde g)
        = \lim_{h \to 0} \frac1{h}\big( \Lpg[\g + h \tilde \g]{g + h \tilde g}^p - \Lpg{g}^p \big)\\
        = &\textstyle \iint_{(\R/\Z)^2} \abs{g}^{p-2} \langle g, \tilde g\rangle \abs{\g'(x)} \abs{\g'(y)}+\abs{g}^p \sum_{(a,b) \in \set{(x,y), (y,x)}} \frac {\langle \g'(a), \tilde \g'(a)\rangle} {\abs{\g'(a)}} \abs{\g'(b)} \d y \d x.
    \end{align*}
    Hölder's inequality implies that this is a linear and bounded operator with respect to $(\tilde \g, \tilde g)$ and thus, we have found a Gateaux-derivative.
    To obtain $C^1$-regularity, consider the difference $\abs{(\delta \Lpg[\g_1]{g_1}^p - \delta \Lpg[\g_2]{g_2}^p)(\tilde \g, \tilde g)}$ and look at each summand on its own.
    The difference of the first summands is
    \begin{align*}
        &\textstyle \abs[\big]{\iint_{(\R/\Z)^2} \abs{g_1}^{p-2} \langle g_1, \tilde g\rangle \abs{\g_1'(x)} \abs{\g_1'(y)} - \abs{g_2}^{p-2} \langle g_2, \tilde g\rangle \abs{\g_2'(x)} \abs{\g_2'(y)} \d y \d x}\\
        \label{lemma:geometricNormC1:eq:C1Norm:Summand1}
        \numberthis
        \le & \textstyle \abs[\big]{\iint_{(\R/\Z)^2} \langle\abs{g_1}^{p-2} g_1 - \abs{g_2}^{p-2} g_2, \tilde g \rangle \abs{\g_1'(x)} \abs{\g_1'(y)} \d y \d x}\\
        \label{lemma:geometricNormC1:eq:C1Norm:Summand2}
        \numberthis
        &\textstyle + \abs[\big]{\iint_{(\R/\Z)^2} \langle\abs{g_2}^{p-2} g_2, \tilde g \rangle (\abs{\g_2'(x)} \abs{\g_2'(y)} - \abs{\g_1'(x)} \abs{\g_1'(y)}) \d y \d x}.
    \end{align*}
    The first term may be bounded above by using Hölder's inequality, yielding
    \[
        \textstyle
        \paren[\big]{\iint_{(\R/\Z)^2} \abs{g_1 \abs{g_1}^{p-2} - g_2 \abs{g_2}^{p-2}}^{\frac {p} {p-1}} \d y \d x}^{\frac {p-1} {p}} \norm[L^p]{\tilde g} \norm[L^\infty]{\g_1} \norm[L^\infty]{\g_2}.
    \]
    As 
    \[
        \textstyle
        \abs[\big]{g_1 \abs{g_1}^{p-2} - g_2 \abs{g_2}^{p-2}}^{\frac {p} {p-1}} 
        \le C(p)(\abs{g_1 \abs{g_1}^{p-2}}^{\frac {p} {p-1}} + \abs{g_2 \abs{g_2}^{p-2}}^{\frac {p} {p-1}})
        = C(p) (\abs{g_1}^p + \abs{g_2}^p),
    \]
    we have an integrable majorant and can thus apply Lebesgue's theorem to show that \eqref{lemma:geometricNormC1:eq:C1Norm:Summand1} goes to $0$ as $g_2 \to g_1$.
    
    The convergence of \eqref{lemma:geometricNormC1:eq:C1Norm:Summand2} to $0$ follows from Hölder's inequality and the $W^{1, \infty}$-convergence of $\g_2 \to \g_1$.
\end{proof}

\begin{lemma}
    \label{lemma:multilinearIntegral}
    Let $\a > 0$, $p \ge 1$, $2 < \a p < 2p + 1$, $k \in \N$, $\varphi \in \R$, $\psi > 0$, and $\g \in \SobSpaceir(\R/\Z,\R^n)$.
    Furthermore, let $b \colon (\R^n)^k \to \R$ be a $k$-multilinear map and 
    \begin{multline*}
        \textstyle
        B^{\varphi, \psi}_\g: (\SobSpace(\R/\Z, \R^n))^k \to \R,\\
        \textstyle (\eta_1, \ldots, \eta_k) \mapsto \iint_{(\R/\Z)^2} \abs[\Big]{b(L_1 \eta_1, \ldots, L_k \eta_k) \frac {\intrinsicDistance(x,y)^{\varphi + \psi + 2 - \a}} {\abs{\Delta \g}^\varphi} \cdot \paren[\Big]{\frac 1 {\abs{\Delta \g}^{2 + \psi}} - \frac 1 {\intrinsicDistance^{2 + \psi}}}}^p \d y \d x
    \end{multline*}
    where $L_iu \in \set{\frac {\Delta} {\intrinsicDistance}, \eta \mapsto D_\g \eta(x), \eta \mapsto D_\g \eta(y)}$.
    Then, $B_\g^{\varphi, \psi}$ is well-defined and
    \[
        \textstyle
        \abs{B_\g^{\varphi, \psi}(\eta_1, \ldots, \eta_k)} \le \Xi(\g) \norm[\LL^k(\R^n,\R)]{b}^p \prod_{i=1}^{k} \norm[L^\infty]{\eta_i'}^p.
    \]
    Here, $\Xi(\g)>0$ continuously depends on $\g$ with respect to the $W^{1 + \frac {\a p - 1} {2p}, 2p}$-norm.
\end{lemma}
\begin{remark}
    This as well as \thref{lemma:boundedDifference} and \thref{lemma:mainTechnicalEstimateOHara} are also valid for the case $\a p = 2$ if one considers the intersection $\SobSpace \cap W^{1,\infty}$.
    We do not make the effort to include this case as we do not want to deal with the more complicated space in \thref{theorem:C1OHara}.
\end{remark}
\begin{proof}
    Let us begin by bounding the $b$-term:
    $
        \textstyle
        \abs{b(L_1\eta_1, \ldots, L_k \eta_k)} \! \le \! \norm[\LL^k(\R^n,\R)]{b} \prod_{i=1}^{k} \abs{L_i \eta_i}.
    $ 
    All $\abs{L_i \eta_i}$ may be bounded above by $\norm[L^\infty]{D_\g \eta_i} \le \frac 1 {v_\g} \norm[L^\infty]{\eta_i'}$, which is clear for the second and third option for $L_i$.
    For the first one, we calculate
    \[
        \label{lemma:multilinearIntegral:eq:estDiffQuot}
        \numberthis
        \textstyle
     d_\g^{-1} {\abs{\Delta \eta}} 
        = d_\g^{-1} {\abs[\Big]{\int_{I_\g} \eta'(t) \frac {\abs{\g'(t)}} {\abs{\g'(t)}} \d t}} 
        \le d_\g^{-1} {\norm[L^\infty]{D_\g \eta} \abs[\big]{\int_{I_\g} \abs{\g'(t)} \d t}} 
        = \norm[L^\infty]{D_\g \eta}.
    \]
    
    We continue by defining a function $\zeta$ which helps reduce the integrand to its most relevant content:
    $\zeta \colon (0, \infty) \to \R, r \mapsto r^{2 + \varphi} \frac {r^{2 + \psi} - 1} {r^2 - 1}$.
    With this, we may rewrite the second part of the integrand:
    \[
        \textstyle
        \frac {\intrinsicDistance(x,y)^{\varphi + \psi + 2 - \a}} {\abs{\Delta \g}^\varphi} \cdot \left( \frac 1 {\abs{\Delta \g}^{2 + \psi}} - \frac 1 {\intrinsicDistance^{2 + \psi}} \right)
        = \frac 1 2 \zeta\paren*{\frac {\intrinsicDistance}{\abs{\Delta\g}}} \frac 1 {\intrinsicDistance^{2 + \a}} (2\intrinsicDistance^2 - 2\abs{\Delta \g}^2)
    \]    
    The last difference can be written in terms of the unit tangents $\tau_\g = \frac {\g'} {\abs{\g'}}$:
    \begin{align*}
        &\textstyle
        2 \intrinsicDistance^2 - 2 \abs{\Delta\g}^2\\ = &\textstyle \iint_{I_\g^2} (\abs{\tau_\g(s)}^2 + \abs{\tau_\g(t)}^2) \abs{\g'(s)} \abs{\g'(t)} \d t \d s - 2 \iint_{I_\g^2} \langle \tau_\g(s), \tau_\g(t)\rangle \abs{\g'(s)} \abs{\g'(t)} \d t \d s \\
        = &\textstyle \iint_{I_\g^2} \abs{\tau_\g(s) - \tau_\g(t)}^2 \abs{\g'(s)} \abs{\g'(t)} \d t \d s
    \end{align*}
    Let us show that $\zeta$ is bounded.
    By L'Hôpital's rule, it is continuous in $r=1$ and for all other $r$, too.
    
    Now, let us bound the argument of $\zeta$:
    Setting $I$ as the interval in $\R/\Z$ connecting $x$ and $y$ with length at most $\frac 1 2$ (so $\abs{I} = \abs{x-y}_{\R/\Z}$), we obtain that
    \[
        \textstyle
        \intrinsicDistance(x,y) = \int_{I_\g} \abs{\g'(s)} \d s \le \int_I \abs{\g'(s)} \d s \le \abs{I} \norm[L^\infty]{\g'} = \abs{x-y}_{\R/\Z}\norm[L^\infty]{\g'}
    \]
    and thus
    \[
        \textstyle
        \frac {\BiLip(\g)}{\norm[L^\infty]{\g'}} 
        \le \frac {\BiLip(\g) \abs{x-y}_{\R/\Z}}{\intrinsicDistance}
        \le \frac {\abs{\Delta \g}}{\intrinsicDistance}
        \stackrel{\smash{\eqref{lemma:multilinearIntegral:eq:estDiffQuot}}}{\le} \frac {\norm[L^\infty]{\g'}} {v_\g}.
    \]
    Consequently, we may bound $\abs{\zeta(\frac {\abs{\Delta \g}}{\intrinsicDistance})}$ in terms of a constant $C_\zeta$ that continuously depends on $\BiLip(\g), v_\g$ and $\norm[L^\infty]{\g'}$.
    It only remains to find an upper estimate for the $L^p$-Norm of
    \[
        \textstyle
        \frac 1 {\intrinsicDistance^{2 + \a}} \iint_{I_\g^2} \abs{\tau_\g(s) - \tau_\g(t)}^2 \abs{\g'(s)} \abs{\g'(t)} \d t \d s.
    \]
    Seeing as we may bound $\abs{\g'(\cdot)}$ by $\norm[L^\infty]{\g'}$, let us concentrate on the remaining integral, 
    \[
        \textstyle
        \iint_{(\R/\Z)^2} \paren[\Big]{\frac 1 {\intrinsicDistance^{2+\a}} \iint_{I_\g^2} \abs{\tau_\g(s) - \tau_\g(t)}^2 \d t \d s }^p \d y \d x
    \]
    which we can bound above by
    \[
        \textstyle
        C(\a, n, p) v_\g^{-(\a + 8) p - 2} \norm[L^\infty]{\g'}^2 \smash{\norm[W^{\frac {\a p - 1} {2p}, 2p}]{\g'}^{4p}}
    \]
    according to \thref{lemma:mainTechnicalEstimateOHara}.
    Combining our estimates, we arrive at
    \[
        \Xi(\g) = v_\g^{-(\a + 8) p - k - 2} C_\zeta(\BiLip(\g), v_\g, \norm[L^\infty]{\g'}) \norm[L^\infty]{\g'}^{2p + 2}C(\a, n, p) \smash{\norm[W^{\frac {\a p - 1} {2p}, 2p}]{\g'}^{4p}}.
    \]
    All the $\g$-dependent parameters of $\Xi$ are continuously dependent on $\g$ with respect to the $W^{1 + \frac {\a p - 1} {2p}, 2p}$-norm.
\end{proof}

\begin{lemma}
    \label{lemma:boundedDifference}
    Let $\a > 0$, $p \ge 1$, $2 < \a p < 2p + 1$ and $\g \in \SobSpaceir$.
    Then, the map
$        B_\g: (\SobSpace(\R/\Z, \R^n))^2 \to \R$ sending
$
        (\eta_1, \eta_2)$ to  the double integral
	$$
        \textstyle
        \iint_{(\R/\Z)^2} \bigl\lvert \paren[\big]{\frac 2 {\intrinsicDistance(x,y)} \int_{I_\g(x,y)} \langle D_\g \eta_1(s), D_\g \eta_2(s)\rangle \abs{\g'(s)} \d s - 2 \big< \frac {\Delta \eta_1} {\intrinsicDistance}, \frac {\Delta \eta_2} {\intrinsicDistance} \big>(x,y)} 
        \textstyle \cdot \frac 1 {\intrinsicDistance(x,y)^\a}\bigr\rvert^p \d y \d x
    $$
        is well-defined and  satisfies
    \[
        \textstyle
        \abs{B_\g(\eta_1,\eta_2)}
        \le C(\a, n, p)\norm[L^\infty]{\g'}^{2p+2} v_\g^{-(8+\a)p - 2} \norm[W^{\frac {\a p -1} {2p},2p}]{\g'}^{2p} \norm[W^{\frac {\a p -1} {2p},2p}]{\eta_1'}^{p} \norm[W^{\frac {\a p -1} {2p},2p}]{\eta_2'}^{p}.
    \]
\end{lemma}
\begin{proof}
    Note that
    \begin{align*}
    & 
	\textstyle \frac 2 {\intrinsicDistance} \int_{I_\g} \langle D_\g \eta_1(s), D_\g \eta_2(s)\rangle \abs{\g'(s)} \d s
	\\
        = 
	& 
	\textstyle \frac 1 {\intrinsicDistance^2} \iint_{I_\g^2} \paren[\big]{\langle D_\g \eta_1(s), D_\g \eta_2(s)\rangle  + \langle D_\g \eta_1(t), D_\g \eta_2(t)\rangle} \abs{\g'(s)} \abs{\g'(t)} \d t \d s
    \end{align*}
    and 
    \begin{align*}
        & \textstyle 2 \left< \frac {\Delta \eta_1} {\intrinsicDistance}, \frac {\Delta \eta_2} {\intrinsicDistance} \right>\\
        = & \textstyle \frac 1 {\intrinsicDistance^2} \iint_{I_\g^2} \paren[\big]{\langle D_\g \eta_1(s), D_\g \eta_2(t)\rangle  + \langle D_\g \eta_1(t), D_\g \eta_2(s)\rangle} \abs{\g'(s)} \abs{\g'(t)} \d t \d s.
    \end{align*}
    With these identities, we may rewrite the integrand and then bound it via the Cauchy-Schwarz-inequality as follows:
    \begin{align*}
        & \textstyle \frac {1} {\intrinsicDistance^{(2+\a)p}} \paren[\big]{\iint_{I_\g^2} \langle \Delta D_\g \eta_1, \Delta D_\g \eta_2 \rangle(s,t) \abs{\g'(s)} \abs{\g'(t)} \d t \d s}^p\\ 
        \le & \textstyle \frac {\norm[L^\infty]{\g'}^{2p}} {\intrinsicDistance^{(2+\a)p}} \prod_{i=1}^2 \paren[\big]{\iint_{I_\g^2} \abs{\Delta D_\g \eta_i(s,t)}^2\d t \d s}^{\frac p 2}
    \end{align*}
    Integrating over $(\R/\Z)^2$ and using Cauchy-Schwarz again, we arrive at a new bound to which we can apply \thref{lemma:mainTechnicalEstimateOHara}:
    \begin{align*}
        & \textstyle \norm[L^\infty]{\g'}^{2p} \prod_{i=1}^2 \paren[\big]{\iint_{(\R/\Z)^2} \frac {1} {\intrinsicDistance^{(2+\a)p}}  \paren[\big]{\iint_{I_\g^2} \abs{\Delta D_\g \eta_i(s,t)}^2\d t \d s}^{p} \d y \d x}^{\frac 1 2}\\
        \le & \textstyle C(\a, n, p) \norm[L^\infty]{\g'}^{2p+2} v_\g^{-(8+\a)p - 2} \norm[W^{\frac {\a p -1} {2p},2p}]{\g'}^{2p} \norm[W^{\frac {\a p -1} {2p},2p}]{\eta_1'}^{p} \norm[W^{\frac {\a p -1} {2p},2p}]{\eta_2'}^{p}
    \end{align*}
\end{proof}

\begin{lemma}
    \label{lemma:mainTechnicalEstimateOHara}
    Let $\a> 0$, $\beta \in (2, \a +2]$, $p \ge 1$, $2 < \a p < 2p + 1$, $\g \!\in\! \SobSpaceir(\R/\Z,\R^n)$ and $\eta \in \SobSpace(\R/\Z,\R^n)$.
    Then,
    \[
        \textstyle
        \iint_{(\R/\Z)^2} \paren[\big]{\frac 1 {\intrinsicDistance^\beta} \iint_{I_\g^2} \abs{\Delta D_\g \eta(s,t)}^2\d t \d s}^p \! \d y \d x
        \!\le\! C v_\g^{-(\beta + 6)p - 2} \norm[L^\infty]{\g'}^{2} \norm[\SobSpaceD]{\g'}^{2p} \norm[\SobSpaceD]{\eta'}^{2p}
    \]
    for some constant $C = C(\a, \beta, n, p)>0$.
\end{lemma}
\begin{proof}
    We prove this statement via substitutions with the inverse arc length function $a_\g := \LL(\g, \cdot)^{-1}$, where $\LL(\g, \cdot): \R \to \R, s \mapsto \int_0^s \abs{\g'(t)} \d t$.
    As $\LL(\g, \cdot)$ is in the space $\SobSpace(\R/\Z, \R)$, see \cite[Lemma~B.2]{KnappmannSchumacherSteenebruggeEtAl:2021:AspeedpreservingHilbertgradientflowforgeneralizedintegralMengercurvature}, it is also in $C_\loc^1(\R, \R)$.
    Since it is strictly increasing and its derivative is equal to $\abs{\g'(x)} \ge v_\g > 0$, $a_\g \in C^1_\loc(\R, \R)$ with $a_\g'(x) = \frac 1 {\LL(\g, \cdot)'(a_\g(x))} = \frac 1 {\abs{\g'(a_\g(x))}}$ and so 
    \[
        \label{lemma:mainTechnicalEstimateOHara:eq:estimateDerivativeA}
        \numberthis
        \textstyle
        \abs{a_\g'(x)} \le \frac 1 {v_\g}.
    \]
    Using integration by substitution, our bound on $a_\g'$ \eqref{lemma:mainTechnicalEstimateOHara:eq:estimateDerivativeA}, and periodicity of the outermost integrand, we obtain that
    \begin{align*}
        &\textstyle \iint_{(\R/\Z)^2} \paren[\big]{\intrinsicDistance(x,y)^{-\beta} \iint_{I_\g(x,y)^2} \abs{D_\g \eta(s) - D_\g \eta (t)}^2 \d t \d s}^p \d y \d x\\
        \le&\textstyle v_\g^{-2} \iint_{(\R/L\Z)^2} \paren[\big]{\intrinsicDistance(a_\g(\tilde x),a_\g(\tilde y))^{-\beta} \iint_{I_\g(a_\g(\tilde x),a_\g(\tilde y))^2} \abs{D_\g \eta(s) - D_\g \eta (t)}^2 \d t \d s}^p \d \tilde y \d \tilde x \\
        =&\textstyle v_\g^{-2} \int_0^L \int_{\tilde x - \frac L 2}^{\tilde x + \frac L 2} \paren[\big]{\intrinsicDistance(a_\g(\tilde x),a_\g(\tilde y))^{-\beta} \iint_{I_\g(a_\g(\tilde x),a_\g(\tilde y))^2} \abs{D_\g \eta(s) - D_\g \eta (t)}^2 \d t \d s}^p \d \tilde y \d \tilde x.
        \label{lemma:mainTechnicalEstimateOHara:eq:firstEstimate}
        \numberthis
    \end{align*}
    Although it may not seem this way at first, this representation is in fact simpler, as our parametrisation fits both $\intrinsicDistance$ and $I_\g$:
    Let $\tilde x, \tilde y \in \R$.
    Then, there is $k \in \Z$ such that
    \begin{align*}
        \textstyle
        \intrinsicDistance(a_\g(\tilde x), a_\g(\tilde y))
        &= \textstyle \abs[\big]{\int_{a_\g(\tilde x)}^{a_\g(\tilde y) + k} \abs{\g'(s)} \d s}
        = \abs{\LL(\g, a_\g(\tilde y) + k) - \LL(\g, a_\g(\tilde x))}\\
        &= \textstyle \abs{\LL(\g, a_\g(\tilde y)) + k L - \LL(\g, a_\g(\tilde x))}
        = \abs{\tilde y + k L - \tilde x}.
    \end{align*}
    As $\intrinsicDistance(a_\g(\tilde x), a_\g(\tilde y)) \le \frac L 2$, we obtain that
$
        \intrinsicDistance(a_\g(\tilde x), a_\g(\tilde y))
        =\abs{\tilde y + k L - \tilde x}
        =\abs{\tilde y - \tilde x}_{\R/L\Z}$.
    In particular, if $\abs{\tilde x - \tilde y} < \frac L 2$, 
    we have that $I_\g(a_\g(\tilde x), a_\g(\tilde y)) =
    [a_\g(\tilde x), a_\g(\tilde y)]$, employing the short-hand notation $[a,b] := [b,a]$ if $b<a$.
    
    Thus, we may rewrite our integral from \eqref{lemma:mainTechnicalEstimateOHara:eq:firstEstimate} as
    \[
        \textstyle
        v_\g^{-2} \int_0^L \int_{\tilde x - \frac L 2}^{\tilde x + \frac L 2} \paren[\big]{\abs{\tilde x -\tilde y}^{-\beta} \iint_{[a_\g(\tilde x),a_\g(\tilde y)]^2} \abs{D_\g \eta(s) - D_\g \eta (t)}^2 \d t \d s}^p \d \tilde y \d \tilde x.
    \]
    First substituting $\tilde y$ by $y = \tilde y - \tilde x$ and then $(s,t)$ by $(a_\g(\tilde x + \theta_1 y), a_\g(\tilde x + \theta_2 y))$,  
    we obtain by virtue of \eqref{lemma:mainTechnicalEstimateOHara:eq:estimateDerivativeA}
    \begin{align*}
        &\textstyle v_\g^{-2} \int_0^L \int_{- \frac L 2}^{\frac L 2} \paren[\big]{\abs{y}^{-\beta} \iint_{[a_\g(\tilde x),a_\g(\tilde x + y)]^2} \abs{D_\g \eta(s) - D_\g \eta (t)}^2 \d t \d s}^p \d  y \d \tilde x\\
        \le&\textstyle v_\g^{-2-2p} \! \int_0^L \int_{- \frac L 2}^{\frac L 2} \paren[\big]{\abs{y}^{-\beta} \iint_{[0,1]^2} \abs{D_\g \eta(a_\g(\tilde x + \theta_1 y)) - D_\g \eta (a_\g(\tilde x + \theta_2 y))}^2 \abs{y}^2 \d \theta_2 \d \theta_1}^p\! \d \tilde y \d \tilde x.
    \end{align*}
    For the moment dropping the factor $v_\g^{-2-2p}$, employing Jensen's inequality, Tonelli's variant of Fubini's theorem, the substitution $u = \tilde x + \theta_2 y$ as well as the $L$-periodicity of the integrand with respect to $u$, we obtain the new estimate
    \begin{align*}
         &\textstyle \int_0^L \int_{- \frac L 2}^{\frac L 2} \abs{y}^{(2-\beta) p} \iint_{[0,1]^2} \abs{D_\g \eta(a_\g(\tilde x + \theta_1 y)) - D_\g \eta (a_\g(\tilde x + \theta_2 y))}^{2p} \d \theta_2 \d \theta_1 \d \tilde y \d \tilde x\\
         =&\textstyle \iint_{[0,1]^2} \int_{- \frac L 2}^{\frac L 2} \int_0^L \abs{y}^{(2-\beta) p} \abs{D_\g \eta(a_\g(u + (\theta_1 - \theta_2)y)) - D_\g \eta (a_\g(u))}^{2p} \d u \d y \d \theta_2 \d \theta_1.
    \end{align*}
    Two further substitutions, $\theta_2$ by $\vartheta = \theta_1 - \theta_2$ and $y$ by $w = \vartheta y$ yield a new upper bound.
    Note that we dropped the $\theta_1$-integral as its domain has measure $1$ and nothing depends on $\theta_1$ after enlarging the domain of integration for $\vartheta$.
    \begin{align*}
        &\textstyle \int_0^1 \int_{\theta_1 - 1}^{\theta_1} \int_{- \frac L 2}^{\frac L 2} \int_0^L \abs{y}^{(2-\beta) p} \abs{D_\g \eta(a_\g(u + \vartheta y)) - D_\g \eta (a_\g(u))}^{2p} \d u \d y \d \vartheta \d \theta_1\\
        \le &\textstyle \int_{-1}^{1} \int_{- \frac L 2}^{\frac L 2} \int_0^L \abs{y}^{(2-\beta) p} \abs{D_\g \eta(a_\g(u + \vartheta y)) - D_\g \eta (a_\g(u))}^{2p} \d u \d y \d \vartheta\\
        = &\textstyle \int_{-1}^{1} \int_{- \abs{\vartheta} \frac L 2}^{ \abs{\vartheta} \frac L 2} \int_0^L \abs*{\frac {w} {\vartheta}}^{(2-\beta) p} \abs{D_\g \eta(a_\g(u + w)) - D_\g \eta (a_\g(u))}^{2p} \vartheta^{-1} \d u  \d w \d \vartheta\\
        \le &\textstyle \int_{-1}^{1} \abs{\vartheta}^{(\beta - 2) p - 1} \d \vartheta \int_{-\frac L 2}^{\frac L 2} \int_0^L \abs{w}^{(2-\beta) p} \abs{D_\g \eta(a_\g(u + w)) - D_\g \eta (a_\g(u))}^{2p} \d u \d w\\
        = &\textstyle C(\beta, p) \seminorm[\frac {(\beta - 2) p -1} {2p}, 2p]{(D_\g \eta) \circ a_\g}^{2p}
        \le \widetilde C(\a, \beta, p) \seminorm[\SobSpaceDParams]{(D_\g \eta) \circ a_\g}^{2p}.
    \end{align*}
    For the definition of the Gagliardo seminorm $\seminorm[s, \rho]{\cdot}$, see Appendix \ref{sec:appendix}.
    In the last line, we used that $2 < \beta \le \a + 2$.
    Using \thref{lemma:compositionSobolevFunction} and \thref{lemma:boundGeometricDerivative}, we may estimate
    \begin{align*}
        \textstyle
        \seminorm[\SobSpaceDParams, \R/L\Z ]{(D_\g \eta) \circ a_\g}
        &\le \textstyle \norm[C^0]{a_\g'}^{\frac 1 {2p} + \frac {(\beta - 2)p -1} {2p}} v_{a_\g}^{- \frac 1 p} \seminorm[\SobSpaceDParams, \R/\Z]{D_\g \eta}\\
        &\le \textstyle C(\beta, n ,p) \norm[C^0]{a_\g'}^{\frac {(\beta - 2)p} {2p}} v_{a_\g}^{- \frac 1 p} v_\g^{-2} \norm[\SobSpaceD]{\g'} \norm[\SobSpaceD]{\eta'}
    \end{align*}
    Seeing as $v_{a_\g} = \norm[L^\infty]{\g'}^{-1}$ and $\norm[C^0]{a_\g'} = v_\g^{-1}$ and taking into account the factor $v_\g^{-2-2p}$ we dropped along the way, we arrive at
    \[
        \textstyle
        \overline C(\a, \beta, n, p) v_\g^{-(\beta + 6)p - 2} \norm[L^\infty]{\g'}^{2}  \norm[\SobSpaceD]{\g'}^{2p} \norm[\SobSpaceD]{\eta'}^{2p}
    \]
    as our final upper bound.
\end{proof}
\appendix

\section{Properties of low-order Sobolev-Slobodecki\v{\i} Functions}
\label{sec:appendix}
    In the following, we collect some basic statements concerning Sobolev-Slobodecki\v{\i} spaces.
    Recall that for $k \in \N$, $s \in (0,1)$, $\rho \ge 1$ and $l > 0$, the Sobolev-Slobodecki\v{\i} space $W^{k+s, \rho}(\R/l \Z, \R^n)$ is the space of all $l$-periodic $W^{k, \rho}_{\loc}(\R, \R^n)$-functions $f$ such that the \newterm{Gagliardo seminorm} of the highest order derivative $f^{(k)}$, 
    \[
        \textstyle \seminorm[s,\rho, \R/l\Z]{f^{(k)}} := \paren[\big]{\int_0^l \int_{-\frac l 2}^{\frac l 2} \frac{\abs{f^{(k)}(u+w)-f^{k}(u)}^\rho} {\abs{w}^{1 + s \rho}} \d w \d u}^{\frac 1 \rho}
    \]
    is finite.
    If there is no concern of confusion, we omit the domain and simply write
    $[\cdot]_{s,\rho}$.
\begin{lemma}[Uniform convexity]\label{lem:A.1}
Let $n\in\N$, $k\in\N_0$, $s\in (0,1)$, and $\rho,\theta\in (1,\infty)$. 
The space 
$\WL:=W^{k+s,\rho}\left( \R/\Z,\R^n \right)$ is uniformly convex
and reflexive. 
Furthermore its norm and the norm of its dual space $\WL^*$
are continuously 
Fr{\'e}chet-differentiable except at the origin. Consequentially, 
the $\theta$-duality mapping $\FJ_{\WL,\theta}$  maps $\WL$ 
homeomorphically onto  $\WL^*$.
	\end{lemma}
\begin{proof}
In order to show all the claims, we exploit the fact that 
$\WL$ is isometrically isomorphic to
a subspace of
	\begin{align*}
        \textstyle
		Z:=\paren[\big]{\bigoplus_{j=0}^k L^\rho(\R/\Z,\R^n)}\oplus 
		L^\rho\big((\R/\Z\times \R/\Z,\nu),\R^n\big)
	\end{align*}
equipped with the $\rho$-norm $|(a_0,\dots, a_k, a_{k+1})|^\rho_\rho = \sum_{j=0}^{k+1} |a_j|^{\rho}$ for the direct sums and where $\nu$ is the measure 
given by
\begin{align*}
\textstyle
\nu(A) = 
\iint_{A} \frac{\dif x \dif y}{|x-y|_{\R/\Z}},
\quad \text{ for Borel subsets } A \subset \R/\Z\times \R/\Z.
	\end{align*}
In the following, we abbreviate 
 $L^\rho 
	:= L^\rho\left( \R/\Z,\R^n \right)$, and $L^\rho(\nu)
:=L^\rho
	\big((\R/\Z\times\R/\Z,\nu),\R^n\big)$. Let 
	the fractional difference quotient
$\Delta^{s}_{\mathrm{H{\"o}l}} : W^{s,\rho} 
\to L^\rho(\nu)$ be given by
	\begin{align*}
    \textstyle
\Delta^{s}_{\mathrm{H{\"o}l}}g(x,y) :=  
\frac{g(x)-g(y)}{|x-y|_{\R/\Z}^{s}}.
	\end{align*}
The map $ 
		\Psi : \WL \to Z$ defined by
		$\Psi(f):=\big(\Psi_0(f),\dots,\Psi_{k+1}(f)\big)$ with $\Psi_j(f) = f^{(j)}$ for $j\in\{0,\dots,k\}$ and $\Psi_{k+1}(f):= \Delta^{s}_{\mathrm{H{\"o}l}}f^{(k)}$
		 is a linear isometry, since
	\begin{align*}
    \textstyle
			\norm[\WL]{f}^\rho &= \textstyle \paren[\big]{\sum_{j=0}^{k}\norm[L^\rho]{f^{(j)}}^\rho} + \int_{\R/\Z} \int_{\R/\Z} \frac{|f^{(k)}(x)-f^{(k)}(y)|^\rho}{|x-y|_{\R/\Z}^{1+s\rho}} \dif x \dif y \\
			&= \textstyle  \paren[\big]{\sum_{j=0}^{k}\norm[L^\rho]{f^{(j)}}^\rho} + \norm[L^\rho(\nu)]{\Delta^{s}_{\mathrm{H{\"o}l}} f^{(k)} }^\rho
		= \textstyle \Norm{\Psi (f)}_{Z}^{\rho}.
		\end{align*}
	
We deduce from the following two results that $Z$ is uniformly 
convex. Firstly, in \cite[Theorem 2 \& p.\ 507]{Day.1941} it is 
shown for non-negative measures $\tilde{\mu}$, numbers 
$1<\tilde{\rho}<\infty$, and Banach spaces $\tilde{\CB}$ that 
$L^{\tilde{\rho}}(\tilde{\mu};\tilde{\CB})$ is uniformly convex if $\tilde{\CB}$ is. 
The space $\tilde{\CB}:=\R^n$ endowed with the Euclidean $2$-norm satisfies this 
condition. Secondly, an immediate consequence of
\cite[Theorem 3 \& p. 504]{Day.1941} is that $l^{\tilde{\rho}}$-direct 
sums of finitely many uniformly convex Banach spaces are uniformly 
convex. It is straightforward to check that uniform convexity is 
inherited by subspaces and preserved by linear isometries. 
Therefore, $\Psi(\WL)$ and $\WL$ are uniformly 
convex.
	
Next, we show the Fr{\'e}chet-differentiability of 
$\Norm{\cdot}_{\WL}$ on $\WL\setminus\{0\}$. 
Since $\Psi$ is a bounded linear operator, it is 
Fr{\'e}chet-differentiable. It remains to investigate 
$\Norm{\cdot}_Z$. In \cite[Theorem 2.5]{Leonard.1974} it is shown 
for every non-negative measure $\tilde{\mu}$, $1<\tilde{\rho}<\infty$, 
and Banach space $\tilde{\CB}$, that the norm on
$L^{\tilde{\rho}}(\tilde{\mu},\tilde{\CB})$ 
is Fr{\'e}chet-differentiable except at $0$ if and only if the norm on 
$\tilde{\CB}$ is Fr{\'e}chet-differentiable except at $0$. Obviously, 
$\tilde{\CB}:=\R^n$ with the Euclidean $2$-norm satisfies this condition. Since $\tilde{\rho}>1$, the map $\Norm{\cdot}^{\tilde{\rho}}_{L^{\tilde{\rho}}(\tilde{\mu},\tilde{\CB})}$ is differentiable everywhere.
We infer that the map
	\begin{align*}
        \textstyle
		Z \ni (g_0,\dots,g_{k+1})\;\mapsto \;\Norm{(g_0,\dots,g_{k+1})}_{Z}^\rho = \left(\sum_{j=0}^{k}\Norm{g_j}_{L^\rho}^\rho\right)  + \Norm{g_{k+1}}_{L^\rho(\nu)}^\rho
	\end{align*}
	is differentiable.
	
The remaining claims follow from general results on Banach spaces. Since $\WL$ is uniformly convex, it is reflexive by the Milman-Pettis Theorem; see e.g., 
\cite[Theorem II.2.9]{Cioranescu:1990:GeometryofBanachSpacesDualityMappingsandNonlinearProblems}.
Another consequence of the uniform convexity of $\WL$ is that the norm of its 
dual space is Fr{\'e}chet-differentiable away from $0$, see 
\cite[Theorem II.2.13]{Cioranescu:1990:GeometryofBanachSpacesDualityMappingsandNonlinearProblems}.
By, \cite[Corollary 8.5]{Fabian.2001}, 
both norms are actually continuously Fr{\'e}chet-differentiable.
Finally, 
\cite[Corollary II.3.15]{Cioranescu:1990:GeometryofBanachSpacesDualityMappingsandNonlinearProblems} 
asserts that the duality mapping $\FJ_{\WL,\theta}$
is a homeomorphism between $\WL$ and its dual $\WL^*$.
\end{proof}

\begin{lemma}[Injective regular curves form open subsets.]
\label{lem:open}
The function space
$C^{1}_{\ir}\left( \R/\Z,\R^n \right)$ is open in 
$C^1\left( \R/\Z,\R^n \right)$ with respect to the $C^1$-topology.
In particular,  
$W^{1+s,\rho}_{\ir}\left( \R/\Z,\R^n \right)\subset 
W^{1+s,\rho}\left( \R/\Z,\R^n \right)$ is    open with respect to
its norm topology, as well
as $\CB_{\ir}\subset\CB$ for any Banach space $\CB$ continuously
embedded
in $W^{1+s,\rho}\left( \R/\Z,\R^n \right)$.
	\end{lemma}
\begin{proof}
The statement for $C^1$-curves was shown in
\cite[Lemma B.3]{KnappmannSchumacherSteenebruggeEtAl:2021:AspeedpreservingHilbertgradientflowforgeneralizedintegralMengercurvature}. 
For the Sobolev-Slobodecki\v{\i} spaces
and therefore for Banach spaces continuously embedded  in those, this
claim follows from the following standard embedding result,
\thref{prop:embedding}.
\end{proof}

The following proposition gathers well-known embedding results for     
Sobolev-Slo\-bo\-dec\-ki\v{\i} spaces. See, e.g., the appendix of 
\cite{Matt.2022} for a proof based on the Besov space theory 
presented in \cite{Triebel.2006}.
\begin{proposition}[Embedding Theorem]
\label{prop:embedding}
Let $k_1,k_2\in\N_0$, $s_1, s_2\in(0,1)$, and $1<\rho_1, \rho_2<\infty$,
and $\mu\in (0,1]$. 
		\begin{enumerate}
\item[\rm (i)] If $k_1+s_1 - (k_2 + s_2) > \max\left\{ \frac{1}{\rho_1} - \frac{1}{\rho_2}, 0 \right\}$, then the identity 
				operator
\begin{align*}
\textstyle
\mathrm{id}: W^{k_1+s_1, \rho_1}( \R/\Z,\R^d ) 
\to W^{k_2+s_2, \rho_2} ( \R/\Z,\R^d )
\end{align*}
				is compact.
\item[\rm (ii)] 
If $k_1+s_1 - (k_2+\mu) > \frac{1}{\rho_1}$, then the identity 
operator
\begin{align*}
\textstyle
\mathrm{id}: W^{k_1+s_1, \rho_1}
( \R/\Z,\R^d ) \rightarrow C^{k_2,\mu}
( \R/\Z,\R^d )
	\end{align*}
				is compact.
			\end{enumerate}
	\end{proposition}

Now we estimate the fractional Sobolev 
norm of the derivative with respect to arc length
$D_\g\eta$ that is repeatedly used
in the calculations of Section \ref{sec:5}
\begin{lemma}
\label{lemma:boundGeometricDerivative}
Let $l>0$, $\rho \in (1, \infty)$, $s \in (\frac 1 \rho, 1)$ and $\g,
 \eta \in W^{1+s,\rho}(\R/l\Z, \R^n)$ with $v_\g > 0$.
Then, the differential operator $D_\g\eta:=\eta'/|\g'|$ satisfies
\[
    \textstyle
    \norm[W^{s,\rho}]{D_\g \eta}
    \le C v_\g^{-2} \norm[W^{s,\rho}]{\g'} \norm[W^{s,\rho}]{\eta'}
\]
for some $C=C(n, s, \rho)>0$.
\end{lemma}
\begin{proof}
    First, we show that with $\g'$, also $\abs{\g'}$ is in $W^{s,\rho}$.
    For the $L^\rho$-norm, this is clear immediately, so let us look at the Gagliardo-seminorm:
    \[
        \textstyle
        \seminorm[s, \rho]{\abs{\g'}}^\rho
        = \int_{\R/l\Z} \int_{- \frac l 2}^{\frac l 2} \frac {\abs{\abs{g'(x+h)} - \abs{g'(x)}}^\rho} {\abs{h}^{1+s\rho}} \d h \d x
        \le \int_{\R/l\Z} \int_{- \frac l 2}^{\frac l 2} \frac {\abs{g'(x+h) - g'(x)}^\rho} {\abs{h}^{1+s\rho}} \d h \d x
        = \seminorm[s, \rho]{\g'}^\rho
    \]
    Now, we need only use the fact that $\norm[W^{s,\rho}]{\frac 1 {\abs{\g'}}} \le v_\g^{-2} \norm[W^{s,\rho}]{\abs{\g'}}$, see \cite[Lem\-ma~A.6]{KnappmannSchumacherSteenebruggeEtAl:2021:AspeedpreservingHilbertgradientflowforgeneralizedintegralMengercurvature}, and the product rule for fractional Sobolev functions which embed into $C^1$, see e.g.\ \cite[Proposition~A.5]{KnappmannSchumacherSteenebruggeEtAl:2021:AspeedpreservingHilbertgradientflowforgeneralizedintegralMengercurvature}, to obtain the desired bound.
\end{proof}

We need to know how fractional Sobolev functions behave under reparametrisations.
To that goal, we prove a simpler variant of \cite[Lemma~A.4]{KnappmannSchumacherSteenebruggeEtAl:2021:AspeedpreservingHilbertgradientflowforgeneralizedintegralMengercurvature}.
\begin{lemma}
    \label{lemma:compositionSobolevFunction}
    Let $s \in (0,1)$, $\rho \ge 1$ and $l,L > 0$.
    Furthermore, let $f \in W^{s, \rho}(\R/l\Z, \R^n)$ and $g \in C^1_\loc(\R, \R)$ such that $g$ is injective, $g(x+L) = g(x) + l$ for all $x \in \R$ as well as $v_g = \inf_{x \in [0,L]} \abs{g'(x)} > 0$.
    Then,
    \[
        \textstyle
        \norm[L^\rho(\R/L\Z, \R^n)]{f \circ g} \le v_g^{-\frac 1 \rho} \norm[L^\rho(\R/l\Z, \R^n)]{f}
        \text{ and }
        \seminorm[s, \rho, \R/L\Z]{f \circ g} \le \norm[C^0]{g'}^{\frac 1 \rho + s} v_g^{- \frac 2 \rho} \seminorm[s, \rho, \R/l\Z]{f}.
    \]
\end{lemma}
\begin{proof}
    The first bound is the simpler one.
    We only need to use the change of variables formula and an estimate:
    \begin{align*}
        \textstyle
        \int_{\R/L\Z} \abs{f(g(x))}^\rho \d x
        &= \textstyle \int_{\R/L\Z} \abs{f(g(x))}^\rho \frac {\abs{g'(x)}} {\abs{g'(x)}} \d x
        \le v_g^{-1} \int_{\R/L\Z} \abs{f(g(x))}^\rho \abs{g'(x)} \d x\\
        & \textstyle
        = v_g^{-1} \int_{\R/l\Z} \abs{f(y)}^\rho \d y
    \end{align*}
    For the second estimate, the idea is similar, but we first need to bound the denominator.
    In order to do this, we first prove that $\abs{g(x) - g(v)}_{\R/l\Z} \le \norm[C^0]{g'} \abs{x-v}_{\R/L\Z}$.
    Seeing as we consider the $l$-periodic distance on the left-hand side and $g(y + L) = g(y) + l$, we may replace $y$ by $y + kL$ for any integer $k$.
    Then,
    \[
        \textstyle
        \abs{g(x) - g(y)}_{\R/l\Z} 
        \le \abs{g(x) - g(y + kL)} = \abs[\big]{\int_x^{y+kL} g'(\tau) 
	\d \tau} \le \abs{y + kL - x} \norm[C^0]{g'}.
    \]
    Taking the minimum over all $k \in \Z$ yields the estimate.
    
    With this and another change of variables, we may calculate
    \begin{align*}
        \textstyle
        \iint_{[0,L]^2} \frac {\abs{f(g(x))- f(g(y))}^\rho} {\abs{x-y}_{\R/L\Z}^{1+s\rho}} \d y \d x
        &\le \textstyle \norm[C^0]{g'}^{1+s\rho} \iint_{[0,L]^2} \frac {\abs{f(g(x))- f(g(y))}^\rho} {\abs{g(x)-g(y)}_{\R/l\Z}^{1+s\rho}} \d y \d x\\
        &\le \textstyle \norm[C^0]{g'}^{1+s\rho} \iint_{[0,L]^2} \frac {\abs{f(g(x))- f(g(y))}^\rho} {\abs{g(x)-g(y)}_{\R/l\Z}^{1+s\rho}} \frac {\abs{g'(x)} \abs{g'(y)}} {v_g^2} \d y \d x\\
        &= \textstyle \norm[C^0]{g'}^{1+s\rho} v_g^{-2} \iint_{[g^{-1}(0),g^{-1}(L)]^2} \frac {\abs{f(\tilde x)- f(\tilde y)}^\rho} {\abs{\tilde x-\tilde y}_{\R/l\Z}^{1+s\rho}} \d \tilde y \d \tilde x.
    \end{align*}
    Note that our assumptions on $g$ imply its surjectivity.
    Finally, we may use that $g^{-1}(L) = g^{-1}(0) + l$ and periodicity to see that the double integral on the right-hand side is indeed the 
    desired seminorm.
\end{proof}

\section*{Acknowledgements}
A substantial part of the contents of Sections \ref{sec:2} and \ref{sec:3} 
is contained in H. Matt's  Ph.D. thesis 
\cite{Matt.2022}. Furthermore, Section \ref{sec:4} is a generalisation of \cite[Chapter~4]{Matt.2022}.
Some results of Section \ref{sec:4} and Appendix \ref{sec:appendix}, 
as well as the content of
Section \ref{sec:5} will be part of D. Steenebrügge's Ph.D. 
thesis \cite{steenebruegge_2022}.
H. Matt and D. Steenebrügge gratefully acknowledge support by the Graduiertenkolleg Energy, Entropy, and Dissipative Dynamics (EDDy) of the Deutsche Forschungsgemeinschaft (DFG) – project no. 320021702/GRK\\2326.
H. von der Mosel's work is partially funded by the Excellence Initiative of the German federal and state governments.

\bibliographystyle{alpha}
\bibliography{COMS}

\bigskip
{
    \noindent
    \small
    (H. Matt)\\
    \textsc{
        RWTH Aachen University, Institut für Mathematik\\
        Templergraben 55, 52062 Aachen, Germany.\\
    }
    \href{mailto:matt@eddy.rwth-aachen.de}{matt@eddy.rwth-aachen.de}\\
    \url{https://www.instmath.rwth-aachen.de/~matt}
    \bigskip\\    
    (D. Steenebrügge)\\
    \textsc{
        RWTH Aachen University, Institut für Mathematik\\
        Templergraben 55, 52062 Aachen, Germany.\\
    }
    \href{mailto:steenebruegge@instmath.rwth-aachen.de}{steenebruegge@instmath.rwth-aachen.de}\\
    \url{www.instmath.rwth-aachen.de/~steenebruegge}
    \bigskip\\
    (H. von der Mosel)\\
    \textsc{
        RWTH Aachen University, Institut für Mathematik\\
        Templergraben 55, 52062 Aachen, Germany.\\
    }
    \href{mailto:heiko@instmath.rwth-aachen.de}{heiko@instmath.rwth-aachen.de}\\
    \url{www.instmath.rwth-aachen.de/~heiko}
}
\end{document}